\documentclass[11pt]{amsart}

\usepackage{slashbox}
\usepackage[colorlinks = true,
            linkcolor = blue,
            urlcolor  = blue,
            citecolor = black,
            anchorcolor = blue]{hyperref}
\usepackage{color}
\usepackage[shortalphabetic,initials,nobysame]{amsrefs}
\usepackage{upgreek}
\usepackage{tikz}
\usepackage[applemac]{inputenc} % for Macs
\usetikzlibrary{arrows}
\usepackage{xcolor}
\usepackage[all]{xy}
\usepackage{graphicx}
\usepackage{wrapfig}
\usepackage{yfonts}
\usepackage{graphicx}
\usepackage{amssymb}
\usepackage{epstopdf}
\usepackage{mathrsfs}
\usepackage[mathcal]{eucal}
\usepackage{bm}

\usepackage{slashed}
\renewcommand{\eprint}[1]{\href{https://arxiv.org/abs/#1}{#1}}

\BibSpec{article}{%
    +{}  {\PrintAuthors}                {author}
    +{,} { \textit}                     {title}
     +{,} {}                            {note}
    +{.} { }                            {part}
    +{:} { \textit}                     {subtitle}
    +{,} { \PrintContributions}         {contribution}
    +{.} { \PrintPartials}              {partial}
    +{,} { }                            {journal}
    +{}  { \textbf}                     {volume}
    +{}  { \PrintDate}                {date}
    +{,} { \issuetext}                  {number}
    +{,} { \eprintpages}                {pages}
    +{,} { }                            {status}
    +{,} { \DOI}                   {doi}
    +{,} { \eprint}        {eprint}
      +{,} {\publisher}                {publisher}
    +{,} { \address}                   {address}
    +{}  { \parenthesize}               {language}
    +{}  { \PrintTranslation}           {translation}
    +{;} { \PrintReprint}               {reprint}
    +{.} {}                             {transition}
    +{}  {\SentenceSpace \PrintReviews} {review}
}

\newtheorem{Thm}{Theorem}[section]

\newtheorem{Prop}[Thm]{Proposition}

\newtheorem{Con}[Thm]{Conjecture}

\theoremstyle{definition}

\theoremstyle{remark}
\newtheorem{Rem}[Thm]{Remark}

\newtheoremstyle{named}{}{}{\itshape}{}{\bfseries}{.}{.5em}{#1 #3}
\theoremstyle{named}

\def\g{\mathfrak{g}}
\def\Frenkel:2013uda{\mathfrak{h}}

\def\h{\theta}

\def\t{\tau}

\def\bo{\textbf{o}}

\def\=>{\Longrightarrow}

\def\to{\longrightarrow}

\def\o+{\oplus}
\def\bo+{\bigoplus}

\def\<{\langle}
\def\>{\rangle}
\def\({\left(}
\def\){\right)}

\def\^{\wedge}
\def\+{\dagger}

\def\dd[#1,#2]{\frac{d#1}{d#2}}
\def\del[#1,#2]{\frac{\partial #1}{\partial #2}}
\def\over[#1]{\overline{#1}}
\def\vec[#1]{\overrightarrow{#1}}

\makeatletter
\def\mr@ignsp#1 {\ifx\:#1\@empty\else #1\expandafter\mr@ignsp\fi}%
\newcommand{\multiref}[1]{\begingroup%\let\protect\string%
\xdef\mr@no@sparg{\expandafter\mr@ignsp#1 \: }%
\def\mr@comma{}%
\@for\mr@refs:=\mr@no@sparg\do{\mr@comma\def\mr@comma{,}\ref{\mr@refs}}%
\endgroup}
\makeatother

\newcommand{\hypref}[2]{\ifx\href\asklFrenkel:2013udaas #2\else\href{#1}{#2}\fi}

\tikzset{->-/.style={decoration={
  markings,
  mark=at position .5 with {\arrow{latex}}},postaction={decorate}}}
\tikzset{
    >=latex
    }

\def\h{\hbar}

\newcommand{\nc}{\newcommand}
\nc{\on}{\operatorname}
\nc{\la}{\lambda}
\nc{\wh}{\widehat}
\nc{\ghat}{\wh\g}
\nc{\mb}{\mathbf}

\begin{document}
\title[The BV double of a Courant algebroid]{The BV double of a Courant algebroid}

\author[A.M. Zeitlin]{Anton M. Zeitlin}
\address{\hspace{-0.42cm}School of Mathematics,\newline
Georgia Institute of Technology,\newline
 686 Cherry Street, \newline
Atlanta, GA 30332-0160, USA
\newline
Email: \href{mailto:zeitlin@gatech.edu}{zeitlin@gatech.edu},\newline
 \href{https://zeitlin.math.gatech.edu}{https://zeitlin.math.gatech.edu}}

\numberwithin{equation}{section}

\maketitle

\begin{abstract}
We characterize a Courant algebroid with a Calabi-Yau structure as a homotopy BV algebra with certain properties. We explain how it fits into recent double copy constructions relating Yang-Mills homotopy algebras to the ones of Double Field Theory and Gravity.
\end{abstract}

\tableofcontents

\section{Introduction}
Scattering amplitudes for various Gauge and Gravity Field Theories have been investigated in physics and mathematics for many years. One of the great discoveries in the study of the gauge theory amplitudes is the so-called {\it color-kinematics duality} \cite{colorkin2}, which allows separating the color (i.e., Lie algebraic) and kinematic building blocks of the amplitude and show that they satisfy the same relations. It also turns out that one can replace the color blocks with kinematic ones and, this way, obtain the gravity amplitudes from gauge theory ones. This phenomenon is known as the {\it double copy} construction \cite{colorkin1}, which allows us to treat gravitational amplitudes as a certain ``doubling" of gauge theory. 

More recently, a lot of this work was reinterpreted from a homological perspective \cite{saemann1}, \cite{saemann2},\cite{saemann3},\cite{hohmdouble1},\cite{hohmdouble2},\cite{hohmdouble3}. 
In the process, various field-theoretic homotopical algebras, such as examples of  $C_{\infty}$, $L_{\infty}$, $G_{\infty}$, $BV_{\infty}$, $BV^{\square}_{\infty}$, $\dots$ emerged as necessary ingredients of the construction: these were previously studied in various incarnations in \cite{movshev}, \cite{homym}, \cite{bvym}, \cite{bvym2}, \cite{cftym},\cite{zeitcour}, \cite{reiterer}. An important question is: where do these field-theoretical examples come from? One of the goals of this paper is to show that all of these algebras are naturally related to one object, namely Courant algebroid and its deformations.

In \cite{zeitcour}, we obtained that one can associate a homotopy BV algebra with the Courant algebroid, which has an extra structure known as the Calabi-Yau (CY) structure \cite{malikov}. In this construction, one first extends the Courant algebroid to the vertex algebroid, which generates a vertex algebra. The short subcomplex of the related semi-infinite (BRST) complex, which we called the {\it complex of light modes} $(\mathcal{F}^{\bullet}, Q)$ with $\mathcal{F}^i$ nonzero only for degrees $i=0,1,2,3$ has the structure of homotopy BV algebra due to Lian and Zuckerman  \cite{lz}. The corresponding quasiclassical limit, reducing the homotopy BV algebra back to the Courant algebroid level, is quite involved and discussed in detail in \cite{zeitcour}. This construction  gives a functor from the category of Courant algebroids to the category of homotopy BV algebras. In particular, the corresponding $L_{\infty}$-subalgebra on the half-complex (in degrees 0 and 1 only) is up to contractible part coincides with the one discovered by D. Roytenberg and A. Weinstein in  \cite{weinstein}. Furthermore, it leads to cyclic structure for the underlying $C_{\infty}$-subalgebra structure under certain mild conditions, while the deformation of this algebra leads to $C_{\infty}$-algebra of Yang-Mills equations \cite{cftym}. 
While one can view the Courant algebroid as the analog of the Manin double for Lie bialgebroid as described in \cite{courant}, this homotopy BV algebra can be viewed as some ``odd double" of the $L_{\infty}$-algebra of \cite{weinstein}. We called the resulting homotopy BV algebra the {\it BV double} of the Courant algebroid. 

The central part of this paper is devoted to constructing the inverse problem to the one from \cite{zeitcour}. Namely, we consider a short chain complex $(\mathcal{F}^{\bullet}, Q)$ with $\mathcal{F}^i$ nonzero only for degrees $i=0,1,2,3$ with a structure of homotopy BV algebra of Lian-Zuckerman type on this complex. We formulate simple conditions for such an algebra so that the homotopy BV algebra operations will degenerate into Courant algebroid operations and their relations. In other words, we want to describe an inverse functor to the one described above from the isomorphic subcategory in the category of homotopy BV algebras.

In Section 2, we describe the necessary conditions on complex $(\mathcal{F}^{\bullet}, Q)$ and the structure of the differential $Q$. We also describe the odd symplectic structure on this complex existing under certain assumptions, which makes the half-complex the Lagrangian subspace. In Section 3, we discuss the Courant algebroid, its realization on the half-complex 
as a Leibniz algebra with extra features, the CY structure, and geometric examples, where CY  structure emerges as a divergence operator for the volume form for the Courant $\mathcal{O}_X$-algebroid. In Section 4, we first discuss homotopy BV algebras of Lian-Zuckerman type. Then, using as a starting point the Courant algeboid operations on half-complex, we show how they uniquely extend to the complex $(\mathcal{F}^{\bullet}, Q)$ to satisfy properties of the homotopy BV algebra. Finally, at the end of Section 4, we prove a Theorem where we drop this initial condition and show the entire list of properties one should impose on the homotopy BV algebra on the complex $(\mathcal{F}^{\bullet}, Q)$ to reproduce Courant algebroid operations on half-complex and all the necessary relations between them. We also talk about $C_{\infty}$ subalgebra, the corresponding cyclic structure produced by the pairing, and $L_{\infty}$-subalgebra on $(\mathcal{F}^{\bullet}, Q)$.

In Section 5 we remind the vertex algebra realization of this homotopy BV algebra and the corresponding quasiclassical limit from \cite{zeitcour}.

The second part of the paper is devoted to the relation of this homotopy BV algebra to the homotopy algebras of Gauge Theory and Gravity reproducing many of the filed-theoretic examples studied recently. First, we describe what we called a flat metric deformation in \cite{zeitcour} (see also \cite{azbg}) of the $C_{\infty}$-subalgebra of this homotopy BV algebra. It is indeed parametrized by the constant symmetric invertible matrix $\eta^{ij}$ producing flat metric, so that in the case of standard Courant algebroid on $T^*M$, $M=\mathbb{R}^D$, it reproduces $C_{\infty}$-algebra of Yang-Mills theory \cite{cftym} on the subcomplex $(\mathcal{F}_{YM}^{\bullet}, Q^{\eta})$, so that the induced odd symplectic 
structure reproduces the standard one induced from the BV formalism 
(see, e.g., \cite{bvym}). Here we give the proof which does not use vertex algebra properties. 

While this deformation destroys the structure of the homotopy BV algebra, it leads to something which is known as $BV^{\square}_{\infty}$ algebra, the term coined in \cite{reiterer}, which takes into account the deformation and the related extra terms emerging in relations of the homotopy BV algebra. We also return to vertex operator algebra interpretation to indicate the relation to first-order sigma model perturbations, producing deformations of the BRST complex, made in the spirit of the deformations of the open-closed homotopy algebras \cite{openclosed}, \cite{openclosedrev}.

Finally, in Section 7, we go through several constructions describing the double copy trick, which relates the homotopy BV algebra to kinematical homotopy Lie algebras associated with Gravity, including Double Field Theory versions. 
First, we describe the construction of \cite{azginf},\cite{azginfrev}, which one can view as a holomorphic double copy: we look at the homotopy Lie algebra on the subcomplex of the completed tensor product $(\mathcal{F}^{\bullet}\otimes\bar{\mathcal{F}}^{\bullet}, Q+\bar{Q})$ corresponding to holomorphic and antiholomorphic Courant algebroids. Following the constructions of \cite{lmz},\cite{zeit2}, \cite{azginf},\cite{azginfrev}, we show that it leads to the homotopy Lie algebra reproducing Einstein equations with B-field and dilaton for specific reductions and formulate a conjecture about the general case. 

The double copy we go through after is more involved. 
First, we will take another look at the Yang-Mills $C_\infty$-algebra and its relation to the conformal field theory. We consider a construction similar to Section 7.1, but on the subcomplex of the completed tensor product of two Yang-Mills $C_{\infty}$-algebras: $(\mathcal{F}_{YM}^{\bullet}\otimes\bar{\mathcal{F}}_{YM}^{\bullet}, Q^{\eta}+\bar{Q}^{\eta})$, which is an algebra on $M\times \bar{M}$, where $M,\bar{M}=\mathbb{R}^D$. In this case, the homotopy Lie algebra doesn't exist in a standard sense: for the bracket operation to satisfy homotopy Jacobi identity, we require the so-called {\it strongly constrained Double Field Theory} bilinear condition on the products of any two elements of the complex. Given that condition, the relations of $L_{\infty}$-algebra hold,  as it was shown in \cite{hohmdouble1},\cite{hohmdouble2},\cite{hohmdouble3}, at least up to the level of homotopy Jacobi identity. 

In \cite{hohmdouble3}, the authors show that the corresponding Maurer-Cartan equations reproduce the Double Field theory equations \cite{dft}, \cite{gadft}, which involves field redefinitions.  
These equations are certain generalizations of the Einstein equations with B-field and dilaton: they reduce to those on the diagonal $\delta: M\to M\times {\bar M}$.  We connect these constructions to the conformal field theory considerations of \cite{zeit3}, \cite{zeit4}, where Einstein equations emerge from Maurer-Cartan-like structures associated with the generalization of Lian-Zuckerman operations.\\

\vspace*{2mm}

\noindent{\bf Acknowledgements.}
The research of A.M.Z. was partially supported by the NSF grant DMS-2203823.

\section{Complex $(\mathcal{F}^{\bullet}, Q)$ and its properties}

\subsection{The structure of differential $Q$}
Consider the following complex $(\mathcal{F}^{\bullet}, Q)$: 
\begin{align}
0\rightarrow\mathcal{F}^0\xrightarrow{Q} \mathcal{F}^1\xrightarrow{Q}\mathcal{F}^2\xrightarrow{Q} \mathcal{F}^3\rightarrow 0
\end{align}
of vector spaces over $\mathbb{K}$, where $\mathbb{K}=\mathbb{R},\mathbb{C}$.
We assume that there exist operators ${\bf b}$, ${\bf c}$, acting on $(\mathcal{F}^{\bullet}, Q)$, satisfying the following properties:
\begin{itemize}
\item
${\bf b}$, ${\bf c}$ are nilpotent of degrees $1$ and  $-1$ correspondingly:
\begin{equation}
{\bf b}:\mathcal{F}^{i}\rightarrow\mathcal{F}^{i-1},\quad{\bf c}:\mathcal{F}^{i}\rightarrow\mathcal{F}^{i+1},\quad 
{\bf b}^2=0,\quad {\bf c}^2=0,
\end{equation} 
\item the commutation relations between ${\bf b }$, ${\bf c}$ and $Q$ are: 
\begin{equation}
[Q,{\bf b}]=0,\quad  [{\bf b},{\bf c}]=1. \\
\end{equation}

\end{itemize}
The cohomology of ${\bf b}$ and ${\bf c}$ is trivial and we have the following decomposition:
\begin{equation}
\mathcal{F}^0=V_0,\quad 
\mathcal{F}^1=V'_1\oplus V'_0 ,\quad
\mathcal{F}^2=V''_1\oplus V''_0 ,\quad 
\mathcal{F}^3=V'''_0
\end{equation}
so that the operators ${\bf b}$, ${\bf c}$ are reduced to the following isomorphisms:
\begin{align}
{\rm b}:\quad V''_1\to V'_1, \quad V'_0\to V_0, \quad V'''_0\to V''_0,\nonumber\\
{\rm c}:\quad V'_1\to V''_1, \quad V_0\to V'_0, \quad V''_0\to V'''_0,
\end{align}
in other words:
\begin{equation}
{\bf b:}\xymatrixcolsep{30pt}
\xymatrixrowsep{3pt}
\xymatrix{
& {V}'_{1} & {V}''_{1}\ar[l]_{\rm{b}}& \\
& \bigoplus & \bigoplus & \\
{V}_{0} & {V}'_{0}\ar[l]_{\rm{b}} & {V}''_{0}  & {V}'''_{0}\ar[l]_{\rm{b}}
} \quad {\bf c:}
\xymatrixcolsep{30pt}
\xymatrixrowsep{3pt}
\xymatrix{
& {V}'_{1}\ar[r]^{\rm c} & {V}''_{1}& \\
& \bigoplus & \bigoplus & \\
{V}_{0} \ar[r]^{\rm c}& {V}'_{0} & {V}''_{0}\ar[r]^{\rm c}  & {V}'''_{0}
} 
\end{equation}
Notice that the operator ${\bf b}$ preserves the new grading indicated by subscript, 
so that 
\begin{equation}
\oplus^3_{i=0}\mathcal{F}^i=\mathcal{V}_0\oplus \mathcal{V}_1,
\end{equation}
where:
\begin{align}
\mathcal{V}_0=(1\oplus {\bf c})\oplus_{i\in even}{{\rm Ker  {\bf b}}}|_{\mathcal{F}^i}=V_0\oplus V'_0\oplus V''_0\oplus V'''_0.&&\nonumber\\
\mathcal{V}_1=(1\oplus {\bf c})\oplus_{i\in odd}{{\rm Ker  {\bf b}}}|_{\mathcal{F}^i}=V'_1\oplus V''_1,&&
\end{align}
The followng proposition describes the structure of the operator $Q$.
\begin{Thm}
The differential $Q$ on the complex $(\mathcal{F}, Q)$ admits the following decomposition:
\begin{equation}
Q={\rm d}+{\rm d^*}+\tilde{Q},
\end{equation}
so that
\begin{equation}
{\rm d}:~ \mathcal{V}_0\to \mathcal{V}_1;\quad {\rm d^*}:~ \mathcal{V}_1\rightarrow \mathcal{V}_0; \quad \tilde{Q}:~ V'_0\to V''_0,
\end{equation}
where $[{\rm d}, {\bf b}]=[{\rm d}, {\bf c}]=0$, $[{\rm d^*}, {\bf b}]=[{\rm d^*}, {\bf c}]=0$ so that
${\rm d^* ~d}=0$ and  ${\rm d}~ {\rm d^*}=0$:
\begin{align}
\xymatrixcolsep{30pt}
\xymatrixrowsep{3pt}
\xymatrix{
V_{0}\ar[r]^{\rm d}& V'_{1}
\ar[r]^{{\rm d}^*} & {V}''_{0}& \\
&&  &&\\
& \bigoplus & \bigoplus & \\
&&  &&\\
& V'_0\ar[uuuur]^{\tilde{Q}}\ar[r]^{\rm d} & V''_1\ar[r]^{{\rm d}^*}  & V'''_0
}
\end{align}

\end{Thm}
\begin{proof}
Notice that if $A\in V_1'$ we have $Q{\rm b}A=0$. Therefore ${\rm b}QA=0$ and therefore $QA|_{V''_1}=0$, which is one arrow which is missing in the action of the differential.
Explicitly we have:
\begin{align}
{\rm d^*}:\quad  V'_1\to V''_0,\quad  V''_1\to V'''_0,&&\\
 {\rm d}: \quad V_0\to V'_1,\quad V'_0\to V''_1,&&\nonumber\\
{\tilde {Q}}=[Q, {\rm c}{\rm b}]:\quad V'_0\to V''_0.&&\nonumber
\end{align}
\end{proof}

\subsection{Complex $(\mathcal{F}^{\bullet}_{1/2},Q)$ and an odd pairing}

One can view complex $(\mathcal{F}^{\bullet},Q)$ as an ``odd double" of a half-complex, which we denote as $(\mathcal{F}^{\bullet}_{1/2},Q)$:
\begin{equation}
0\rightarrow \mathcal{F}^0\xrightarrow{{\rm d}} \mathcal{F}^1
\end{equation}
Assume that there is a non-degenerate symmetric pairing on $\mathcal{F}^{\bullet}$ non-trivial for the following degrees:
\begin{equation}\label{pairing}
(~\cdot~,~ \cdot ~): \mathcal{F}^i\otimes \mathcal{F}^{3-i}\to \mathbb{K},
\end{equation}
and  satisfying the following conditions:
\begin{align}
(Qa_1,a_2)+(-1)^{|a_1||a_2|}(Qa_2,a_1)=0, \nonumber\\
({\bf b}a_1,a_2)-(-1)^{|a_1||a_2|}({\bf b}a_2,a_1)=0,\nonumber\\
({\bf c}a_1,a_2)+(-1)^{|a_1||a_2|}({\bf c}a_2,a_1)=0.
\end{align}
These conditions provide a non-degenerate pairing between $V_0$, $V'''_0$, between $V'_0$  and $V''_0$ as well as $V'_1$  and $V''_1$,  
making ${\rm d}$ and ${\rm d}^*$ conjugate to each other up to a sign.
We will later see examples of such a pairing.

Consider the subcomplex $(\mathcal{F}^{\bullet}_c,Q)$of 
$(\mathcal{F}^{\bullet},Q)$ consisting of the following subspaces:
\begin{align}
\mathcal{F}_c^0=V_0, \quad \mathcal{F}_c^1=\{(A,v):~ A\in V'_1,~v\in V'_0, ~v=-\tilde{Q}^{-1}{\rm d}^*A\},\quad \mathcal{F}_c^2=V''_1,\quad \mathcal{F}_c^3=V''_0
\end{align}
\begin{Prop}
If $\tilde{Q}$ is an isomorphism  then we have the following orthogonal decomposition with respect to $(\cdot ,\cdot )$:
$$
(\mathcal{F}^{\bullet},Q)=(\mathcal{F}_c^{\bullet},Q)\oplus(\mathcal{G}^{\bullet}, Q),
$$
where acyclic complex $(\mathcal{G}^{\bullet}, Q)$ is nonvanishing in degrees 1 and 2:
\begin{equation}
\mathcal{G}^1=V'_0, \quad \mathcal{G}^2=\{(\tilde{A}, \tilde{v}):~ \tilde{A}={\rm d}v\in V''_1,~ \tilde{v}=\tilde{Q}{v}\in V''_0\}.
\end{equation}
\end{Prop}
\begin{proof}
Indeed, let $(A_2,v_2)=(A_2,-\tilde{Q}^{-1}{\rm d}^*A)$, so that $Q(A_2, v_2)=-{\rm d}\tilde{Q}^{-1}{\rm d}^*A\in V''_1$ and 
$$
(Qv_1, (A_2,v_2))=(Q(A_2,v_2), v_1)=0.
$$
Since $Qv$, where $v\in \mathcal{G}^1$, generates $\mathcal{G}^2$ this leads to the statement of the Proposition.
\end{proof}
From now on we will assume that $\tilde{Q}$ is an isomorphism in the complex $(\mathcal{F}^{\bullet}, Q)$.
%we obtain the following proposition.
%\begin{Prop}
%If ${\bf b}$ is a symmetric operator with respect to the pairing  (\ref{pairing}) we have a non-degenerate pairing between $V_0$, $V'''_0$, and also between $V'_0$  and $V''_0$ and $V'_1$  and $V''_1$. $Q$ is symmetric with respect to the pairing if and only if the operators ${\rm div}$ and ${\rm d}$ are adjoint to each other and $\tilde Q$ is symmetric.
%\end{Prop}
%In the following we always assume that map ${\tilde Q}$ is an isomorphism, since we want the following pairing
%\begin{eqnarray}
%\langle~ \cdot~,Q ~\cdot ~\rangle : \mathcal{V}_0^{i}\otimes \mathcal{V}_0^{2-i}\rightarrow \mathbb{K}.
%\end{eqnarray}
%to be nondegenerate.
%One can see that it is an ``odd" double of a complex 
%\begin{equation}
%\mathcal{F}^0\xrightarrow{d} \mathcal{F}^1
%\end{equation}

%At the same time that creates another grading:
%\begin{equation}
%\oplus^3_{i=0}\mathcal{F}^i=\mathcal{V}_0\oplus \mathcal{V}_1
%\end{equation}
%with 
%$$
%{\rm d}: \mathcal{V}_0\to \mathcal{V}_1, \quad {\rm div}: \mathcal{V}_1\to \mathcal{V}_0,
%$$
%so that 
%\begin{equation}
%\mathcal{V}_1=\oplus_{i\in odd}{{\rm Ker  {\bf b}}}|_{\mathcal{F}^i}\oplus {\rm b}^{-1}{\rm Ker  {\bf b}}|_{\mathcal{F}^i}\quad \mathcal{V}_0=\oplus_{i\in even}{{\rm Ker  {\bf b}}}|_{\mathcal{F}^i}\oplus {\rm b}^{-1}{\rm Ker  {\bf b}}|_{\mathcal{F}^i}
%\end{equation}

%where ${\rm div ~d}=0$ and  ${\rm d}~ {\rm div}=0$.
\section{Courant algebroids and Leibniz algebras on $(\mathcal{F}^{\bullet}_{1/2},Q)$}

\subsection{Courant algebroids and homotopy algebras} 
In this section let us start from a slightly smaller subcomplex of a half-complex $(\mathcal{F}^{\bullet}_{1/2},Q)$ namely 
\begin{equation}\label{cc}
V_0\xrightarrow{\rm d} V'_1,
\end{equation}
which we denote as $(\mathcal{V}^{\bullet}_c, Q)$, so that 
$\mathcal{V}^0_c=V_0=\mathcal{F}^0$, $\mathcal{V}_c^1=V'_1$ and $Q={\rm d}$.
Let us denote $V_1=V'_1[1]$, so that $V_0$ and $V_1$ are of degree $0$.

We say that there is a structure of {\it $V_0$-Courant algebroid} on $V_1$ if the following properties  are satisfied (see, e.g., \cite{roytenberg}):
\begin{itemize}
\item $V_0$ is a commutative $\mathbb{K}$-algebra,
\item $V_1$ is a $V_0$-module, 
\item There is a symmetric bilinear form $\langle\cdot, \cdot \rangle: V_1\otimes_{V_0} V_1\rightarrow V_0$,
\item  $\partial: V_0\rightarrow V_1$ is a derivation
\item There is an operation $[\cdot, \cdot]: V_1\otimes V_1\rightarrow V_1$,
\end{itemize}
and the following conditions hold:
\begin{enumerate}
\item $[A_1, u A_2]=u[A_1, A_2]+\langle A_1, \partial u \rangle A_2$,
\item $\langle A_1, \partial \langle A_2, A_3 \rangle\rangle=\langle [A_1,A_2], A_3 \rangle+\langle A_2, [A_1,A_3] \rangle$,
\item $[A_1, A_2]+[a_2, a_1]=\partial \langle A_1, A_2 \rangle$,
\item $[A_1, [A_2, A_3]]=[[A_1,A_2], A_3]+[A_2, [A_1,A_3]]$,
\item $[\partial u, A]=0$,
\item $\langle \partial u_1, \partial u_2\rangle=0$,
\end{enumerate}
where $A, A_1, A_2, A_3\in V_1$ and $u, u_1, u_2\in V_0$.\\

Let us now reformulate the axioms of Courant bracket on $(\mathcal{V}^{\bullet}_c, Q)$. 
as a Leibniz algebra, namely, suppose there exist 
an operation of degree $-1$:
\begin{equation}
[~\cdot~, ~\cdot~]_c: \mathcal{V}_c^i\otimes  \mathcal{V}_c^j\to \mathcal{V}_c^{i+j-1}, 
\end{equation}
satisfying Leibniz algebra relation:
\begin{equation}\label{Leib}
[a_1,[a_2,a_3]_c]_c=[[a_1,a_2]_c,a_3]_c+(-1)^{(|a_1|-1)(|a_2|-1)}
[a_2,[a_1,a_3]_c]_c.
\end{equation}
We assume that the operation $[~\cdot~, ~\cdot~]_c$ is homotopy-anticommutative, i.e., 
\begin{align}
[a_1,a_2]_c+(-1)^{(|a_1|-1)(|a_2|-1)}[a_2,a_1]_c=\\
(-1)^{|f_1|-1}(Qn(a_1,a_2)-n(Qa_1, a_2)-(-1)^{|a_1|}n(a_1, Qa_2)),\nonumber
\end{align}
so that 
%\begin{eqnarray}
%n: \mathcal{F}^1\otimes \mathcal{F}^1\to \mathcal{F}^0
%\end{eqnarray}
%We also assume that it is nontrivial on $\mathcal{G}$ only, in other words, gives rise to the following bilinear form:
\begin{eqnarray}
n: V'_1\otimes V'_1 \to V_0.
\end{eqnarray}
is a symmetric pairing. 

We also assume that $[A,~\cdot~ ]_c$ is nontrivial only for $A\in V'_1$ and that  ${\rm d}$ is a derivation for the bracket.
Thus we have for $u_1, u_2\in V_0$, $A,B\in V'_1$: 
\begin{align}
[{\rm d}u_1, u_2]_c=0, \quad [A, B]_c=-[B, A]_c+{\rm d}n(A,B), \quad [A,u]_c=n(A,{\rm d}u), \quad [A,{\rm d}u]_c={\rm d}[A,u]
\end{align}
These two properties, in particular imply that
\begin{equation}
[{\rm d}u, A]_c=0,
\end{equation}
since $[{\rm d}u, A]_c=-[A,{\rm d}u]_c+{\rm d}n({\rm d} u, A)_c= -[A,{\rm d}u]+{\rm d}[A,u]_c=0$.

Suppose now that there is an abelian algebra structure on $V_0$ and $V_1'$ is a module for $V_0$. 
\begin{equation}
\mu_0: V_0\times V_0 \to V_0;\quad  \mu_0: V_0\times V'_1\to V'_1
\end{equation}
so that 
\begin{equation}
{\rm d}\mu_0(u_1,u_2)=\mu_0(u_1, {\rm d}u_2)+\mu_0(u_2, {\rm d}u_1),
\end{equation}
and we assume that $[A,~\cdot~]$ is a derivation of 
$\mu_0$ and $n$ is bilinear with respect to the $V_0$-action.
Explicitly that means the following:
\begin{align}
&&\mu_0(u, n(A,B))=n(\mu_0(u,A),B)=n(A, \mu_0(u,B))\\
&&[A, \mu_0(u_1,u_2)]_c=\mu_0([A, u_1]_cu_2)+\mu_0(u_1,[A, u_2]_c)\nonumber \\
&&[A, \mu_0(u,B)]_c=\mu_0([A, u]_c, B)+\mu_0(u,[A,B]_c)\nonumber\\
&&[A, n(B,C)]_c=n([A, B]_c, C)+n(B,[A,C]_c)\nonumber
\end{align}

If we set 
\begin{align}
&&uv:=\mu_0(u, v),\quad  uA=\mu_0(u, A)[1],  \quad \partial:={\rm d}[1],\nonumber\\
&&[A[1], B[1]]:=[A, B]_c[1], \quad \langle A[1], B[1]\rangle:=n(A,B)
\end{align}
we obtain that the resulting operations satisfy axioms of Courant algebroid. Let us summarize all of it in a Proposition.

\begin{Prop}
The complex $(\mathcal{V}^{\bullet}_c, {\rm d})$ with operations 
$$[~\cdot~, ~\cdot~]_c: \mathcal{V}_c^i\otimes  \mathcal{V}_c^j\to \mathcal{V}_c^{i+j-1}$$
satisfying Leibniz algebra relation (\ref{Leib}) and $\mu_0:\mathcal{V}_c^0\otimes \mathcal{V}^i_c\to \mathcal{V}^i_c$ satisfying the following properties:
\begin{enumerate}
\item ${\rm d}$ is a derivation of degree 1 for $[~\cdot~,~\cdot~]_c$, and $\mu_0(~\cdot~, ~\cdot~)$,
%\item  ${\rm d}$ is a derivation for $\mu_0$, i.e.: 
%$
%{\rm d}\mu_0(u_1,u_2)=\mu_0(u_1, {\rm d}u_2)+\mu_0(u_2, {\rm d}u_1),
%$
\item $[~\mathcal{V}_c^0~,~\cdot~]_c=0,$
\item $\mu_0$ gives the structure of abelian algebra to $\mathcal{V}_c^0$ and gives the complex $(\mathcal{V}^{\bullet}_c,{\rm d})$ the structure of module over $\mathcal{V}_c^0$,
\item The operation $[~\cdot~,~\cdot~]_c$ is homotopically commutative with homotopy being a symmetric bilinear form $$n:\mathcal{V}^i_c\otimes_{\mathcal{V}^0_c} \mathcal{V}^j_c\to \mathcal{V}_c^{i+j-2},$$ 
\item The operation $[~a~,~ \cdot~]_c$, for $a\in \mathcal{V}_c^{\bullet}$ is a derivation for $\mu_0(~\cdot~,~\cdot~)$, and $n(~\cdot~,~\cdot~)$.
\end{enumerate}
is equivalent to the structure of Courant $\mathcal{V}^0_c$-algebroid on $\mathcal{V}^1_c[1]$.
\end{Prop}
The homotopy commutativity condition for the bracket can be removed through symmetrization and and leads to the $L_{\infty}$-algebra \cite{stashbook}. The following result is due to Roytenberg and Weinstein \cite{weinstein}. 
\begin{Thm}\label{RWth}
The antisymmetrized version of the bracket $[~\cdot~,~ \cdot~]^{sym}_c$ on the complex $(\mathcal{V}_c,{\rm d})$ given by
\begin{equation}
[a_1,a_2]^{sym}_c=\frac{1}{2}([a_1,a_2]_c-(-1)^{{(|a_1|-1)}{(|a_2|-1)}}[a_2,a_1]_c), \quad a_1,a_2\in \mathcal{V}^{\bullet}_c
\end{equation} 
extends to the structure of an $L_3$-algebra.
\end{Thm} 
We will return to the discussion of the explicit trilinear operation forming this $L_3$-algebra in Section 5.3.

Our goal is to extend the result of the Theorem \ref{RWth} and ``invert" it, namely we will describe the Courant algebroid as a homotopy BV-algebra of a certain kind (see the  Subsection 4.1) on a complex $(\mathcal{F},Q)$. For that we will need an extra structure for Courant algebroid, which as we shall see essentialy describes the properties of operator ${\rm d}^*$ with respect to $\mu_0(~\cdot~ ,~\cdot~)$ and $[~\cdot~ ,~\cdot~]_c$ and $n(~\cdot~,~\cdot~)$.  
 
\subsection{Calabi-Yau structures and Courant $\mathcal{O}_X$-algebroids}  

 We say that the Courant algebroid admits a {\it Calabi-Yau (CY) structure}  if there is an operator 
 $${\rm div}:V_1\rightarrow V_0$$ 
such that the following properties hold:
\begin{align}
{\rm div}~{\partial}=0,&&\nonumber\\
{\rm div}(u A)=u {\rm div}A+\langle \partial u, A\rangle,&&\nonumber\\
{\rm div}[A_1,A_2]=[A_1,{\rm div}A_2]-[A_2, {\rm div}A_1].&&
\end{align}
This gives rise to an operator ${\rm div}: \mathcal{V}^1_c\to \mathcal{V}^0_c$ such that
\begin{align}
{\rm div}~{\rm d}=0,&&\nonumber\\
{\rm div}\mu_0(u,A)=\mu_0(u, {\rm div}A)+n ({\rm d} u, A),&&\nonumber\\
{\rm div}[A_1,A_2]_c=[A_1,{\rm div}A_2]_c-[A_2, {\rm div}A_1]_c.&&
\end{align}
for $u\in \mathcal{V}^0_c$ and $A, A_1, A_2\in \mathcal{V}^1_c$.

Now let us discuss Courant algebroids with CY structures in the geometric context. Namely, we consider Courant $\mathcal{O}_X$-algebroids (see, e.g., \cite{bressler}). Let $\mathcal{O}_X$ be the structure sheaf on smooth variety $X$, e.g, of differentiable or holomorphic functions. If we replace in the definition of the $V_0$-Courant algebroid, the commutative algebra $V_0$ by  structure sheaf $\mathcal{O}_X$ and $V_1$ by the corresponding $O_X$-module, where the map $$a\to \langle a, ~\partial ~\cdot~ \rangle,$$ where $a\in V_1$, $u\in V_0$ is a so-called {\it anchor map}:
$$
\pi: ~V_1\rightarrow \mathcal{T}_X,
$$
where $\mathcal{T}_X$ is the tangent sheaf. The standard example is given by $V_1=\mathcal{T}_X\oplus\mathcal{T}^*_X$ and the bracket is given by the Dorfman bracket:
\begin{align*}
[v_1,v_2]_c=[v,w]_{\rm Lie}, \quad [v, \omega]_c=\mathcal{L}_v\omega ,\quad 
[\omega, v]_c=-i_vd\omega,\quad [\omega_1, \omega_2]_c=0, 
\end{align*}
where $v, v_1,v_2\in \mathcal{T}_X$ and 
$\omega, \omega_1,\omega_2 \in \mathcal{T}_X$.

Now, as expected from the name, the emergence of CY structure in this case is a consequence of an existance of a global section of $\Lambda^n\mathcal{T}^*_X$, where $n={\rm dim} X $.

If $\Omega$ is a volume form, then 
$$
\mathcal{L}_{\pi(a)}\Omega=({\rm div}_{\Omega}a) \Omega,
~{\rm or, ~informally:}~
{\rm div}_{\Omega}a=\mathcal{L}_{\pi(a)}\log \Omega.
$$ 
Thus the volume form produces the ${\rm div}_{\Omega}$ operator staisfying properties of CY structure.
In particular, for compact $X$ we have:
$$
(a_1,a_2)_{\Omega}=\int \langle a_1, a_2 \rangle \Omega, \quad 
(u_1,u_2)_{\Omega}=\int u_1u_2\Omega,
$$
where $a_1, a_2\in V_1$ and $u_1, u_2\in V_0$ and then 
 we naturally have $-\partial$ and ${\rm div}$ conjugate to each other:
$$
({\rm div}_{\Omega}a, u)_{\Omega}=-(a, \partial u) _{\Omega}.
$$
These properties lead to the odd pairing from the Subsection 2.2, 
which we discuss later, in the Subsection 4.4.
\subsection{Extending the Leibniz algebra to $(\mathcal{F}^{\bullet}_{1/2},Q)$}

We are now interested in the half-complex $(\mathcal{F}^{\bullet}_{1/2},Q)$:

\begin{equation}
\mathcal{F}_{1/2}^0\xrightarrow{Q} \mathcal{F}_{1/2}^1,
\end{equation}

where $\mathcal{F}_{1/2}^0=V_0$, $\mathcal{F}_{1/2}^1=
V'_1\oplus V'_0$. 

Notice that we have operators:
\begin{equation}
{\rm b}: \mathcal{F}^1\to \mathcal{F}^0, \quad {\rm c}: \mathcal{F}^0\to \mathcal{F}^1
\end{equation}
and now we have two gradings. Let us denote: 
\begin{equation}
\mathcal{F}_{1/2}^0\oplus \mathcal{F}_{1/2}^1=\mathcal{V}^{1/2}_0\oplus \mathcal{V}^{1/2}_1,
\end{equation}
where $\mathcal{V}^{1/2}_0=V_0\oplus V'_0$. 
Assume that there is a degree $-1$ bracket operation $\{~\cdot ~,~\cdot ~\}$ on $\mathcal{F}_{1/2}$ so that 
\begin{equation}
\{\mathcal{V}_c^{\bullet}, \mathcal{V}_c^{\bullet}\}|_{\mathcal{V}_c^{\bullet}}=
[\mathcal{V}_c^{\bullet}, \mathcal{V}_c^{\bullet}]_c
\end{equation} 
%and in particular 
%$$[\mathcal{V}^c_{\bullet},\mathcal{V}^c_{\bullet}]\subset \mathcal{V}^c_{\bullet}$$
Let us assume that $\{~\cdot ~,~\cdot~\}$ satisfies the following property:
\begin{equation}
\{\mathcal{V}^{1/2}_i, \mathcal{V}^{1/2}_j\}\subset\oplus_{k\ge 1}\mathcal{V}^{1/2}_{i+j-k},
\end{equation}
and operator ${\rm b}$ is a derivation of degree -1 for this bracket. Explicitly we obtain:
\begin{align}
&&\{V_0,V'_0\}=\{V'_0, V_0\}=0, \\
&& \{V_1,V'_0\}\subset V'_0\quad  \{V'_0, V_1\},\subset V'_0\nonumber\\
&&\{V_1,V_1\}\subset V_1\oplus V_0'.\nonumber
\end{align}

The fact that ${\rm b}$-operator is a derivation leads to 
\begin{eqnarray}
{\rm b}\{V_1, V_1\}=0,
\end{eqnarray}
and thus coincides with the Courant algebroid bracket:
$$
\{\mathcal{V}_c^{\bullet},\mathcal{V}_c^{\bullet}\}=[\mathcal{V}_c^{\bullet}, \mathcal{V}_c^{\bullet}]_c
\subset \mathcal{V}_c^{\bullet}.
$$
At the same time,
\begin{equation}
{\rm b}\{A, v\}=\{A, {\rm b}v\}=n(A, {\rm d}{\rm b}v), \quad {\rm b}\{v, A\}=\{{\rm b}v, A\}=0,
\end{equation}
where $v\in V'_0$, $A\in V'_1$. That fixes uniquely the $[\cdot ,\cdot ]$ operation on $\mathcal{F}^{\bullet}_{1/2}$. 
Moreover, assume there is an abelian algebra structure on $\mathcal{F}_{1/2}^0$ and $\mathcal{F}^1$ is a module for  $\mathcal{F}_{1/2}^0$, producing 
the operation $\mu$:
\begin{eqnarray}
\mu:\mathcal{F}_{1/2}^0\otimes \mathcal{F}_{1/2}^i\to \mathcal{F}_{1/2}^i, \quad i=0,1.
\end{eqnarray}
We also assume that 
\begin{equation}
\mu(\mathcal{V}^0_c, \mathcal{V}^i_c)|_{\mathcal{V}^i_c}=\mu_0(\mathcal{V}^0_c, \mathcal{V}^i_c), \quad i=0,1.
\end{equation}
We assume the following important properties:
\begin{align}
{\rm c}~\mu(\mathcal{F}^0_{1/2}, \cdot)=\mu(\mathcal{F}^0_{1/2}, ~{\rm c}~\cdot)\\
\{~\cdot~, ~\cdot~\}=[{\rm b},\mu](~\cdot~, ~\cdot~),
\end{align} 
which hold wherever $\mu$ is defined.
Notice that 
\begin{equation}
[{\rm b},\mu](V_0, V'_1)={\rm b}\mu(V_0, V'_1)=\{V_0,V'_1\}=0,
\end{equation}
and thus the operation has the following form:
\begin{equation}
\mu: V_0\times V_0 \to V_0;\quad  \mu: V_0\times V'_1\to V'_1;\quad  \mu: V_0\times V'_0\to V'_0
\end{equation}
The only unknown product 
\begin{equation}
\mu: V_0\times V'_0\to V'_0
\end{equation}
is fixed  $u\in V_0$, $v\in V'_0$
\begin{equation}
0=\{u,v\}={\rm b}\mu(u,v)-\mu(u, {\rm b}v),
\end{equation}
and hence $\mu(u,v)={\rm c}\mu(u,{\rm b}v)$. Thus we obtain the following Proposition.

\begin{Prop}\label{halfleib}
There is a unique Leibniz algebra structure on the complex 
$\mathcal{F}_{1/2}^{\bullet}$ and the operation $\mu$, extending the Courant bracket structure on $(\mathcal{V}_c^{\bullet}, {\rm d})$
satisfying the following conditions:
\begin{itemize}
\item The bracket operation $\{~\cdot~,~\cdot~\}: \mathcal{F}_{1/2}^i\otimes \mathcal{F}^j_{1/2}\to  \mathcal{F}^{i+j-1}_{1/2}$ satisfies:
\begin{equation}
\{\mathcal{V}^{1/2}_i, \mathcal{V}^{1/2}_j\}\subset\oplus_{k \ge 1}\mathcal{V}^{1/2}_{i+j-k}, \quad
\{\mathcal{V}_c^{\bullet}, \mathcal{V}_c^{\bullet}\}|_{\mathcal{V}_c^{\bullet}}=
[\mathcal{V}_c^{\bullet}, \mathcal{V}_c^{\bullet}]_c
\end{equation}
\item The operator ${\rm b}$ is its derivation for $\{~\cdot~,~\cdot~\}$.
\item The operation $\mu: \mathcal{F}_{1/2}^0\otimes \mathcal{F}^\bullet_{1/2}\to  \mathcal{F}^\bullet_{1/2}$ satisfies:
\begin{equation}
{\rm c}~\mu(\mathcal{F}^0_{1/2}, \cdot)=\mu(\mathcal{F}^0_{1/2}, ~{\rm c}~\cdot),\quad [{\rm b}, \mu](\mathcal{F}^0_{1/2}, ~\mathcal{F}^\bullet_{1/2})= [\mathcal{F}^0_{1/2},\mathcal{F}^\bullet_{1/2}],\quad 
\mu(\mathcal{V}^0_c, \mathcal{V}^\bullet_c)|_{\mathcal{V}^\bullet_c}=\mu_0(\mathcal{V}_c^0, \mathcal{V}^\bullet_c).
\end{equation}
\end{itemize}

\end{Prop}

However, the homotopy anticommutativity of the bracket requires extra structure as one can see while attempting produce a homotopy for the bracket on $V'_0$. That will require extension of these structures to the entire complex.
%It is homotopy commutative and the homotopy $n$ has the following form:
%\begin{itemize}
%\item
%$n:\mathcal{V}_i\otimes \mathcal{V}_j\to \mathcal{V}_{i+j-2}$
%\end{itemize}
%and coincides with the form $n$ we had in the previous section.

\section{Homotopy BV algebra on $(\mathcal{F}^{\bullet},Q)$ and the BV double}

\subsection{Homotopy BV-LZ algebra}

There are several versions of the homotopy BV algebras in the literature. In this paper we call {\it homotopy BV algebra of Lian-Zuckemran type \cite{lz}} or, simply {\it homotopy BV-LZ algebra} the following object. 

Consider the $\mathbb{Z}$-graded vector space $V$ with an operator $Q$ of degree $1$ such that $Q^2=0$ and the bilinear operation $\mu$ is homotopy $Q$-commutative and homotopy $Q$-associative:

\begin{align}\label{lzrel}
&&Q\mu(a_1,a_2)=\mu(Q a_1,a_2)+(-1)^{|a_1|}\mu(a_1,Q a_2),\\
&&\mu(a_1,a_2)-(-1)^{|a_1||a_2|}\mu(a_2,a_1)= Qm(a_1,a_2)+m(Qa_1,a_2)+(-1)^{|a_1|}m(a_1,Qa_2),\nonumber\\
&&\mu(\mu(a_1,a_2),a_3)-\mu(a_1,\mu(a_2,a_3))=\nonumber\\
&& Q\nu(a_1,a_2,a_3)+\nu(Qa_1,a_2,a_3)+(-1)^{|a_1|}\nu(a_1,Qa_2,a_3)+
(-1)^{|a_1|+|a_2|}\nu(a_1,a_2,Qa_3),\nonumber
\end{align}
where $\nu$ is a trilinear operation on $V$.

Finally assume that the there is a differential ${\bf b}$ of degree -1, such that ${\bf b}^2=0$ and define the bilinear bracket operation via the following formula:
\begin{align}\label{brack}
\{a_1,a_2\}=(-1)^{|a_1|}({\bf b}\mu(a_1,a_2)-\mu({\bf b}a_1,a_2)-(-1)^{|f_1|}\mu(a_1,{\bf b}a_2).
\end{align}
so that operation $\{a~, ~\cdot\}$ is a derivation of $\mu$ of degree $|a|-1$ and $Q$ is a derivation of $\{\cdot~, ~\cdot\}$,  
i.e.:
\begin{align}\label{deriv}
Q\{a_1,a_2\}=\{Qa_1,a_2\}+(-1)^{|a_1|-1}\{a_1, Qa_2\}\\
\{a_1,\mu(a_2,a_3)\}=\mu(\{a_1,a_2\},a_3)+(-1)^{(|a_1|-1)||a_2|}\mu(a_2,\{a_1, a_3\}).\nonumber\\
\end{align}
In the language of Penkava and Schwarz \cite{penkava} the second property means that ${\bf b}$ is second order superderivation. We will call the $\mathbb{Z}$-graded vector space with operations $\mu, \nu, \{\cdot~, ~\cdot\}$, $n$ satisfying (\ref{lzrel}), (\ref{brack}), (\ref{deriv}) {\it the homotopy BV-LZ algebra}. This in particular means the following properties \cite{lz}, \cite{penkava}.

\begin{Prop}

The operations ${\bf b}$, $\mu(\cdot~,~\cdot)$, and $\{\cdot~,~\cdot\}$ for the BV-LZ algebra satisfy the relations: 
\begin{enumerate}
\item ${\bf b}$ is a derivation of degree -1 for $\{\cdot ~, ~\cdot\}$:
$$
{\bf b}\{a_1,a_2\}=\{{\bf b}a_1,a_2\}+(-1)^{|a_1|-1}\{a_1, {\bf b}a_2\}.
$$
\item $\{\cdot, ~, ~\cdot\}$ is symmetric up to $Q$-homotopy:
$$
\{a_1,a_2\}+(-1)^{(|a_1|-1)(|a_2|-1)}\{a_2,a_1\}=
(-1)^{|a_1|-1}(Qn(a_1,a_2)-n(Qa_1,a_2)-(-1)^{|a_2|}n(a_1,Qa_2)),
$$
where $n=[{\bf{b}}, m]$.
\item $\{\cdot ~, ~\cdot\}$ satisfies the Jacobi identity:
$$
\{\{a_1,a_2\},a_3\}-\{a_1,\{a_2,a_3\}\}+(-1)^{(|a_1|-1)(|a_2|-1)}\{a_2,\{a_1,a_3\}\}=0.
$$
\item $\{\cdot, ~, ~\cdot\}$ and $\mu(\cdot~,~\cdot)$ satisfy derivation rule up to homotopy:
\begin{align}
&&\{\mu(a_1,a_2),a_3\}-\mu(a_1,\{a_2,a_3\})-(-1)^{(|a_3|-1)|a_2|}\mu(\{a_1,a_3\},a_2)=\nonumber\\
&&(-1)^{|a_1|+|a_2|-1}(Qn'(a_1,a_2,a_3)-n'(Qa_1,a_2,a_3)-\nonumber\\
&&(-1)^{|a_1|}n'(a_1,Qa_2,a_3)-(-1)^{|a_1|+|a_2|}n'(a_1,a_2,Qa_3),\nonumber
\end{align}
where
\begin{equation}
n'(a_1,a_2, a_3)=\mu(a_1, n(a_2, a_3))-n(\mu(a_1, a_2), a_3)+(-1)^{|a_2||a_3|}\mu(n(a_1,a_2), a_3).
\end{equation}
\end{enumerate}

\end{Prop}

\subsection{The BV double of a Courant algebroid} We would like to extend the bracket and deform the multiplication $\mu$ in such a way that the multipication and the bracket would satisfy the homotopy BV-LZ algebra. 
We will extend the properties from the Proposition \ref{halfleib} to satisfy the following:
The operation $\mu$ is a biliniear operation of degree $0$ on the entire complex, namely
\begin{equation}
\mu: \mathcal{F}^i\otimes \mathcal{F}^j\rightarrow \mathcal{F}^{i+j},
\end{equation}
which coincides with the one from Proposition \ref{halfleib} on $\mathcal{F}^0\otimes \mathcal{F}^1$ and satisfies the conditions:
\begin{equation}
 \mu:\mathcal{V}_i\otimes\mathcal{V}_j\rightarrow \oplus_{k\ge 0}\mathcal{V}_{i+j-k};\quad {\bf c}~\mu(a_1, a_2)=(-1)^{|a_1|}\mu(a_1, {\bf c}~a_2).
\end{equation}
This, in particular, gives $\mathcal{V}_0, \mathcal{V}_1$ the module structure for the abelian algebra $V_0=\mathcal{F}^0$, i.e. 
$$
\mu:\mathcal{F}^0\otimes \mathcal{V}_i \to \mathcal{V}_i, \quad i=0,1.
$$

We assume that $\mu$ is abelian up to homotopy, and therefore we have the operation
\begin{equation}
m: \mathcal{F}^i\otimes \mathcal{F}^j\to \mathcal{F}^{i+j-1}
\end{equation}
such that 
\begin{equation}
\mu(a_1,a_2)=(-1)^{|a_1||a_2|}\mu(a_2,a_1)+Qm(a_1,a_2)+m(Qa_1, a_2)+(-1)^{|a_1|}m(a_1, Qa_2).
\end{equation}
In addition, we need the following property:
\begin{align}
\{f_1,f_2\}=(-1)^{|f_1|}({\rm b}\mu(f_1,f_2)-\mu({\rm b}f_1,f_2)-(-1)^{|f_1|}\mu(f_1,{\rm b}f_2),
\end{align}
which leads to the folowing relation:
\begin{equation}
\mu(A,u)|_{V'_0}=-{\rm c} n(A,{\rm d} u)=-{\rm c}\{A,u\}.
\end{equation}
We require that operation $m$, which is a homotopy for 
commutativity of $\mu$ is such that 
\begin{equation}
m:\mathcal{V}_i\otimes \mathcal{V}_j\to\mathcal{V}_{i+j-2}.
\end{equation}
This restriction leads to:
\begin{equation}
\mu(A,u)=\mu(u,A)-{\rm c}n(A, {\rm d} u), \quad m(A, {\rm d} u)={\rm c}n(A,{\rm d}u).
\end{equation}
Notice that the operation $n$ is related to $m$ as follows: 
\begin{equation}
n=[{\rm b},m].
\end{equation}

%Given that we know that we have nontrivial
%\begin{align}
%&&m: V'_1\otimes V'_1\to V'_0\\
%&&m:V'_1\otimes V''_1\to V''_0\\
%&&m: V''_1\otimes V''_1\to V'''_0
%\end{align}

%Because $[{\rm b},n]=0$, we immediately obtain that 
%$n'={\rm b}^{-1}n(~\cdot ~,~  {\rm b}~\cdot~)
%$
We impose the restriction on the operation 
$m$ so that it is nontrivial only on the half-complex, i.e. it is nonzero on the following subspaces:
\begin{equation}
m:V'_1\otimes V'_1\to V'_0.
\end{equation}

Then the following theorem holds.

\begin{Thm}\label{extfull}
There is a unique homotopy BV-LZ algebra structure on the complex 
$(\mathcal{F}^{\bullet},Q)$ generated by operators $\mu(~\cdot~,~\cdot~)$, $\{~\cdot~,~\cdot~\}$, ${\bf b}$, extending the Courant algebroid structure on $(\mathcal{V}_c^{\bullet}, {\rm d})$,
satisfying the following conditions:
\begin{enumerate}
\item The bracket operation $\{~\cdot~,~\cdot~\}: \mathcal{F}^i\otimes \mathcal{F}^j\to  \mathcal{F}^{i+j-1}$ satisfies:
\begin{equation}
\{\mathcal{V}_i, \mathcal{V}_j\}\subset\oplus_{k \ge 1}\mathcal{V}_{i+j-k}, \quad
\{\mathcal{V}_c^{\bullet}, \mathcal{V}_c^{\bullet}\}|_{\mathcal{V}_c^{\bullet}}=
[\mathcal{V}_c^{\bullet}, \mathcal{V}_c^{\bullet}]_c
\end{equation}
\item The operation $\mu: \mathcal{F}^i\otimes \mathcal{F}^j\to  \mathcal{F}^{i+j}$ satisfies:
\begin{equation}
\mu(\mathcal{V}_i, \mathcal{V}_j)\subset\oplus_{k\ge 0}\mathcal{V}_{i+j-k}, \quad 
\quad 
\mu(\mathcal{V}^0_c, \mathcal{V}^\bullet_c)|_{\mathcal{V}^\bullet_c}=\mu_0(\mathcal{V}_c^0, \mathcal{V}^\bullet_c), 
\end{equation}
while the restriction $\mu(\mathcal{V}_1,\mathcal{V}_1) |_{\mathcal{V}_0}$
is a symmetric bilinear form.

\item The operator {\bf c} satisfies the following property:
\begin{equation}
{\bf c}~\mu(a_1, a_2)=(-1)^{|a_1|}\mu(a_1,{\bf c}~ a_2)
\end{equation}
where $a_1, a_2\in (\mathcal{F}^\bullet, Q)$.
\item The homotopy of the product $m:  \mathcal{F}^i\otimes\mathcal{F}^j\subset\mathcal{F}^{i+j-1}$ is non-vanishing only on a half-complex:
\begin{equation}
m:\mathcal{V}^{1/2}_i\otimes \mathcal{V}^{1/2}_j\to\mathcal{V}^{1/2}_{i+j-2}
\end{equation}
\item The operator ${\rm d^*}:~V'_1\rightarrow V''_0$ gives a Calabi-Yau structure 
for a Courant algebroid on $\mathcal{V}^{\bullet}$, so that ${\rm b}\tilde{Q}^{-1}{\rm d^*}|_{V'_1}=\frac{1}{2}{\rm div}$.
\end{enumerate}

\end{Thm}

\begin{proof}
We have to show that the bilinear operations are uniquely defined and satisfy the given properties while its relation with operator ${\rm div}$ give a Calabi-Yau structure.

We notice that for anny $f_1, f_2\in (\mathcal{F}^{\bullet}, Q)$:
\begin{equation}\label{cbrprop}
{\bf c}~\{a_1, a_2\}=(-1)^{|a_1|-1}\{a_1,{\bf c}~ a_2\}
\end{equation}
Indeed,
\begin{align}
{\bf c}\{a_1,a_2\}=&&\\
(-1)^{|a_1|}({\bf c}{\bf b}\mu(a_1,a_2)-{\bf c}\mu({\bf b}a_1,a_2)-(-1)^{|f_1|}{\bf c}\mu(a_1,{\bf b}a_2))=&&\nonumber\\
(-1)^{|f_1|}(1-{\bf b}{\bf c})\mu(a_1,a_2)-{\bf c}\mu({\bf b}a_1,a_2)-\mu(a_1,(1-{\bf b}{\bf c})a_2))=&&\nonumber\\
(-1)^{|a_1|}(-(-1)^{|a_1|}{\bf b}\mu(a_1,{\bf c}a_2)+(-1)^{|a_1|}\mu({\bf b}a_1,{\bf c}a_2)+\mu(a_1,{\bf b}{\bf c}a_2))=&&\nonumber\\
(-1)^{|a_1|-1}\{a_1,{\bf c}~ a_2\}.&&\nonumber
\end{align}

That in particular implies that 
the operation $\{V'_1, \cdot\}$ preserves both degrees, namely:
\begin{equation}
\{V'_1,\mathcal{F}^i\}\subset \mathcal{F}^i, \quad \{V'_1,\mathcal{V}_i\}\subset \mathcal{V}_i.
\end{equation}
To show that it is enough to prove that $\{V'_1,V''_1\}\subset V''_1$ which follows from
\begin{equation}
\{V'_1, {\rm  c}V'_1\}\subset {\rm c}V'_1
\end{equation} 
since ${\rm c}$ is isomorphism.

%Notice that because ${\bf b}$ is a derivation of the bracket operation, namely
%\begin{equation}
%{\rm b}[V_1,~\cdot~]=[V_1,~ {\rm b}~\cdot~]
%\end{equation}

Following the property of ${\bf c}$-operator 
we have 
\begin{equation}
\mu(v, {\bf c} u)=-{\bf c}\mu(v, u)=0
\end{equation}
for $u\in V_0, v\in V'_0$. Therefore, since  $\mu(u,v)=\mu(v,u)\in V'_0$,
\begin{eqnarray}
 \mu(v_1,v_2)=0 
\end{eqnarray}
for any $v_1, {v}_2\in V'_0$. 
%Similarly following $[v_1,{v}_2]=0$ we have 
%\begin{eqnarray}
%\mu({\rm b}v_1, {v}_2)=\mu({\rm b}{v}_2,{v}_1)
%\end{eqnarray}

We also obtain that if $v\in V'_0, A\in V'_1$ and the fact that 
${\rm c}\mu(A,u)=-\mu(A,{\rm c}u)$,
\begin{equation}
\mu(A, v)=-\mu(v, A)=-{\rm c}\mu_0(A, {\rm b} v)\in V_1''
\end{equation}
given the structure of the homotopy commutativity of the product.

Now an interesting product is between two elements $V'_1$. That leads to 
\begin{equation}
\mu(A_1,A_2)=[A_1,A_2]_1+\frac{1}{2}\langle A_1,A_2 \rangle_1,
\end{equation}
where the first and the second term belong respectively to $V''_1$, $V''_0$. 
Note, that $\langle A_1,A_2 \rangle_1$ is symmetric pairing since $\mu(\mathcal{V}_1,\mathcal{V}_1) |_{\mathcal{V}_0}$ is.
The relation between the bracket operation and the product shows that 
\begin{equation}
[A_1,A_2]_c=-{\rm b}\mu(A_1,A_2)=-{\rm b}[A_1,A_2]_1
\end{equation}
Similarly, given that  
\begin{equation}
\mu(A_1,A_2)+\mu(A_2,A_1)=(Qm+mQ)(A_1,A_2)
\end{equation}
we obtain 
\begin{eqnarray}
{\tilde{Q}}~m(A_1,A_2)=\langle A_1,A_2 \rangle_1.
\end{eqnarray}
If $\tilde{A}\in V''_0$, we have:
\begin{equation}
\mu(u, \tilde{A})=c\mu_0(u, {\rm b}\tilde{A}) 
\end{equation}
for $u\in V_0$. Also, the following product vanishes, since $V'''_1=0$:
\begin{eqnarray}
\mu(v, \tilde{A})=-{\rm c}\mu(v, {\rm b}\tilde{A})=
\mu(\tilde{A}, v)=0.
\end{eqnarray}
Furthermore, if ${v} \in V'_0$, we have:
\begin{equation}
Q\mu(u,v)=\mu(Qu,v)+\mu(u,Qv)
\end{equation}
and thus $\tilde{Q}\mu(u,v)=\mu(u,\tilde{Q}v)=\mu(\tilde{Q}v, u)$ and hence we have:
\begin{equation}
\mu(u, \tilde{v})=\mu(\tilde{v},u)=\tilde{Q}{\rm c}\mu_0(u, {\rm b}\tilde{Q}^{-1}\tilde{v})
\end{equation}
for $\tilde{v} \in V''_0$ and 
\begin{equation}
\mu(u, \tilde{u})=\mu(\tilde{u},u)={\rm c}\mu(u, {\rm b}\tilde{u})=
{\rm c}\tilde{Q}{\rm c}\mu_0(u,{\rm b}\tilde{Q}^{-1}{\rm b}\tilde{u}).
\end{equation}
At the same time:
\begin{eqnarray}
\mu(\tilde{v},{A})=\mu({A}, \tilde{v})=
-{\rm c}\{A, \tilde{v}\}\in V'''_0.
\end{eqnarray}
The last equality follows from the relation between the bracket and the product.
One can relate that indeed with the bracket structure:
\begin{eqnarray}
\{{A}, \tilde{u}\}=-\mu(A, {\rm b}\tilde{u})={\rm c}\{A, {\rm b}\tilde{u}\}.
\end{eqnarray}
To obtain the relation between these brackets and the original anchor maps of Courant algebroid, relate it to $\{A,v\}$, where $v\in V'_0$.
Indeed,
\begin{equation}
{\tilde{Q}}\{A,v\}=Q\{A,v\}|_{V'_0}=\{A,Qv\}|_{V''_0}=\{A, \tilde{Q}v\}.
\end{equation}
The last equality follows from the fact that $\{V'_1, V''_1\}\subset V''_1$. 

Since ${\tilde Q}$ is an isomorphism we can characterize
\begin{align}
&&\{A, \tilde{v}\}={\tilde Q}\{A, {\tilde Q}^{-1}\tilde{v}\}=
{\tilde Q}{\rm c}[A, {\rm b}{\tilde Q}^{-1}\tilde{v}]_c.\nonumber\\
&&\mu(\tilde{v},{A})=\mu({A}, \tilde{v})=-{\rm c}{\tilde Q}{\rm c}[A, {\rm b}{\tilde Q}^{-1}\tilde{v}]_c
\end{align}
The product structure between $v\in V'_0$ and $\tilde{v}\in V''_0$ is obtained from teh following relation
\begin{equation}
\mu( v, \tilde v)=\mu(\tilde v, v)=c\mu(\tilde v, {\rm b}v)={\rm c}\mu_0({\rm b}v,\tilde v).
\end{equation}
Finally, the product structure of $\tilde{A}\in V''_1$ and $A\in V'_1$ 
\begin{equation}
\mu(A, \tilde{A})= \mu(\tilde{A}, A)\in V'''_0
\end{equation}
since the homotopy is nontrivial only on $V'_1\otimes V'_1$.  
We have 
\begin{equation}
\mu(A, \tilde{A})=\mu(A, {\rm cb}\tilde{A})=-{\rm c}\mu(A, {\rm b}\tilde{A})=
-\frac{1}{2}{\rm c}\langle A, {\rm b}\tilde{A}\rangle_1=
-\frac{1}{2}{\rm c}\tilde{Q}m(A, {\rm b}\tilde{A})
\end{equation}
%From here we obtain that 
%\begin{equation}
%[A, \tilde{A}]=-{\rm b}\mu(A, \tilde{A})-\mu(A, {\rm b}\tilde{A})\in V''_1
%\end{equation}
%and thus ${\rm b}\mu(A, \tilde{A})=-\frac{1}{2}\langle A, {\rm b}\tilde{A}\rangle$. 
%\begin{equation}
%[\tilde{A}, A]= {\rm b}\mu(A, \tilde{A})+\mu({\rm b}\tilde{A}, A),
%\end{eqnarray}
%and thus which belongs to $\mathcal{F}^2=V''_0+V''_1$.
At the same time, the relation 
\begin{equation}
{\rm b}\{A, \tilde{A}\}=\{A, {\rm b}\tilde{A}\}
\end{equation}
fixes the entire bracket structure of $\{A, \tilde{A}\}$. 
%Finally, the anticommutativity gives:
%\begin{equation}
%[\tilde{A}, {A}]+[{A}, \tilde{A}]=(Qn+nQ)(\tilde{A},A) 
%\end{equation}
%and 
%\begin{equation}
%[\tilde{A}_1, \tilde{A}_2]=\mu(\tilde{A}_1, {\rm b}\tilde{A}_2)+\mu({\rm b}\tilde{A}_1,\tilde{A}_2)
%\end{equation}
%\bibliography{biblio}

At the same time we have:
\begin{equation}
\{\tilde{A}, A\}= {\rm b}\mu( \tilde{A}, A)-\mu({\rm b}\tilde{A}, A),
\end{equation}
where the right hand side is known. 
Finally for $\tilde{A}_1, \tilde{A}_2\in V''_1$ we obtain:
\begin{eqnarray}
\{\tilde A_1, \tilde A_2\}=-\mu(\tilde{A}_1, {\rm b}\tilde{A_2})-\mu({\rm b}\tilde{A}_1, \tilde{A_2})={\rm c}\langle {\rm b}\tilde{A}_1, {\rm b} \tilde{A}_2\rangle_1.
\end{eqnarray}

Thus we fixed uniquely the bilinear operations $\mu(~\cdot~,~\cdot ~)$ and $\{\cdot, \cdot\}$. 

Let us summarize the results:

\begin{equation}\label{mutab}
\mu(a_1,a_2)=
\end{equation}
\begin{equation}
\scalebox{0.65}[0.7]{
\begin{tabular}{|l|c|c|c|c|c|r|}
\hline
 \backslashbox{$a_2$}{$a_1$}&          $u_1$ & $A_1$ & $v_1$ &$\tilde A_1$ &$\tilde v_1$& $\tilde u_1$ \\
\hline
$u_2$ &                  $\mu_0(u_1,u_2)$    &$\mu_0(u_2,A_1)$ &${\rm c}\mu_0(u_2, {\rm b}v_1)$&
${\rm c}\mu_0(u_2, {\rm b}\tilde{A_1})$&
${\tilde{Q}}'\mu_0(u_2, \tilde{Q'}^{-1}\tilde{v}_1)$ 
&${\tilde{Q}}''\mu_0(u_2,{\tilde{Q}}''^{-1}\tilde{u}_1)$\\
&  & $-m(A_1,{\rm d}u_2)$  &  &   && \\
\hline
$A_2$     &  $\mu_0(u_1,A_2)$ & $-c[A_1,A_2]_c+$ &${\rm c}\mu_0(A_2, {\rm b} v_1)$  
&  $-\frac{1}{2}{\rm c}\tilde{Q}m(A_2, {\rm b}\tilde{A}_1)$&$-{\tilde Q}''[A_2, {\rm b}{\tilde Q}'^{-1}\tilde{v}_1]_c$&0\\                    
& & $\frac{1}{2}\tilde{Q}m(A_1, A_2)$  &   &   && \\
\hline
$v_2$& ${\rm c}\mu_0(u_1, {\rm b}v_2)$ &  $-{\rm c}\mu_0(A_1, {\rm b} v_2)$ & 0 & $0$ &${\rm c}{\tilde Q}'\mu_0({\rm b}v_2,{\tilde Q}'^{-1}\tilde{v}_1)$& 0\\
\hline
$\tilde A_2$ & ${\rm c}\mu_0(u_1, {\rm b}\tilde{A_2})$ &$-\frac{1}{2}{\rm c}\tilde{Q}m(A_1, {\rm b}\tilde{A}_2)$&  
0 & 0& 0&0\\
\hline
$\tilde v_2$& ${\tilde{Q}}'\mu_0(u_1, \tilde{Q'}^{-1}\tilde{v}_2)$ &   $-{\tilde Q}''[A_1, {\rm b}{\tilde Q}'^{-1}\tilde{v}]_c$& ${\rm c}{\tilde Q}'\mu_0({\rm b}v_1,{\tilde Q}'^{-1}\tilde{v}_2)$ & 0  &0& 0\\
\hline
$\tilde u_2$    & ${\tilde{Q}}''\mu_0(u_1,{\tilde{Q}}''^{-1}\tilde{u}_2)$ &   0& 0 & 0  &0& 0\\
\hline
\end{tabular}}\nonumber\\
\end{equation}
\vspace{3mm}

Here $\tilde{Q}'=\tilde{Q}{\rm c}$, $\tilde{Q}''={\rm c}\tilde{Q}{\rm c}$,and $u_1, u_2\in V_0$, $v_1, v_2\in V'_0$, $A_1, A_2\in V'_1$, $\tilde{A}_1, \tilde{A}_2\in V''_1$, $\tilde{u}_1, \tilde{u}_2\in V'''_0$.

Let us investigate how the differential affects the bilinear operation and show that it is indeed a Calabi-Yau structure.

Let us look at the property 
\begin{eqnarray}\label{qmu}
Q\mu(a_1, a_2)=\mu(Qa_1, a_2)+(-1)^{|a_1|}\mu(a_1, Qa_2). 
\end{eqnarray}
Suppose $a_1=u\in V_0$, $a_2=A \in V'_1$, then 
\begin{align}
Q\mu(u, A)={\rm d^*}\mu_0(uA),&&\\
\nonumber\mu(Qu, A)+\mu(u, QA)=[{\rm d}u, A]_1+\frac{1}{2}\langle {\rm d}u, A\rangle_1+
\mu(u, {\rm d^*}A)=\frac{1}{2}\tilde{Q}{\rm c} n({\rm d}u, A)+\mu(u, {\rm d^*}A),&&
\end{align}
where we remind that $[{\rm d}u,A]_c=0$. Thus, if ${\rm b}\tilde{Q}^{-1}{\rm d^*}|_{V'_1}=\frac{1}{2}{\rm div}$ then we obtain the first property of the Calabi-Yau structure.

Let's consider the case $f_1=A_1, f_2=A_2 \in V'_1$:
\begin{align}
Q\mu(A_1, A_2)=-{\rm d}^*{\rm c}[A_1, A_2]_c={\rm c}{\rm d}^*[A_1, A_2]_c,\\
\mu(QA_1, A_2)+(-1)^{|f_1|}\mu(A_1, QA_2)=-{\rm c}[A_2, {\rm d}^*A_1]_c+{\rm c}[A_1, {\rm d}^*A_2]_c,\nonumber
\end{align}
which leads exactly to the second Calabi-Yau property.
The rest of the cases corresponding to (\ref{qmu}) are verified in a similar way.

Let us verify he associativity conditions: we will verify it in the nontrivial situations.

Let $A_1, A_2, A_3 \in\mathcal{F}^1$. Then 
\begin{align}
&&\mu(\mu(A_1, A_2), A_3)-\mu(A_1, \mu(A_2, A_3))=\\
&&\mu(-{\rm c}[A_1, A_2]_c+\frac{1}{2}\tilde{Q}m(A_1,A_2), A_3)-
\mu(A_1, -{\rm c}[A_2, A_3]_c+\frac{1}{2}\tilde{Q}m(A_2,A_3))=\nonumber\\
&&\frac{1}{2}{\rm c}\tilde{Q}\Big(m(A_3, [A_1, A_2]_c)-\{A_3, m(A_1, A_2)\}-
m(A_1, [A_2,A_3]_c)+\{A_1, m(A_2, A_3)\}\Big)=\nonumber\\
&&\frac{1}{2}{\rm c}\tilde{Q}\Big(m(A_3, [A_1, A_2]_c)-\{A_3, m(A_1, A_2)\}-
m(A_1, [A_2,A_3]_c)+\nonumber\\
&&m([A_1,A_2], A_3)+m(A_2, [A_1,A_3])\Big)=\nonumber\\
&&\frac{1}{2}{\rm c}\tilde{Q}\Big(m(A_3, [A_1, A_2]_c)-\{A_3, m(A_1, A_2)\}-
m(A_1, [A_2,A_3]_c)+\nonumber\\
&&m(dn(A_1, A_2), A_3)-m([A_2,A_1]_c, A_3)+m(A_2, [A_1,A_3]_c)\Big)=\nonumber\\
&&\frac{1}{2}{\rm c}\tilde{Q}\Big(m(A_3, [A_1, A_2]_c)-
m(A_1, [A_2,A_3]_c)-m([A_2,A_1]_c, A_3)+m(A_2, [A_1,A_3]_c)\Big)=\nonumber\\
&&\frac{1}{2}c\tilde{Q}\Big(\{A_1, m(A_2, A_3)\}-\{A_2, m(A_1, A_3)\}\Big)=\nonumber\\
&&\frac{1}{2}{\rm c}\tilde{Q}(m(A_1, {\rm d}n(A_2, A_3))-m(A_2, {\rm d}n(A_1, A_3))=\nonumber\\
&&\frac{1}{2}{\rm c}\tilde{Q}{\rm c}({\rm div}(\mu(n(A_2,A_3),A_1)-\mu(n(A_2,A_3),{\rm div}A_1)-\mu(n(A_1,A_3),A_2)+\mu(n(A_1,A_3),{\rm div}A_2)=\nonumber\\
&&\frac{1}{2}{\rm c}\tilde{Q}{\rm c}~{\rm div}~{\rm b}(\mu(m(A_2,A_3),A_1)-\mu(m(A_1,A_3),A_2))+\nonumber\\
&&\mu (m(A_1, A_3), \frac{1}{2}\tilde{Q}{\rm c}{\rm div}A_2)-
\mu(m(A_2,A_3), \frac{1}{2}\tilde{Q}{\rm c}{\rm div}A_1)=\nonumber\\
&&{\rm d^*}(\mu(m(A_1,A_3),A_2)-\mu(m(A_2,A_3),A_1))+\mu (m(A_1, A_3), {\rm d}^*A_2)-
\mu(m(A_2,A_3), {\rm d}^*A_1).\nonumber
\end{align}
We used the fact that 
$$
{\rm d^*}|_{V''_1}=-\frac{1}{2}{\rm c}\tilde{Q}{\rm c}{\rm div} {\rm b}, \quad m(A,B)={\rm c}n(A, B)
$$
 for $A, B\in V'_1$, and
\begin{align}
\{A,m(B, C)\}=m([A, B]_c, C)+m(B, [A, C]_c), \quad [A,u]_c=n(A, {\rm d}u), \nonumber\\
 {\rm div}\mu(u, A)=\mu(u,{\rm div}A)+n({\rm d u}, A),
\end{align}
where $u\in V_0$, $A,B,C\in V'_1$.

Thus we obtain explicit formulae for the homotopy operation $\nu$:
\begin{align}
\nu(A_1, A_2, A_3)=\mu(m(A_1,A_3),A_2)-\mu(m(A_2,A_3),A_1)\\
\nu(\tilde{v}, A_2, A_3)=\nu(A_2,\tilde{v}, A_3)=-m(A_2, A_3)\tilde{v},\nonumber
\end{align}
where $\tilde{v}\in V''_0$. One can show that this operation is nontrivial
only on the spaces $V'_1\otimes V'_1\otimes V'_1$, $V''_0\otimes V'_1\otimes V'_1$, and $V'_1\otimes V''_0\otimes V'_1$ and their values is given by above formulas. We leave the reader to check all other cases. 
\end{proof}

We will call the resulting homotopy BV-LZ algebra {\it the BV double of Courant algebroid}. 
 
The Theorem \ref{extfull} actually allows us to reformulate Courant algebroid with a Calabi-Yau structure as a homotopy BV algebra. Indeed, one can impose certain properties on the BV-LZ-operations on $(\mathcal{F}^{\bullet}, Q)$ to reproduce the Courant algebroid on $(\mathcal{V}^{\bullet}_c, {\rm d})$ and the related CY structure. Namely, the following Theorem holds.

\begin{Thm}\hfill
\begin{enumerate}
\item The homotopy BV-LZ algebra on the complex $(\mathcal{F}^{\bullet}, Q)$, satisfying the following properties:\\
\begin{enumerate}
\item The action $\mu(\mathcal{F}^0~,~\cdot):\mathcal{F}^{i}\rightarrow \mathcal{F}^{i}$ produces associative algebra structure on $\mathcal{F}^0$ and $\mathcal{F}^0$-module structure on $\mathcal{F}^{i}$ for all $i$.
\item $\{\mathcal{V}_i, \mathcal{V}_j\}\subset\oplus_{k \ge 1}\mathcal{V}_{i+j-k}.$
\item
$
\mu(\mathcal{V}_i, \mathcal{V}_j)\subset\oplus_{k\ge 0}\mathcal{V}_{i+j-k},\nonumber
$
while the restriction $\mu(\mathcal{V}^{1/2}_1,\mathcal{V}^{1/2}_1) |_{\mathcal{V}_0}$
is a symmetric bilinear form.
\item 
${\bf c}~\mu(a_1, a_2)=(-1)^{|a_1|}\mu(a_1,{\bf c}~ a_2)$
\item \label{flinear}The homotopy of the product $m:  \mathcal{F}^i\otimes\mathcal{F}^j\subset\mathcal{F}^{i+j-1}$ is non-vanishing only on a half-complex:
$$m:\mathcal{V}^{1/2}_i\otimes_{\mathcal{F}^0} \mathcal{V}^{1/2}_j\to\mathcal{V}^{1/2}_{i+j-2},$$
\end{enumerate}
%is equivalent to Courant $V_0$-algebroid structure when restricted to subcomplex $(\mathcal{V}^{\bullet}_c, {\rm d})$ with the Calabi-Yau structure given by ${\bf b}\tilde{Q}^{-1}{\rm d^*}|_{V'_1}=\frac{1}{2}{\rm div}[-1]$.
gives rise to the Courant algebroid on $(\mathcal{V}^{\bullet}_c, {\rm d})$ and its BV double on $(\mathcal{F}^{\bullet}, Q)$. The corresponding CY structure is given by ${\rm div}=2{\rm b}\tilde{Q}^{-1}{\rm d^*}|_{V'_1}$.\\
 
\item The associativity homotopy $\nu$ for this homotopy BV-LZ algebra is nontrivial only only on the spaces $V'_1\otimes V'_1\otimes V'_1$, $V''_0\otimes V'_1\otimes V'_1$, and $V'_1\otimes V''_0\otimes V'_1$ and the values are given by:
\begin{align}
\nu(A_1, A_2, A_3)=\mu(m(A_1,A_3),A_2)-\mu(m(A_2,A_3),A_1)\nonumber\\
\nu(\tilde{v}, A_2, A_3)=\nu(A_2,\tilde{v}, A_3)=-m(A_2, A_3)\tilde{v},\nonumber
\end{align}
where $A_i\in V'_1$ $(i=1,2,3)$,  $\tilde{v}\in V''_0$.

\end{enumerate}

\end{Thm}

\begin{proof}
We have to show that the projection on $(\mathcal{V}_c^{\bullet}, Q)=V_0\xrightarrow{\rm d} V'_1$ reproduces the properties of the Courant algebroid. 
Property (1) gives the structure of associative algebra reproducing the operation $\mu_0$. 
It indeed satisfies property:
$$
{\rm d}\mu_0(u_1, u_2)= \mu_0({\rm d}u_1, u_2)+\mu_0(u_1, {\rm d}u_2)
$$
Notice that if $A, B\in V'_1$ we obtain that the bracket defined as 
$[A, B]_c:=\{A,B\}|_{(\mathcal{V}_c^{\bullet}, Q)}$ satisfies:
\begin{align}
&&[A, \mu_0(u_1,u_2)]_c=\mu_0([A, u_1]_c,u_2)+\mu_0(u_1,[A, u_2]_c)\nonumber \\
&&[A, \mu_0(u,B)]_c=\mu_0([A, u]_c, B)+\mu_0(u,[A,B]_c)\nonumber\\
\end{align}

Since $m:V'_1\otimes V'_1\to V_0$ we obtain that the homotopy for the bracket is $n:={\bf b}m$, so that 
\begin{equation}\label{pracom}
[A,B]_c=-[B,A]_c+{\rm d}n(A, B)
\end{equation}
for $A, B\in V'_1$, and  by property (\ref{flinear}) we have:
\begin{eqnarray}
\mu_0(u, n(A,B))=n(\mu_0(u,A),B)=n(A, \mu_0(u,B))
\end{eqnarray}
From the relation between the product and the bracket as well as the property (\ref{pracom}) we have:
$$
[u, A]_c=0, \quad [A, u]_c=n(A, {\rm d}u)
$$
The Leibniz property of the bracket with respect to differential implies
$$
[{\rm d}u_1, u_2]_c=0, \quad [A, {\rm d}u]={\rm d}[A, u]_c
$$

There is one final relation one has to check and it is that $[A,\cdot]_c$ satisfies the derivation property for $n$. 

\begin{equation}\label{nleibpr}
[A, n(B,C)]_c=n([A, B]_c, C)+n(B,[A,C]_c)
\end{equation}

To show that we notice that 
$$
\mu(A_1,A_2)=[A_1,A_2]_1+\frac{1}{2}\langle A_1,A_2 \rangle_1,
$$
where $[A_1,A_2]_1\in V''_1$, $\langle A_1,A_2 \rangle_1\in V_0''$, so that 
$$
\langle A_1,A_2 \rangle_1=\tilde{Q}m(A_1, A_2), \quad [A_1,A_2]_1=-{\bf c}[A_1, A_2]_c.
$$
Note that this immediately leads to:
\begin{equation}
\{A, \langle B,C\rangle\}=
\langle [A, B]_c, C\rangle_1+\langle B, [A, C]_c\rangle_1.
\end{equation}
Using the fact that the $\{\cdot, \cdot\}$ is $Q$-invariant, we have
\begin{equation}
Q\{A, m(B,C)\}=-\{QA, \langle B,C\rangle_1\}+\{A, Qm(B,C)\}.
\end{equation}
Notice, that by construction $\{V'_1, V''_1\}\in V''_1$, 
since
${\bf c}\{f_1, f_2\}=(-1)^{|f_1|-1}\{ f_1, {\bf c}f_2 \}$. and therefore we can project the relation above to:
\begin{equation}
\tilde{Q}\{A, m(B,C)\}=-\{QA, \langle B,C\rangle_1\}+\{A, \tilde{Q}m(B,C)\}.
\end{equation}
Now $\{QA, \langle B,C\rangle_1\}=0$ because $\mu(V''_0, V''_0)=0$ implies $\{V''_0,V''_0\}=0$ and thus 
\begin{equation}
\{A, m(B,C)\}= m([A,B]_c,C)+m(B,[A,C]_c),
\end{equation}
which implies $\ref{nleibpr}$ by application of ${\bf b}$ operator and using its derivation property for the bracket operation.
\end{proof}

\subsection{$C_\infty$-subalgebra}

The $A_{\infty}$-algebra (see, e.g., \cite{stashbook}) is a generalization of a differential graded associative algebra. Namely, consider a graded vector space $V$ with a differential $Q$. Consider the multilinear operations $\mu_i: V^{\otimes i}\to V$ of the degree $2-i$, such that $\mu_1=Q$. 

The graded space $V$ is an $A_{\infty}$-algebra if the operations 
$\mu_n$ satisfy bilinear identity:
\begin{eqnarray}\label{arel}
\sum^{n-1}_{i=1}(-1)^{i}M_i\circ M_{n-i+1}=0 
\end{eqnarray}
on $V^{\otimes n}$,  
where $M_s$ acts on $V^{\otimes m}$ for any $m\ge s$ as the sum of all possible operators of the form 
${\bf 1}^{\otimes^l}\otimes\mu_s\otimes{\bf 1}^{\otimes^{m-s-l}}$ taken with appropriate signs. In other words, 
\begin{eqnarray}
M_s=\sum^{n-s}_{l=0}(-1)^{l(s+1)}{\bf 1}^{\otimes^l}\otimes\mu_s\otimes{\bf 1}^{\otimes^{m-s-l}}.
\end{eqnarray}
Let us write several relations which are satisfied by $Q$, $\mu_1$, $\mu_2$, $\mu_3$:
\begin{eqnarray}
&&Q^2=0,\\
&&Q\mu_2(a_1,a_2)=\mu_2(Q a_1,a_2)+(-1)^{|a_1|}\mu_2(a_1,Q a_2),\nonumber\\
&&Q\mu_3(a_1,a_2, a_3)+\mu_3(Q a_1,a_2, a_3)+(-1)^{|a_1|}\mu_3(a_1,Q a_2, a_3)+\nonumber\\
&&(-1)^{|a_1|+|a_2|}\mu_3( a_1, a_2, Q a_3)=\mu_2(\mu_2(a_1,a_2),a_3)-\mu_2(a_1,\mu_2(a_2,a_3)).\nonumber
\end{eqnarray}

If the operations $\mu_n$ vanish for all $n>k$, such $A_{\infty}$-algebras are sometimes called 
$A_{k}$-algebras \cite{stasheff}, so, e.g., differential graded algebras (DGA) is an $A_{2}$ algebra.   

We observe that putting $\mu_2\equiv \mu$ and $\mu_3=\nu$, these relations are manifestly the same as 
the ones relating $Q$, $\mu$ and $\nu$.

An important object in the theory of $A_{\infty}$-algebras is the generalized Maurer-Cartan (GMC) equation. The space of elements $X\in V$ of degree 1, is known as the space of Maurer-Cartan elements. 
The equation 
\begin{eqnarray}
QX+\sum_{n\ge2}\mu_n(X,...,X)=0
\end{eqnarray}
is called the \emph{generalized Maurer-Cartan equation} on $X$. It is worth mentioning that it is well defined in general only on nilpotent elements, i.e. such that $\mu_n(X,\dots, X)=0$ for $n>k$. This will not be a problem in the following, because the only $A_{\infty}$ algebras we will consider in this article, are $A_{3}$-algebras and therefore GMC equation will be well defined for all elements of degree 1. 
 
Generalized Maurer-Cartan equation is known to have the following infinitesimal symmetry:
\begin{eqnarray}
X\mapsto X +\epsilon (Q\alpha+\sum_{n\ge2,k}(-1)^{n-k}\mu_n(X,...,\alpha,...,X)), 
\end{eqnarray}
where $\epsilon$ is infinitesimal, $\alpha$ is an element of degree 0 and $k$ means the position of $\alpha$ in $\mu_n$. 

A $C_{\infty}$-algebra is an $A_{\infty}$-algebra $(V,\{\mu_n\}^{\infty}_{n=1})$ such that each map  $\mu_n: A^{\otimes n}\rightarrow A$ is a Harrison cochain, i.e. $\mu_n$ vanishes on the sum of all $(p,q)$-shuffles for $p + q$ = n, the sign of the shuffle coming from the grading of $V$ shifted by 1.

One can tensor $C_{\infty}$-algebra and $A_{\infty}$-algebra in general, with an  associative algebra. The induced product structure gives $A_{\infty}$-algebra.

\begin{Thm}
The symmetrized operation $\mu^{sym}(a_1,a_2)=\frac{1}{2}\mu(a_1,a_2)+(-1)^{|a_1||a_2|}\mu(a_2, a_1)$ and the operation $\nu^{sym}$ with nontrivial values given by the following expressions:
\begin{align}
&&\nu^{sym}(A_1, A_2, A_3)=\mu(m(A_1,A_3),A_2)-\frac{1}{2}\mu(m(A_2,A_3),A_1)-\frac{1}{2}\mu(m(A_1,A_2),A_3),\nonumber\\
&&\nu^{sym}(A_1, A_2, \tilde{v})=-\frac{1}{2}m(A_1, A_2)\tilde{v},\nonumber\\
&&\nu^{sym}(\tilde{v}, A_2, A_3)=-\frac{1}{2}m(A_2, A_3)\tilde{v},\nonumber\\
&&\nu^{sym}(A_1,\tilde{v}, A_3)=-m(A_1, A_3)\tilde{v},\nonumber
\end{align} 
where $A_1, A_2, A_3\in V'_1$, $\tilde{v}\in V''_0$, generate a $C_{\infty}$-algebra.
\end{Thm}
\begin{proof}
To show that $Q$,  $\mu$, $\nu$ satisfy the relations of $A_{\infty}$ algebra (with all higher products vanishing), one must check the higher associativity relations. The first one is:
\begin{align}\label{23rel}
&&(-1)^{|a_1|}\mu(a_1,\nu(a_2,a_3,a_4))+\mu(\nu(a_1,a_2,a_3),a_4)=\\
&&\nu(\mu(a_1,a_2),a_3,a_4)-\nu(a_1,\mu(a_2,a_3),a_4)+\nu(a_1,a_2,\mu(a_3,a_4)).\nonumber
\end{align}
Note, that the higher relations are satisfied automatically by counting degrees. 
Let us prove (\ref{23rel}) in the most nontrivial case, when $a_i=A_i\in \mathcal{F}^1$: 

\begin{align}
&&(-1)^{|A_1|}\mu(A_1,\nu(A_2,A_3,A_4))+\mu(\nu(A_1,A_2,A_3),A_4)=\nonumber\\
&&\frac{1}{2}\tilde{Q}m(A_1,A_3){\rm c}m(A_2,A_4)-
\frac{1}{2}\tilde{Q}m(A_1,A_2){\rm c}m(A_3,A_4)-\nonumber\\
&&-\frac{1}{2}\tilde{Q}m(A_2,A_4){\rm c}m(A_1,A_3)+
\frac{1}{2}\tilde{Q}m(A_1,A_4){\rm c}m(A_2,A_3)=\nonumber\\
&&-\frac{1}{2}\tilde{Q}m(A_1,A_2){\rm c}m(A_3,A_4)+
\frac{1}{2}\tilde{Q}m(A_1,A_4){\rm c}m(A_2,A_3)=\nonumber\\
&&\nu(\mu(A_1,A_2),A_3,A_4)-\nu(A_1,\mu(A_2,A_3),A_4)+\nu(A_1,A_2,\mu(A_3,A_4)).
\end{align}

Similarly, one can check this relation for symmetrized version. Finally, it is easy to check that trilnear relation satisfy $C_{\infty}$-algebra relation for shuffled permutations:
\begin{equation}
\nu(a_1, a_2, a_3)-(-1)^{a_1||a_2|}\nu(a_2, a_1, a_3)+(-1)^{|a_1|(|a_2|+|a_3|)}\nu(a_2, a_3, a_1)=0
\end{equation}

which is the last step required by the theorem.
\end{proof}

\subsection{$\mathcal{O}_X$-Courant algebroid and the cyclic structure}

Consider the subcomplex $(\mathcal{F}^{\bullet}_c,Q)$ of 
$(\mathcal{F}^{\bullet},Q)$ consisting of the following subspaces:
\begin{align}
\mathcal{F}_c^0=V_0, \quad \mathcal{F}_c^1=\{(A,v):~ A\in V'_1,v\in V'_0, ~v=-\tilde{Q}^{-1}{\rm d}^*A\},\quad \mathcal{F}_c^2=V''_1,\quad \mathcal{F}_c^3=V''_0
\end{align}
We denote elements of $\mathcal{F}_c^1$ as $A^c=(A,-\tilde{Q}{\rm d}^*A)$.  
Note that $(\mathcal{F}_c^{\bullet}, Q)$ is not ${\bf b}$-invariant. At the same time, the following statement is true.

\begin{Prop}
The complex $(\mathcal{F}_c^{\bullet}, Q)$ is invariant under the $C_{\infty}$ operations $\mu^{sym}$, $\nu^{sym}$ and induce the $C_{\infty}$ structure on $(\mathcal{F}_c^{\bullet}, Q)$.
\end{Prop}
\begin{proof}
We have to check the cases which involve the modified vector space $\mathcal{F}^1_c$, $\mathcal{F}^2_c$:
\begin{eqnarray}
&&\mu^{sym}(u, A^c)=\mu^{sym}(u, A)=\\
&&\mu(uA)-\mu(u, \tilde{Q}^{-1}{\rm d}^*A)-\frac{1}{2}{\bf c} n(A, {\rm d} u)=\nonumber\\
&&\mu(u,A)-\mu(u, {\bf c}{\rm div}A)-\frac{1}{2}{\bf c} n(A, {\rm d} u)=\mu(u,A)-\frac{1}{2}{\bf c}{\rm div}\mu(u, A)=\mu(u,A)^c.\nonumber
\end{eqnarray}
Here we used the fact that ${\rm b}\tilde{Q}^{-1}{\rm d^*}|_{V_1}=\frac{1}{2}{\rm div}$.
Similarly we have: 
$$
\mu^{sym}(A^c_1,A^c_2)=-{\bf c}[A_1,A_2]_c-\frac{1}{2}{\rm d}m(A_1,A_2)-\mu(\tilde{Q}^{-1}{\rm d}^*A_1,A_2)-\mu(A_1,\tilde{Q}^{-1}{\rm d}^*A_2)\in \mathcal{F}_c^2.
$$
One doesn't need to check the trilinear operations since the modifications doesn't contribute there.
\end{proof}

Given a Courant $\mathcal{O}_X$-algebroid with the CY structure given by the volume form $\Omega$, we can define the following symmetric pairing:
\begin{equation}
(~\cdot~,~ \cdot ~): \mathcal{F}^i\otimes \mathcal{F}^{3-i}\to \mathbb{K},
\end{equation}
so that nontrivial pairings between subs
\begin{align}\label{omegapair}
(u,\tilde u)=2\int {\bf b}\tilde{Q}^{-1}{\bf b}\mu(u,\tilde u)~\Omega, \quad 
(A,\tilde A)=-\int n(A,{\bf b}\tilde{A})~ \Omega, \quad (v, \tilde v)=-2\int {\bf b}\tilde{Q}^{-1}{\bf b}\mu(v,\tilde{v})~\Omega,
\end{align}
where $u\in V_0$, $v\in V'_0$, $\tilde{v}\in V''_0$, $\tilde u\in V'''_0$,  $A\in V'_1$, 
$\tilde{A}\in V''_1$.

The subcomplex $(\mathcal{F}^{\bullet}_c, Q)$ satisfies the following properties:
\begin{Prop}
Consider the $\mathcal{O}_X$-Courant algebroid on a compact manifold $X$ with the CY structure generated by the volume form $\Omega$, then the corresponding pairing
$( \cdot, \cdot )_{\Omega}$ is such that 
\begin{itemize}

\item The pairing (\ref{omegapair}) satisfies the following relations:
\begin{align}
 (Qa_1,a_2)+(-1)^{|a_1||a_2|}(Qa_2,a_1)=0, &&\\ 
 ({\bf c}a_1,a_2)+(-1)^{|a_1||a_2|}({\bf c}a_2,a_1)=0,\quad ({\bf b}a_1,a_2)-(-1)^{|a_1||a_2|}({\bf b}a_2,a_1)=0.&&\nonumber
\end{align}

\item There is a following orthogonal decomposition:
$$
(\mathcal{F}^{\bullet},Q)=(\mathcal{F}_c^{\bullet},Q)\oplus(\mathcal{G}^{\bullet}, Q),
$$
where $(\mathcal{G}^{\bullet}, Q)$ is an acyclic complex.

\item The $C_{\infty}$-structure on complex $(\mathcal{F}_c^{\bullet}, Q)$ is cyclic, namely 
the forms 
\begin{align}
&(\Phi_1, \Phi_2)=(Q\Phi_1, \Phi_2),\nonumber\\
&(\Phi_1, \Phi_2, \Phi_3)=(\mu^{sym}(\Phi_1, \Phi_2), \Phi_3),\nonumber\\
&(\Phi_1, \Phi_2, \Phi_3, \Phi_4)=(\nu^{sym}(\Phi_1,\Phi_2, \Phi_3), \Phi_4)
\end{align}
produce the following sign under the cyclic permutation:
\begin{equation}
(\Phi_1, \Phi_2,\dots, \Phi_n)=(-1)^{n-1}(-1)^{|\Phi_n|(|\Phi_{1}|+\dots+|\Phi_{n-1}|)}(\Phi_n, \Phi_1,\dots, \Phi_{n-1}).
\end{equation}
 
\end{itemize}
\end{Prop}
\begin{proof}
The first two parts of the statement follows from the results of Section 1. 
To obtain the third we need to compare:
\begin{align}\label{AAA1}
(\mu^{sym}(A^c_1,A^c_2), A_3^c)=&&\\
(-{\bf c}[A_1,A_2]_c-\frac{1}{2}{\rm d}m(A_1,A_2)-\mu(\tilde{Q}^{-1}{\rm d}^*A_1,A_2)-\mu(A_1,\tilde{Q}^{-1}{\rm d}^*A_2),A_3)&&\nonumber\\
(\mu^{sym}(A^c_1,A^c_3), A^c_2)=&&\label{AAA2}\\
(-{\bf c}[A_1,A_3]_c-\frac{1}{2}{\rm d}m(A_1,A_3)-\mu(\tilde{Q}^{-1}{\rm d}^*A_1,A_3)-\mu(A_1,\tilde{Q}^{-1}{\rm d}^*A_3), A_2).&&\nonumber
\end{align}
Using the following relations: 
\begin{align}
[A_1,n(A_2,A_3)]_c=n([A_1,A_2]_c,A_3)+n([A_1,A_3]_c,A_2), \\ 
\int n({\rm d}n(A_1,A_2), A_3)\Omega=-\int \mu({\rm div}A_3[1],n(A_1,A_2)),\nonumber
\end{align}
we see that the expressions (\ref{AAA1}) and (\ref{AAA2}) are the same up to a negative sign as expected.

Similarly for trilinear operation we have:
\begin{align}
(A_4,\nu^{sym}(A_1, A_2, A_3))=(A_4,\mu(m(A_1,A_3),A_2)-\frac{1}{2}\mu(m(A_2,A_3),A_1)-\frac{1}{2}\mu(m(A_1,A_2),A_3)),\nonumber\\
(A_1,\nu^{sym}(A_2, A_3, A_4))=(A_1,\mu(m(A_2,A_4),A_3)-\frac{1}{2}\mu(m(A_3,A_4),A_2)-\frac{1}{2}\mu(m(A_2,A_3),A_4)),\nonumber
\end{align}
which are equal to each other. We leave to the reader to check the rest of the relations.
\end{proof}

\subsection{$L_\infty$-subalgebra} To describe the $L_{\infty}$ structure on the complex $(\mathcal{F},Q)$ one doesn't need to compute much, because the homotopy was already computed by Roytenberg and Weinstein on the half-complex \cite{weinstein}, namely, on $V'_1$ we obtain:
\begin{align}
[[A_1, A_2],A_3]+[[A_3, A_1],A_2]+[[A_2, A_3],A_1]=d[A_1, A_2, A_3],
\end{align}
where $[~\cdot~,~\cdot~]$ denotes antisymmetrized version of $\{~\cdot~,~\cdot~\}$-operation and 
\begin{equation}
[A_1, A_2, A_3]=\frac{1}{3}(n(A_1,[A_2,A_3])+n(A_2,[A_1,A_3])+n(A_3,[A_1,A_2])),
\end{equation}
which was computed in \cite{weinstein}.
An elementary computation which uses property (\ref{cbrprop}) is as gives the following proposition.
\begin{Prop} 
The $L_{\infty}$ structure generated by antisymmetrized version of $\{~\cdot~,~\cdot~\}$ is an $L_{3}$-algebra
and the nontrivial homotopies are nontrivial on 
$\mathcal{V}_1$ only:
\begin{align}
[A_1, A_2, A_3]=\frac{1}{3}(n(A_1,[A_2,A_3])+n(A_2,[A_1,A_3])+n(A_3,[A_1,A_2]))\nonumber\\
[A_1,A_2,\tilde{A}]=-\frac{1}{3}(m(A_1,[A_2,{\bf b}\tilde{A}])+m(A_2,[A_1,{\bf b}\tilde{A}])+m({\bf b}\tilde{A},[A_1,A_2])),
\end{align}
where $A_1,A_2, A_3 \in V'_1$ and $\tilde{A}\in V''_1$   
This trilinear operation satisfies the following derivation property:
\begin{equation}
{\bf b}[a_1,a_2, a_3]=-[{\bf b}a_1, a_2,a_3]-(-1)^{|a_1|}[a_1,{\bf b}a_2, a_3]-(-1)^{|a_1|+|a_2|}[a_1,a_2,{\bf b} a_3].
\end{equation}
\end{Prop}

\section{Vertex algebra realization}

\subsection{VOA and Topological VOA}

Let's remind some notion from the theory of vertex operator (super)algebras (see, e.g.,  \cite{benzvi}, \cite{fhl}).  

Let us briefly remind the basic ingredients of a vertex operator sueperalgebra (VOA), which are summarized in the points $(1)$-$(5)$ below:

\begin{enumerate}

\item
Graded vector space $$\mathscr{V}=\oplus_i \mathscr{V}_i=\oplus_{i,\mu}\mathscr{V}_i(\mu),$$ where $i$ represents grading of $V$ with respect to {\it conformal weight} and $\mu$ 
represents {\it fermionic grading} of $\mathscr{V}_i$.\\

\item Vertex operators $Y: \mathscr{V}\rightarrow {\rm End}(\mathscr{V})[[z^{\pm 1}]]$ explicitly described as:

$$Y: A\mapsto A(z)=\sum_{n\in \mathbb{Z}}A_nz^{-n-1},$$

where $A_n$ are called modes of the vertex operator $A(z)$.\\

\item A vector $|0\rangle\in V_0[0]$, such that:

$$\lim_{z\rightarrow 0}A(z)|0\rangle=A; \quad Y(|0\rangle,z)=Id_V.$$

\item Locality property: 
$$(z-w)^N[A(z), B(w)]=0,$$ 
where $[,\cdot, \cdot]$ stands for supercommutator.\\

\item Virasoro element $|L\rangle\in \mathscr{V}_2(0)$, such that $L(z)=\sum_nL_nz^{-n-2}$ satisfies the relations of Virasoro algebra:
$$
[L_n, L_m]=(n-m)L_{n+m}+\frac{c}{12}(n^3-n)\delta_{n,-m}
$$
$L_0$ provides grading with respect to conformal weight and $L_{-1}$ is a translation operator: 
$$[L_{-1}, A(z)]=\partial_zA(z), \quad L_{-1}|0\rangle=0.$$
\end{enumerate}

An important feature of vertex operator algebra is the so-called 
operator product expansion (OPE):
\begin{equation}
A(z)B(w)C=\sum_n\frac{(A_nB) (w)C}{(z-w)^n}\in \mathscr{V}((w))((z-w)),
\end{equation}
where $A,B,C\in \mathscr{V}$ and by locality it is an element of of 
$\mathscr{V}[[z,w]][z^{-1}, w^{-1},(z-w)^{-1}]$ which is further embedded into $\mathscr{V}((w))((z-w))$. 
It turns out that the singular part in $z-w$ of OPE determines the commutators between the modes of vertex operators and thus it is common to write only singular terms:
\begin{equation}
 A(z)B(w)\sim \sum_{n\in \mathbb{N}}\frac{(A_nB) (w)}{(z-w)^n},
\end{equation}
On the other hand, the nonsingular part of OPE is called a normal ordered product and is denoted  as $:A(z)B(w):$.

{\it Topological vertex operator algebra} (TVOA) is a vertex superalgebra (see e.g.\cite{benzvi}) that has an additional odd operator 
$Q$ which makes the graded vector space of VOA a chain complex, such that the Virasoro element $L(z)$ is $Q$-exact. The formal definition is as follows (see e.g. \cite{zuck} for more details).

We call $\mathscr{V}$ a {\it topological vertex operator algebra (TVOA)} if there exist four elements: $J\in \mathscr{V}_1(1)$, $b\in \mathscr{V}_2(-1)$, $F\in \mathscr{V}_1(0)$, $L\in \mathscr{V}_2(0)$, such that 
\begin{equation}
[Q,b(z)]=\mathcal{L}(z),\quad  Q^2=0,\quad b_0^2=0,
\end{equation}
where 
\begin{eqnarray}
&&Q=J_0,~J(z)=\sum_n J_nz^{-n-1},\\
&&b(z)=\sum_n b_nz^{-n-2}, \quad \mathcal{L}(z)=\sum_n\mathcal{L}_nz^{-n-2},~
F(z)=\sum_nF_nz^{-n-1}.\nonumber
\end{eqnarray}
Here $\mathcal{L}(z)$ is the Virasoro element of $\mathscr{V}$; the operators $F_0$, $\mathcal{L}_0$ are diagonalizable, commute with each other and 
their egenvalues coincide with fermionic grading and conformal weight correspondingly.

Lian and Zuckerman \cite{lz} observed that each TVOA possesses operations leading to homotopical algebra. Namely, one can define two operations which are cochain maps with respect to $Q$: 
\begin{eqnarray}\label{mub}
\mu(a_1,a_2)=Res_z\frac{ a_1(z)a_2}{z},\quad \{a_1,a_2\}=(-1)^{|a_1|} Res_z(b_{-1}a_1)(z)a_2.
\end{eqnarray}

\begin{Thm}
The operations $\mu(\cdot~,~ \cdot)$, $\{\cdot~,~\cdot\}$ satisfy the relations of the homotopy BV-LZ algebra, where: 
\begin{align}
m(a_1,a_2)=\sum_{i\ge 0}\frac{(-1)^i}{i+1}Res_wRes_{z-w}(z-w)^iw^{-i-1}b_{-1}
(a_1(z-w)a_2)(w)\mathbf{1},
\nonumber\\
\nu(a_1,a_2,a_3)=\sum_{i\ge 0}\frac{1}{i+1}Res_zRes_w w^iz^{-i-1}(b_{-1}a_1)(z)a_2(w)a_3+\nonumber\\
\ \ \ \ \ \ (-1)^{|a_1||a_2|}\sum_{i\ge 0}\frac{1}{i+1}Res_wRes_z z^iw^{-i-1}(b_{-1}a_2)(w)a_1(z)a_3.
\end{align}
\end{Thm}

\begin{Rem}
An interesting question is whether one can extend this algebra to become $G_{\infty}$- or furthermore $BV_{\infty}$-algebra \cite{gj}, \cite{tamarkin}. The first attempt was in \cite{zuck} with a follow-up by \cite{voronov} (with the use of \cite{huang}), which led to the extension to what the author called ``weak" $G_{\infty}$-algebra. At the same time, in the papers \cite{gorbounov}, \cite{vallette} it was proven that BV-LZ homotopy algebra can be extended to $G_{\infty}$- and $BV_{\infty}$-algebra from \cite{tamarkin} correspondingly, but for vertex algebras of positive conformal weight only. Notice, however, that semi-infinite complex is not one of them since it involves $c$-vertex operator corresponding to the state of conformal dimension -1. So, existence of such an extance of $G_{\infty}$-, $BV_{\infty}$ -algebras in this case is still an open problem.
\end{Rem}

\begin{Rem} One can also describe parameter-dependent version of Lian-Zukerman operations which can be applied to 2d quantum field theories of various kinds, see \cite{aztcft1}, \cite{aztcft2}. The higher homotopies in this can be described as integrals over various compactified moduli spaces of configurations of points, in particular, to Stasheff polytopes \cite{stashbook}.
\end{Rem}

A natural example of such TVOA, which will be crucial in the following, is the semi-infinite cohomology (or simply BRST) complex for the Virasoro algebra \cite{fgz} of some VOA with central charge equal to $26$. The necessary setup for the construction of the semi-infinite complex (for more details, see, e.g, \cite{fgz}) 
is the VOA $\Lambda$,obtained from the following super Heisenberg algebra:
\begin{equation}
\{b_n, c_m\}=\delta_{n+m,0}, \quad n,m \in \mathbb{Z}.
\end{equation}
One can construct the space of $\Lambda$ as a Fock module:
\begin{eqnarray}
\Lambda&=&\{b_{-n_1}\dots b_{-n_k}c_{-m_1}\dots c_{-m_l}\mathbf{1},  n_1,\dots,  n_k>0, m_1,\dots, m_l>0;\nonumber\\
&& c_k \mathbf{1}=0,\ k\geqslant 2;\quad b_k\mathbf{1}=0, \ k\geqslant -1\}.
\end{eqnarray}
Then one can define two fields:
\begin{equation}
b(z)=\sum_m b_mz^{-m-2}, \qquad c(z)=\sum_n c_nz^{-n+1},
\end{equation}
which according to the commutation relations between modes have the following operator product:
\begin{equation}
b(z)c(w)\sim\frac{1}{z-w}.
\end{equation}
The Virasoro element is given by the formula:
\begin{equation}
L^{\Lambda}(z)=2:\partial b(z)c(z):+:b(z)\partial c(z):,
\end{equation}
so that $b(z)$ has conformal weight $2$, and $c(z)$ has conformal weight $-1$. Here, as usual, $: :$ stand for normal ordered product, e.g. $b(z)c(w)=\frac{1}{z-w}+:b(z)c(w):$ (for more details see e.g. \cite{benzvi}, Section 2.2). 

Now let $V$ be a VOA with the Virasoro element $L(z)$. Let us consider the tensor product $V\otimes\Lambda$. Then we have the following Proposition. 
\\

\begin{Prop} If $\mathscr{V}$ is a VOA with the central charge of Virasoro algebra equal to 26, then $\mathscr{V}\otimes \Lambda$ is a topological vertex algebra, where 
\begin{eqnarray}
&&J(z)=c(z)L(z)+:c(z)\partial c(z) b(z):+\frac{3}{2}\partial^2 c(z), \quad G(z)=b(z),\nonumber\\ 
&&F(z)=:c(z)b(z):, \quad \mathcal{L}(z)=L(z)+L^{\Lambda}(z).
\end{eqnarray}
\end{Prop}
The operator $Q=J_0$ is traditionally called the $BRST$ $operator$ and the eigenvalue of $F_0$, i.e. fermionic grading is usually called the $ghost$ $number$.

However, there is a subspace of   $\mathscr{V}\otimes \Lambda$ which is a complex with respect to BRST operator regardless of the value of the central charge.

\subsection{Homotopy BV algebra on the complex of light modes}

We call the subspace 
$F_{\mathcal{L}_0}\subset\mathscr{V}\otimes \Lambda$ 
which is annihilated by $\mathcal{L}_0$, i.e. the space of states of conformal weight zero, the space of {\rm light modes}. Since the conformal weights in the VOA are greater than zero, the following proposition holds.

\begin{Prop} \hfill
\begin{enumerate}
\item The space of light modes $F_{\mathcal{L}_0}$ 
is linearly spanned by the elements which correspond to the operators:
\begin{eqnarray}
&&u(z), \quad c(z)A(z), \quad \partial c(z) a(z), \quad c(z)\partial c(z) \tilde A(z), \nonumber\\
&& c(z)\partial^2 c(z) \tilde a(z), \quad c(z)\partial c(z)\partial^2 c(z) \tilde u(z).  
\end{eqnarray}
Here $u, \tilde u$, $a, \tilde a\in \mathscr{V}_0$ and $A, \tilde A\in\mathscr{V}_1$. \\
\item The space $F_{\mathcal{L}_0}$ is a chain complex, quasi-isomorphic to the semi-infinite complex when the central charge is equal to 26. The differential acts on $F_{\mathcal{L}_0}$ in the 
following way (we recall that the grading is given by the ghost number):

\begin{align}
\xymatrixcolsep{30pt}
\xymatrixrowsep{3pt}
\xymatrix{
\mathscr{V}_{0}\ar[r]^{L_{-1}}& \mathscr{V}_{1}[-1]
\ar[r]^{-\frac{1}{2}L_1} & \mathscr{V}_{0}[-2]& \\
&&  &&\\
& \bigoplus & \bigoplus & \\
&&  &&\\
& \mathscr{V}_{0}[-1]\ar[uuuur]^{\rm s}\ar[r]^{L_{-1}} & 
\mathscr{V}_{1}[-2]\ar[r]^{\frac{1}{2}L_1}  & \mathscr{V}_{0}[-3]
}
\end{align}

where $s$ is a suspension, $\mathscr{V}_{i}$ $(i=0,1)$ are the spaces of the elements of $\mathcal{V}$ of conformal weight $i$. 
\end{enumerate}
\end{Prop}

\subsection{Quasiclassical limit: Courant algebroid from vertex algebroid}

Vertex algebroid on a complex $(\mathcal{V}_v^{\bullet}, {\rm d}):~
\mathcal{V}_v^0\xrightarrow{{\rm d}_h} \mathcal{V}_v^1$ is defined as follows. 

Let $\mathcal{V}_v^0$ be an abelian 
algebra and $\mathcal{V}^1_v$ is a $\mathcal{V}^0_v$ module.
Assume there also exists a pairing 
\begin{equation}
*: ~\mathcal{V}_v^0\otimes\mathcal{V}_v^{1}  \rightarrow\mathcal{V}_v^{1}, ~{\rm i.e.}~  
f\otimes A  \mapsto  f*A, 
\end{equation}
such that $1* A = A$, equipped with
a structure of a Leibniz $\mathbb{C}$-algebra 
\begin{equation}
[\ \cdot,\cdot \ ]_h :
\mathcal{V}_v^i\otimes\mathcal{V}_v^j\to \mathcal{V}_v^{i+j-1},
\end{equation}
 so that $[\mathcal{V}_v^0, \cdot]\equiv 0$, and a symmetric $\mathbb{C}$-bilinear pairing 
\begin{equation}
n_{h}: \mathcal{V}_v^1\otimes\mathcal{V}_v^1\to \mathcal{V}_v^0\end{equation}

which satisfies the following properties:

\begin{align}
&& f*(g*A) - (fg)*A  =  [A,f]_h*{\rm d}_hg +
[A,g]_h*{\rm d}_hf,\nonumber\\
&&[A_1,f*A_2]_h  =  [A_1, f]_h*v_2 + f*[A_1,A_2]_h, 
\nonumber\\
&&[A_1,A_2]_h + [A_2,A_1]_h  =  d_hn_h(A_1,A_2),
\quad
[f*A~, g]_h = f[A~,g]_h,  \nonumber\\
&&n_h(f*A_1, A_2)  =  fn (A_1,A_2) -
[A_1,[A_2,f]], \nonumber\\
&&[A, n_h(A_1, A_2)]  = n_h([A,A_1],A_2) +
n_h(A_1,[A,A_2]),\nonumber \\
&&{\rm d}_h(fg)  =  f*{\rm d}_hg + g*{\rm d}_hf, \nonumber\\
&&[A,{\rm d}_hf] =  {\rm d}[A,f]_h, \quad
n_h (A,{\rm d}_h f)  =  [A, f]_h, \quad [A, fg]_h=[A, f]_hg+f[A, g]_h,
\end{align}
where $A,A_1,A_2\in \mathcal{V}_v^1$, $f,g\in \mathcal{V}_v^0$.

The resulting object is known as {\it vertex algebroid} \cite{cdr}, \cite{malikov}. 

It is known that every positively graded VOA gives rise to vertex 
algebroid, where 
$$\mathcal{V}^1_v=\mathscr{V}_1[-1], \quad \mathcal{V}^0_v=\mathscr{V}_0,\quad {\rm d}_h=L_{-1},$$
and the corresponding operations are given by:
\begin{align}
&&u*A=Res_z\Big(\frac{u(z)A}{z}\Big), \quad [A, u]_h=Res_z(A(z) u), \quad [A_1, A_2]_h=Res_z(A_1(z) A_2), \nonumber\\
&&n_h(A_1, A_2)=Res_z(zA_1(z)A_2)
\end{align}

The $L_1$-operator, i.e. ${\rm div}_h=L_1:\mathcal{V}_v^1\rightarrow \mathcal{V}_v^0$ gives a {\it Calabi-Yau structure for  vertex algebroid} \cite{malikov}, namely an operator ${\rm div}$, which satisfies the following properties:
\begin{align}
{\rm div}_h~{\rm d}=0,&&\\
{\rm div}_h(u* A)=u~{\rm div}_hA+n_h({\rm d}u, A),&&\\
{\rm div_h}[A_1,A_2]_h=[A_1,{\rm div}_hA_2]_h-[A_2, {\rm div}_hA_1]_h.&&
\end{align}

Thus we obtain the following Proposition.

\begin{Prop}
Operations (\ref{mub}) restricted to light cone complex give a functor from category of vertex algebroids to the category of homotopy BV algebras of LZ type on the complex $(F^{\bullet}_{\mathcal{L}_0}, Q)$.
\end{Prop}

The {\it commutative vertex algebroid} with Calabi-Yau structure is the one which has trivial bracket and CY structure. 
Let us assume that $\mathcal{V}^1_v=\mathcal{V}^1_c[[h]]$, where $h$ is a formal parameter, and we assume that the resulting vertex algebroid $\mathcal{V}^1_c=\mathcal{V}^1_v/h\mathcal{V}^1_v$ is commutative, namely both pairing and bracket are both identically equal to zero. Let $\mathcal{V}^{\bullet}\to \mathcal{V}^\bullet_c$ be the $\h\rightarrow 0$ reduction map, such that $A\rightarrow {\bar A}$
Then, the following operation 
\begin{equation}
\mu_0(f,\bar{A})=\overline{f*A}
\end{equation}
gives $\mathcal{V}_c^\bullet$, where $\mathcal{V}^0_v=\mathcal{V}^0_c$ the structure of the Courant $\mathcal{V}_0$-algebroid.
In addition we can define 
\begin{align}
&&[\bar{A}, u]_c=\overline{\frac{1}{h}[A, u]_h}, \quad 
[\bar{A}_1,\bar{A}_2]_c=\overline{\frac{1}{h}[A_1, A_2]_h}, \nonumber\\
&&n(A_1,A_2)=\overline{\frac{1}{h}n_h(A_1,A_2)},\quad 
{\rm div}\bar{A}=\overline{\frac{1}{h}{\rm div}_hA}, \quad {\rm d}u=\overline{{\rm d}_hu}.
\end{align}
Then the following Proposition holds \cite{bressler}.

\begin{Prop}
 
The bracket $[~\cdot ~,~ \cdot ~]_c$ and the $\mathcal{V}^0_c$-bilinear  pairing $n(~\cdot~,~\cdot ~)$, together with operation ${\rm div}$ satisfy the axioms of Courant algebroid with a Calabi-Yau structure on $(\mathcal{V}^{\bullet}_c, {\rm d})$.

\end{Prop}

\subsection{Quasiclassical limit for homotopy BV algebra}

Consider the VOAs on the spaces of the form $\mathscr{V}^h=\oplus_{n\in \mathbb{Z}_{\ge 0}}\mathscr{V}^h_n$, where $n$ stands for the grading with respect to conformal weight, 
and $\mathscr{V}^h_n=\mathscr{V}_n((h))$, and $\mathscr{V}_n$ is some vector spaces with $h$ being a formal parameter. Let us also denote $\mathscr{V}=\oplus_{n\in \mathbb{Z}_{\ge 0}}\mathscr{V}_n$. 

\noindent Then, we require the operator products to meet the following conditions: \\

\begin{itemize}

\item  For any state $A$ the associated vertex operator $A(z)=\sum_nA_nz^{-n-1}$ is such that its Fourier modes $A_n\in End_{\mathbb{C}((h))}(\mathscr{V})$, i.e. they commute with the natural action of $\mathbb{C}((h))$ on $\mathscr{V}$.\\
  
\item  Let $A,B\in \mathscr{V}$ and $A(z)=\sum_nA_nz^{-n-1}$,  then
\begin{equation}
A_nB\in h\mathscr{V}[h],  \quad n\ge 0. 
\end{equation}
\end{itemize}

Moreover, we put the following conditions on the Virasoro action:\\

\begin{itemize}

\item $\mathscr{V}$ is invariant under the action of the operator $L_{-1}$. \\

\item  Let $A\in \mathscr{V}_1$, then $L_1A\in h\mathscr{V}_0$.\\

\end{itemize}

We note immediately that $\mathscr{V}$ is a vertex algebra as well. 
For such parameter-dependent $VOA$ the 
complex 
\begin{equation}\label{+comp}
\mathscr{V}_0[[h]]\xrightarrow{L_{-1}} \mathscr{V}_1[[h]][-1]
\end{equation}
has the structure of vertex algebroid with Calabi-Yau structure given by $L_1$ operator and the following proposition holds.

\begin{Prop}
Vertex algebroid corresponding to (\ref{+comp}) generates a vertex subalgebra $\mathscr{V}^h_+$, and $\mathscr{V}^h_+\otimes \Lambda$ is invariant under the action of BRST operator. 
\end{Prop}

We denote the resulting complex of light modes as $(\mathcal{F}^{\bullet}_{+,\mathcal{L}_0}, Q^h)$ and the one for $\mathscr{V}$ as $(\mathcal{F}^{\bullet}_{\mathcal{L}_0}, Q)=(\mathcal{F}^{\bullet}_{+,\mathcal{L}_0}, Q^h)|_{h=1}$.

Notice, that  $\mathscr{V}_0\xrightarrow{ {L}_{-1}} \mathscr{V}_1[-1]$ has the structure of Courant algebroid, where the Calabi-Yau structure is given by $\frac{1}{h}L_1$. However, the quasiclassical limit of the corresponding homotopy BV algebra is a little more involved.

\begin{Prop}
 Let $(\mathcal{F}^{\bullet}_{0,\mathcal{L}_0}, Q^h)$ 
 be the subcomplex of $(\mathcal{F}^{\bullet}_{+,\mathcal{L}_0}, Q^h)$, spanned by the elements corresponding to the following vertex operators: 
\begin{align}
&&u(z),\quad c(z)A(z),\quad h\partial c(z) v(z), \quad hc\partial c(z) \tilde A(z),\nonumber\\
&& hc(z)\partial^2 c(z) \tilde v(z),\quad h^2c(z)\partial c(z)\partial^2 c(z) \tilde u(z),
\end{align}
where $u,\tilde u, v, \tilde v\in \mathscr{V}_0$, $A, \tilde A\in \mathscr{V}_1$. Then the following two statements hold:
\begin{itemize}
\item  The complex $(\mathcal{F}_{\mathcal{L}_0}^{\bullet}, Q^h)$ is isomorphic to $(\mathcal{F}^\bullet_{\mathcal{L}_0}, Q)$. Moreover, it is invariant with respect to the operators $h^{-1}b_0$ and $hc_0$.

\item The Lian-Zuckerman operations act as follows on $(\mathcal{F}^{\bullet}, Q)$:
\begin{eqnarray}
&&\mu_h: \mathcal{F}_{\mathcal{L}_0}^i\otimes \mathcal{F}_{\mathcal{L}_0}^j\to \mathcal{F}_{\mathcal{L}_0}^{i+j}[h], \quad  m_h: \mathcal{F}_{\mathcal{L}_0}^i\otimes \mathcal{F}_{\mathcal{L}_0}^j\to \mathcal{F}_{\mathcal{L}_0}^{i+j-1}[h], \nonumber\\
&& \nu_h: \mathcal{F}_{\mathcal{L}_0}^i\otimes\mathcal{F}_{\mathcal{L}_0}^j\otimes \mathcal{F}_{\mathcal{L}_0}^k\to \mathcal{F}_{\mathcal{L}_0}^{i+j+k-1}[h], 
\quad  \{\mathcal{F}_{\mathcal{L}_0}^i,\mathcal{F}_{\mathcal{L}_0}^j\}_h\to h\mathcal{F}_{\mathcal{L}_0}^{i+j-1}[h].
\end{eqnarray}
\end{itemize}
\end{Prop}

\noindent Therefore, one can consider the expansions: 
\begin{align}
{\bf b}=h^{-1}b_0, \quad {\bf c}=hc_0&&\\
\mu_h(\cdot,\cdot)=\mu(\cdot,\cdot)+O(h),\quad \{\cdot,\cdot\}_h=h\{\cdot,\cdot\}+O(h^2)&&\nonumber \\
m_h(\cdot,\cdot)=m(\cdot,\cdot)+ O(h),\quad \nu_h(\cdot,\cdot,\cdot)=\nu(\cdot,\cdot,\cdot)+ O(h),&&\nonumber
\end{align}
and the following theorem is true.\\

\begin{Thm} 
\begin{enumerate}
\item The operations $\mu(\cdot,\cdot)$, $m(\cdot,\cdot)$, $\{\cdot,\cdot\}$, $\nu(\cdot,\cdot,\cdot)$ together with ${\bf b}$-operator defined on the space of the light modes of $\mathscr{V}$, satisfy the relations of the homotopy BV algebra of LZ type.\\

\item The complex of light modes of $\mathscr{V}$ is isomorphic to the full complex $(\mathcal{F}^{\bullet}, Q)$ with $L_{-1}={\rm d}$, $\frac{1}{2}L_1={\rm d}^{*}$ and the homotopy BV algebra structure is the one extending the Courant algebroid structure induced by VOA on  the complex $\mathscr{V}_0\xrightarrow{ {L}_{-1}} \mathscr{V}_1[-1]$.
\end{enumerate}
\end{Thm}

\section{Flat metric deformation, Yang-Mills equations and $BV_{\infty}^{\square}$-algebra}

\subsection{Flat metric deformation}

Let us consider the homotopy BV algebra of LZ type 
on the complex $(\mathcal{F}, Q)$ associated with Courant algebroid. 
one can introduce the following set of elements $\{f_i\}^d_{i=1}$ of degree $1$, satisfying the following conditions:
\begin{eqnarray}
Qf_i=0, \quad \mu(f_i,f_j)=0, \quad \forall i,j.
\end{eqnarray}

The immediate consequence  of $\mu(f_i,f_j)=0$ is that $\{f_i,f_j\}=0$.

Introducing the operator 
\begin{eqnarray}\label{r}
R^{\eta}=\sum_{i,j}\eta^{ij}\mu(f_i,\{f_j, \cdot\}),
\end{eqnarray}
depending on the constant matrix $\eta^{ij}$, we obtain the following proposition:

\begin{Prop}
The operator $R^{\eta}$ obeys the  properties:
\begin{equation}
{(R^{\eta})}^2=0,\quad [Q,R^{\eta}]=0.
\end{equation}
\end{Prop}

\begin{proof} The second relation is an immediate consequence of the fact that $Q$ is a derivation of both $\mu(\cdot,\cdot)$ and $\{\cdot,\cdot\}_0$. Let us prove the first one. In order to do that let us write in detail ${(R^{\eta})}^2a$ for some 
$a\in V$. 
\begin{align}
&&{(R^{\eta})}^2 a=\sum_{i,j,k,l }\eta^{ij}\mu(f_i,\{f_j, \eta^{kl}\mu(f_k,\{f_l, a\})\})=\sum_{i,j,k,l}\eta^{ij}\eta^{kl}\mu(f_i\mu(f_k,\{f_j\{f_l, a\}_0\}_0)=\nonumber\\
&&\sum_{i,j,k,l}\eta^{ij}\eta^{kl}(\mu(\mu(f_i,f_k), \{f_j\{f_l, a\}\})+(Q\nu+\nu Q)(f_i,f_k, \{f_j\{f_l, a\}\})).
\end{align}
Since $\{f_j,\{f_l, a\}\}=\{f_l,\{f_j, a\}\}$, and $n$ is antisymmetric in the first two arguments, ${(R^{\eta})}^2 a=0.$
\end{proof}

\begin{Rem}
In principle Proposition above could be considered for LZ algebras for any vertex algebra and thus true for LZ-algebras associated to vertex algebroids with CY structure.
\end{Rem}

Let us look in more detail at the structure of the operator $R^{\eta}$:

\begin{align}
\xymatrixcolsep{30pt}
\xymatrixrowsep{3pt}
\xymatrix{
V_{0}~\ar[r]^{\rm \hat{d}}~\ar[ddddr]^{-{\bf c}\Delta}& ~V'_{1}~\ar[ddddr]^{-{\bf c}\Delta}
\ar[r]^{{\rm \hat{d}}^*} & {V}''_{0}\ar[ddddr]^{-{\bf c}\Delta}& \\
&&  &&\\
& \bigoplus & \bigoplus & \\
&&  &&\\
& V'_0~\ar[r]^{\rm \hat{d}}~ & ~V''_1\ar[r]^{\hat{\rm d}^*} ~ & ~V'''_0
}
\end{align}

where 

\begin{align}
&&\Delta~\cdot =\eta^{ij}\{f_i, \{f_j, \cdot\}\}, \quad \hat{\rm d}~\cdot=(-1)^{|\cdot|}\eta^{ij}\mu(\{f_j,~ \cdot\},f_i),\nonumber\\
&&\hat{\rm d}^*~\cdot=\frac{1}{2}\eta^{ij}\tilde Q m (f_i, \{f_j,~ \cdot\})~ {\rm on} ~ V'_1, \quad \hat{\rm d}^*~\cdot=-\frac{1}{2}\eta^{ij}{\bf c}\tilde Q m (f_i, {\bf b}\{f_j,~ \cdot\})~ {\rm on} ~ V''_1, 
\end{align}

%\begin{align}
%&&\Delta~\cdot =\sum_{i,j}\eta^{ij}\{f_i, \{f_j, \cdot\}\}, \quad \hat{\rm d}~\cdot=(-1)^{|\cdot|}\sum_{i,j}\eta^{ij}\mu(\{f_j,~ \cdot\},f_i),\nonumber\\
%&&\hat{\rm d}^*~\cdot=\frac{1}{2}\sum_{i,j}\eta^{ij}\tilde Q m (f_i, \{f_j,~ \cdot\})~ {\rm on} ~ V'_1, \quad \hat{\rm d}^*~\cdot=-\frac{1}{2}\sum_{i,j}\eta^{ij}{\bf c}\tilde Q m (f_i, {\bf b}\{f_j,~ \cdot\})~ {\rm on} ~ V''_1.
%\end{align}
and for simplicity of the calculations we assumed Einstein summation convention, i.e. when an index variable appears twice in a single term, once in an upper (superscript) and once in a lower (subscript) position, 
it implies summation over all of possible values of subscript and superscript. 

If we deform  the differential 
\begin{equation}
Q\to Q+R_{\eta}
\end{equation}

we obtain that such deformation extends to the entire $C_{\infty}$-structure.

\begin{Thm}
\begin{enumerate}
\item There exist a flat metric deformation of the homotopy associative subalgebra of 
quasiclassical LZ homotopy BV algebra, i.e. there exist $\eta$-deformed multilinear maps $\mu^{\eta}$, $m^{\eta}$, $\nu^{\eta}$ on $(\mathcal{F}^{\bullet}, Q^{\eta})$, which satisfy the relations of homotopy associative algebra. Moreover, $m^{\eta}\equiv m$ and $\nu^{\eta}\equiv \nu$.\\ 
\item 
The homotopy associative algebra on $(\mathcal{F}^{\bullet}, Q^{\eta})$ genrated by $\mu^{\eta}, \nu^{\eta}$ is an $A_{\infty}$-algebra, such that all multilinear operations vanish starting from the tetralinear one. The symmetrized version gives a $C_{\infty}$-algebra.
\end{enumerate}
\end{Thm}

\begin{proof} 

Let us find an obstruction for $R^{\eta}$ to be a derivation of $\mu$:
\begin{align}
&&R^{\eta}\mu(a_1,a_2)=\mu(f_i, \mu(\{f_j,a_1\},a_2))\eta^{ij}+\mu(f_i, \mu(a_1,\{f_j,a_2\}))\eta^{ij}=\nonumber\\
&&=\mu(\mu(f_i, \{f_j,a_1\}),a_2)\eta^{ij}+\mu(\mu(f_i, a_1),\{f_j,a_2\})\eta^{ij}-\nonumber\\
&&(Q\nu+\nu Q)(f_i,\{f_j,a_1\},a_2)\eta^{ij}-(Q\nu+\nu Q)(f_i,a_1,\{f_j,a_2\})\eta^{ij}=\nonumber\\
&&\mu(\mu(f_i, \{f_j,a_1\}),a_2)+(-1)^{|a_1|}\mu(\mu(a_1,f_i),\{f_j,a_2\}_0)-\nonumber\\
&&(Q\nu+\nu Q)(f_i,\{f_j,a_1\},a_2)\eta^{ij}-(Q\nu+\nu Q)(f_i,a_1,\{f_j,a_2\})\eta^{ij}+\nonumber\\
&&(Qr+rQ)(f_i,a_1,\{f_j,a_2\})\eta^{ij}=\nonumber\\
&&\mu(\mu(f_i, \{f_j,a_1\}),a_2)\eta^{ij}+(-1)^{|a_1|}\mu(a_1, \mu(f_i,\{f_j, a_2\}))\eta^{ij}+ \nonumber\\
&&(Qr+rQ)(f_i,a_1,\{f_j,a_2\})\eta^{ij}=\nonumber\\
&&\mu(R^{\eta} a_1, a_2)+(-1)^{|a_1|}\mu(a_1, R^{\eta} a_2) -(Q{\bar \mu}^{\eta}+{\bar \mu}^{\eta}Q)(a_1,a_2),
\end{align}
where $r(a_1,a_2)=\mu(m(f_i,a_1),\{f_j,a_2\})\eta^{ij}$ and 
\begin{eqnarray}
\bar{\mu}^{\eta}(a_1,a_2)=\nu(f_i,\{f_j,a_1\},a_2)\eta^{ij}-\mu(m(f_i,a_1),\{f_j,a_2\})\eta^{ij}, 
\end{eqnarray}
and explicitly:

\begin{center}
$\bar{\mu}^{\eta}(a_1,a_2)$=
\end{center}
\begin{tabular}{|l|c|c|c|c|c|r|}
\hline
\backslashbox{$ a_2$}{$a_1$}&          $u_1$ & $A_1$ & $v_1$ &$\tilde A_1$ &$\tilde v_1$& $\tilde u_1$ \\
\hline
$u_2$ &    0    &$-\eta^{ij}\mu(m(f_i,A_1),\{f_j,u_2\})$ & 0 & 0& 0 & 0\\
\hline
$A_2$     &  0 & $\bar{\mu}^{\eta}(A_1,A_2)$ &0  & 0 &$-\eta^{ij}\mu(m( f_i,A_2),\{f_j,\tilde v_1\})$& 0 \\                    
\hline
$v_2$& 0 &  0 & 0 & 0  & 0 & 0  \\
\hline
$\tilde A_2$ & 0&  0 &0 & 0&0&0\\
\hline
$\tilde v_2$& 0 & $-\eta^{ij}\mu(m(f_i,A_1),\{f_j,\tilde v_2\})$  & 0 & 0 &0& 0\\
\hline
$\tilde u_2$    & 0 &   0& 0 & 0  &0& 0\\
\hline
\end{tabular}\\
\vspace{3mm}

\noindent Here
\begin{eqnarray}
&&\bar{\mu}^{\eta}(A_1,A_2)=\\
&&-\mu(m(f_i,A_1)\{f_j,A_2\})\eta^{ij}-\mu(m(\{f_j,A_1\},A_2),f_i)\eta^{ij}+\mu(m(f_i,A_2),\{f_j,A_1\})\eta^{ij}.\nonumber
\end{eqnarray}

Thus, since
$[R^{\eta},\mu]+[Q,\bar{\mu}^{\eta}]=0$. If one shows that $[R^{\eta}, \bar{\mu}^{\eta}]=0$, i.e.,
\begin{equation}
R^{\eta}\bar{\mu}(a_1,a_2)={\bar \mu} (R^{\eta}a_1, a_2)+(-1)^{|a_1|}{\bar{\mu}}(a_1,R^{\eta} a_2),
\end{equation} 
then $Q^{\eta}$ is a derivation of the bilinear operation $\mu^{\eta}=\mu+{\bar\mu}^{\eta}$. 
One can show it explicitly. 
%We check the most nontrivial situations, i.e. 
%when $a_1\in \mathcal{F}^1$, $a_2\in \mathcal{F}^1$, and also cases when $a_2\in \mathcal{F}^0$, $a_1\in \mathcal{F}^1$ and 
%$a_1\in \mathcal{F}^0$, $a_2\in \mathcal{F}^1$. 

Let $a_1=u\in\mathcal{F}^0$ and $a_2=A\in\mathcal{F}^1$. Then 
$\bar{\mu}^{\eta}(u,A)=0$ and $\bar{\mu}^{\eta}(u,R_0^{\eta}A)=0$. 
At the same time 
\begin{eqnarray}
&&\bar{\mu}^{\eta}(R^{\eta}u,A)=\bar{\mu}^{\eta}(\mu(f_i,\{f_j,u\}),A)\eta^{ij}=\nonumber\\
&&-\mu(m(\{f_l, \mu(f_i,\{f_j,u\}\}),A), f_r)+
\mu(m(f_i, A),\{f_j,\mu(f_r,\{f_l,u\}\}) \eta^{rl}\eta^{ij}=0.
\end{eqnarray}
Now we check the case $a_2=u\in\mathcal{F}^0$, $a_1=A\in\mathcal{F}^1$.
The bilinear operation between these two elements is nontrivial, moreover:
\begin{equation}
R^{\eta}\bar{\mu}^{\eta}(A,u)=-\eta^{ij}\eta^{kl}\mu(f_k,\{f_l, \mu(m(f_i, A),\{f_j,u\})\}).
\end{equation}
From the table above we see that ${\bar\mu}^{\eta}(R^{\eta}A, u)=0$ and 
\begin{align}
-{\bar{\mu}}^{\eta}(A,R^{\eta}u)=
\mu(m( f_k, A)\{ f_l,\mu(f_i,\{f_j,u\})\})\eta^{ij}\eta^{kl}+\mu(m(\{f_j,A\},\mu(f_k,\{f_l,u\}), f_i)\eta^{ij}\eta^{kl}.\nonumber
\end{align}
and thus $R^{\eta}$ is a derivation by using the fact that  $\{f_k, \cdot \}$ is a derivation. Let us consider the last nontrivial case, $a_1=A_1\in\mathcal{F}^1$ and $a_2=A_2\in \mathcal{F}^1$:
\begin{align}
&&R^{\eta}\bar{\mu}^{\eta}(A_1,A_2)=\mu( f_k,\{f_l,
\mu( m(f_i,A_2),\{f_j,A_1\})\})\eta^{ij}\eta^{kl}-
\nonumber\\
&&\mu(f_k,\{f_l,\mu(m(f_i,A_1),\{f_j,A_2\})\})\eta^{ij}\eta^{kl},\nonumber\\
&&\bar{\mu}^{\eta}(R^{\eta}A_1,A_2)=-\mu(m(f_i, A_2),\{f_j, \mu(f_k, \{f_l, A_1\})\})\eta^{ij}\eta^{kl}\nonumber\\
&&-\bar{\mu}^{\eta}(A_1,R^{\eta}A_2)=
\mu(m(f_i, A_1), \{f_j, \mu(f_k, \{f_l, A_2\})\})\eta^{ij}\eta^{kl}.
\end{align}
Using the Leibniz property of the bracket and aanticommutativity between elements of $V'_0$ and $V'_1$, we find that $R^{\eta}{\bar \mu}^{\eta}(A_1,A_2)={\bar \mu}^{\eta}(R^{\eta}A_1,A_2)-{\bar \mu}^{\eta}(A_1,R^{\eta}A_2)$, i.e. $R^{\eta}$ is a derivation.
We leave all the remaining cases to the reader.

The next step is to show that $\mu^{\eta}$ satisfies the homotopy commutativity relation. 
From the table for $\bar{\mu}^{\eta}$ we obtain:
\begin{eqnarray}
&&\bar{\mu}^{\eta}(a_1,a_2)-\bar{\mu}^{\eta}(a_2,a_1)=\nonumber\\
&&R^{\eta}m^{\eta}(a_1,a_2)+m^{\eta}(R^{\eta}a_1,a_2)+
(-1)^{|a_1|}m^{\eta}(a_1,R^{\eta}a_2),
\end{eqnarray}
and thus $m^{\eta}\equiv m$. Now we prove that $\mu^{\eta}$ satisfies the homotopy associativity relation and we show this in the case when all the arguments belong to $\mathcal{F}^1$, leaving to check the other combinations to the reader. Let $A_1,A_2,A_3\in \mathcal{F}^1$, then
\begin{eqnarray} 
&&\mu(\bar{\mu}^{\eta}(A_1,A_2),A_3)=
-\mu(\mu(m(f_i,A_1),\{f_j,A_2\}), A_3)\eta^{ij}-\\
&&\mu(\mu(m(\{f_j,A_1\}, A_2),f_i), A_3)\eta^{ij}+\mu(\mu(m(f_i,A_2),\{f_j,A_1\}), A_3)\eta^{ij}\nonumber\\
&&\bar{\mu}^{\eta}(\mu(A_1,A_2),A_3)=\bar{\mu}^{\eta}(\frac{1}{2}\tilde{Q}m( A_1,A_2),A_3)=-\frac{1}{2}\eta^{ij}\mu(m (f_i,A_3),\{f_j,\tilde{Q} m(A_1,A_2)\}),\nonumber\\
&&{\bar\mu}^{\eta}(A_1,\mu(A_2,A_3))={\bar\mu}^{\eta}(A_1,\frac{1}{2}\tilde{Q} m(A_2,A_3))=-\frac{1}{2}\eta^{ij}\mu(m(f_i,A_1), \{f_j,\tilde Q m(A_2,A_3)\}),\nonumber\\
&&\mu(A_1,{\bar\mu}^{\eta}(A_2,A_3))=
-\mu(A_1,\mu(m(f_i,A_2)\{f_j,A_3\}))\eta^{ij}-\nonumber\\
&&\mu(A_1,\mu(m(\{f_j,A_2\},A_3),f_i)\eta^{ij}+\mu(A_1,\mu(m(f_i,A_3),\{f_j,A_2\}))\eta^{ij}.\nonumber
\end{eqnarray}
Therefore, 
\begin{eqnarray} \label{defass}
&&\mu({\bar\mu}^{\eta}(A_1,A_2),A_3)+{\bar\mu}^{\eta}(\mu(A_1,A_2),A_3)-{\bar \mu}^{\eta}(A_1,\mu(A_2,A_3))-\mu(A_1,{\bar \mu}^{\eta}(A_2,A_3))=\nonumber\\
&&\mu(m(A_1, f_i),\{f_j,\mu(A_2,A_3)\}) -\mu(m(A_2, f_i), \{f_j,\mu(A_1, A_3)\}).
\end{eqnarray}
On the other hand, 
\begin{eqnarray} \label{defhomass}
&&R^{\eta}\nu(A_1,A_2,A_3)+\nu(R^{\eta}A_1,A_2,A_3)- \nu(A_1,R^{\eta}A_2,A_3)+\nu(A_1,A_2,R^{\eta}A_3)=\nonumber\\
&&\eta^{ij}\mu(f_i,\{f_j,\mu(m(A_1,A_3),A_2)\})-
\eta^{ij}\mu(f_i,\{f_j,\mu(m(A_2,A_3),A_1)\})-\nonumber\\
&&\eta^{ij}m(A_2, A_3)\mu(f_i,\{f_j, A_1\})+\eta^{ij}m(A_1, A_3)\mu(f_i,\{f_j, A_2\}).
\end{eqnarray}
One can see that  (\ref{defass}) and (\ref{defhomass}) are equal to each other
and therefore $\mu^{\eta}$ is associative up to homotopy provided by $\nu$. From the table we see that $\mu^{\eta}$ has no contribution to the higher associativity relation involving 
$\mu^{\eta}, \nu^{\eta}$. Therefore that fact that it forms $A_{\infty}$ algebra follows from the proof of the similar statement for $Q,\mu, \nu$ and that also leads to the fact it satisfies axioms of $C_{\infty}$-algebra as well.
\end{proof}

The Maurer-Cartan equation for such $C_{\infty}$ algebra 
is linear due to the symmetry of properties of $C_{\infty}$-operations. 
One can multiply the  $C_{\infty}$-algebra with some noncommutative associative algebra $S$ so that the resulting object is an $A_{\infty}$-algebra. 
Let us choose $S=U(\mathfrak{g})$, the universal 
enveloping algebra of some Lie algebra $\mathfrak{g}$. 
Let us consider the $A_{\infty}$-algebra on 
$(\mathcal{F}^{\bullet}\otimes U(\mathfrak{g})), Q^{\eta}$ generated by 
$Q^{\eta}$, $\mu^{\eta}$, $\nu^{\eta}$. Let consider the
 Maurer-Cartan element: $\Psi\in\mathcal{F}^1\otimes \mathfrak{g}$. We call the Maurer-Cartan equation
\begin{eqnarray}
Q^{\eta}\Psi+\mu^{\eta}(\Psi,\Psi)+\nu^{\eta}(\Psi,\Psi,\Psi)=0
\end{eqnarray}
the {\it generalized Yang-Mills equation} associated to the corresponding Courant algebroid with Calabi-Yau structure and the flat metric $\eta$.
The symmetries 
\begin{eqnarray}
\Psi\to \Psi+(Q^{\eta}u+\mu^{\eta}(\Psi,u)-\mu^{\eta}(u,\Psi)),
\end{eqnarray}
where $u\in\mathcal{F}^0\otimes\mathfrak{g}$ is infinitesimal will be referred to as {\it gauge symmetries}.

It is indeed true that the homotopy BV structure is destroyed through the deformation given by $R^{\eta}$, in particular
\begin{equation}
[Q, {\bf b}]=-\Delta=-\eta^{ij}\{f_i,\{f_j,\cdot\}\}
\end{equation}
which is not a derivation for the product and thus $Q$ is not a derivation for the bracket operation. Defining the corresponding bracket operation at usual:
\begin{equation}
\{a_1,a_2\}^{\eta}=(-1)^{|a_1|}({\bf b}\mu^{\eta}(a_1,a_2)-\mu^{\eta}({\bf b}a_1,a_2)-(-1)^{|a_1|}\mu^{\eta}(a_1,{\bf b}a_2)
\end{equation}
we find that it certainly won't satisfy Jacobi or other relations in homotopy algebra of LZ type, there will be the $\eta$-corrections, which can be described through the combination of the action of the operation 
$m$ and applications of $\{f_i~, ~\cdot~\}$ and $\mu(~f_i~, \cdot~)$ contracted with metric $\eta^{ij}$.
One called the resulting algebraic structure a {\it  $BV^{\square}_{\infty}$-algebra}.
One example of the algebra of this type was constructed in \cite{reiterer}. We will return to this reduction in the next section.

\subsection{Homotopy algebras related to Yang-Mills equations}

Here we will look at an example of Courant algebroid, where $V_0=\mathscr{F}(M)$, the space of differentiable functions on some manifold $M$, and $V_1$ is the space of sections $\Gamma(TM\oplus T^*M)[-1]$. The corresponding differential ${\rm d}$ is given by de Rham operator.
Calabi-Yau structure is given by the divergence operator ${\rm div}_{\Phi}$ which vanishes on covectors and is equal to divergence operator for some volume form $\Phi$ on $M$.

If the action of the BV operator ${\bf b}$ is given by the following diagram:

\begin{align}
\xymatrixcolsep{40pt}
\xymatrixrowsep{5pt}
\xymatrix{
& \Gamma(TM)[-1] & \Gamma(TM)[-2]\ar[l]_{-\rm{id}} & \\
& \bigoplus & \bigoplus & \\
\mathscr{F}(M) & \mathscr{F}(M)[-1]\ar[l]_{\rm{id}} & \mathscr{F}(M)[-2] &\mathscr{F}(M)[-3]\ar[l]_{-\rm{id}}\\
& \bigoplus & \bigoplus &\\
& \Gamma(T^*M)[-1] & \Gamma(T^*M)[-2]\ar[l]_{-\rm{id}} &
}
\end{align}

one can write down the action of differential $Q$:

\begin{align}\label{qbg}
\xymatrixcolsep{40pt}
\xymatrixrowsep{5pt}
\xymatrix{
& \Gamma(TM)[-1]\ar[ddr]^{\frac{1}{2}\rm{div}_{\Phi}} & \Gamma(TM)[-2]\ar[ddr]^{-\frac{1}{2} \rm{div}_{\Phi}} & \\
& \bigoplus & \bigoplus & \\
\mathscr{F}(M)\ar[ddr]_{d} & \mathscr{F}(M)[-1]\ar[r]^{\rm{id}}\ar[ddr]_{d} & \mathcal{F}(M)[-2] &\mathscr{F}(M)[-3]\\
& \bigoplus & \bigoplus &\\
& \Gamma(T^*M)[-1] & \Gamma(T^*M)[-2] &
}
\end{align}

It is useful to write down the explicit values for the 
operation $\mu$ on the complex $(\mathcal{F}^{\bullet}, Q)$:

\begin{center}
$\mu(a_1,a_2)$=\\
\vspace*{2mm}
\begin{tabular}{|l|c|c|c|c|c|r|}
\hline
 \backslashbox{$a_2$}{$a_1$}&          $u_1$ & $\mathcal{X}_1$ & $v_1$ &$\tilde {\mathcal{X}_1}$ &$\tilde v_1$& $\tilde u_1$ \\
\hline
$u_2$ &                  $u_1u_2$    &$\mathcal{X}_1 u_2$ &$v_1u_2$&
$\tilde {\mathcal{X}_1} u_2$&$\tilde v_1 u_2$ &$\tilde u_1 u_2$\\
&  & $-L_{\mathcal{X}_1}u_2$  &  &   && \\
\hline
$\mathcal{X}_2$     &  $u_1\mathcal{X}_2$ & $(\mathcal{X}_1,\mathcal{X}_2)_D+$ &$-v_1\mathcal{X}_2$  
&  $-\frac{1}{2}\langle \tilde {\mathcal{X}_1},\mathcal{X}_2\rangle$&$L_{\mathcal{X}_2}\tilde v_1$&0\\                    
& & $\frac{1}{2}\langle \mathcal{X}_1, \mathcal{X}_2\rangle$  &  &  && \\
\hline
$v_2$& $u_1\t u_2$ &  $v_2\mathcal{X}_1$ & 0 & 0  &$-\tilde v_1v_2$& 0\\
\hline
$\tilde {\mathcal{X}_2}$ & $u_1\tilde {\mathcal{X}_2}$ &$-\frac{1}{2}\langle \mathcal{X}_1,\tilde {\mathcal{X}_2}\rangle$&  
0 & 0& 0&0\\
\hline
$\tilde v_2$& $u_1\tilde u_2$ &   $L_{\mathcal{X}_1}\tilde v_2$& $-v_1\tilde v_2$ & 0  &0& 0\\
\hline
$\tilde u_2$    & $u_1\tilde u_2$ &   0& 0 & 0  &0& 0\\
\hline
\end{tabular}\\
\vspace{3mm}
\end{center}

\noindent where $\mathcal{X}_i,\tilde {\mathcal{X}_i}$ stand for pairs 
$(\mathbf{A}_i,\mathbf{B}_i), (\tilde{\mathbf{A}}_i,\tilde{\mathbf{B}}_i)$, 
such that $\mathbf{A}_i,\tilde{\mathbf{A}}_i\in \Gamma(TM)$ and 
$\mathbf{B}_i,\tilde{\mathbf{B}}_i\in \Gamma(T^*M)$.  
The expression $(\mathcal{X}_1,\mathcal{X}_2)_D$ stands for the Dorfman bracket (see e.g. \cite{courant}, \cite{roytenberg1}) defined as 
\begin{eqnarray}
((\mathbf{A}_1,\mathbf{B}_1), (\mathbf{A}_2,\mathbf{B}_2))_D=(L_{\mathbf{A}_1}\mathbf{A}_2, L_{\mathbf{A}_1}\mathbf{B}_2-i_{\mathbf{A}_2}d\mathbf{B}_1).
\end{eqnarray}
The pairing $\langle \cdot,\cdot \rangle$ between two elements 
$(\mathbf{A}_1,\mathbf{B}_1)$ and $(\mathbf{A}_2,\mathbf{B}_2)$ is defined 
as follows:
\begin{equation}
\langle (\mathbf{A}_1,\mathbf{B}_1), (\mathbf{A}_2,\mathbf{B}_2)\rangle=\langle \mathbf{A}_1,\mathbf{B}_2\rangle+\langle \mathbf{A}_2,\mathbf{B}_1\rangle,
\end{equation}
where $\langle \mathbf{A}_i,\mathbf{B}_j\rangle$ is the standard pairing between vectors and covectors.
The operation $L_\mathcal{X} u$ for $\mathcal{X}=(\mathbf{A},\mathbf{B})$  
denotes the Lie derivative with respect to the vector field $\mathbf{A}$. 

From now on we assume $M=\mathbb{R}^D$ and $\eta^{ij}$ is a symmetric tensor. We assume that $\Phi$ is a canonical volume form for $\mathbb{R}^D$

\begin{Prop} Let $f_i=\partial_i$. Then $R^{\eta}=\sum_{\i,j}\eta^{ij}\mu(f_i,\{f_j,\cdot\})$ acts on the complex $(\mathcal{F}^{\bullet}, Q)$ as follows:
\begin{align}
\xymatrixcolsep{40pt}
\xymatrixrowsep{5pt}
\xymatrix{
& \Gamma(TM)[-1]\ar[r]^{\Delta} & \Gamma(TM)[-2] & \\
& \bigoplus & \bigoplus & \\
\mathscr{F}(M)\ar[uur]^{\hat{\rm{d}}}\ar[r]^{-\Delta} & \mathscr{F}(M)[-1]\ar[uur]^{\hat{\rm{d}}} & \mathscr{F}(M)[-2]\ar[r]^{\Delta} & \mathscr{F}(M)[-3]\\
& \bigoplus & \bigoplus &\\
& \Gamma(T^*M)[-1]\ar[uur]^{\frac{1}{2}\widehat{\rm{div}}} \ar[r]^{\Delta} & \Gamma(T^*M)[-2]\ar[uur]^{-\frac{1}{2}\widehat{\rm{div}}}&
}
\end{align}
where 
\begin{align}
&&\Delta=\sum_{i,j}\eta^{ij}\partial_i\partial_j,\nonumber\\
&&\widehat{\rm{div}}\mathbf{B}=\sum_{i,j}\eta^{ij}\partial_iB_j,\quad  \mathbf{B}\in \Gamma(T^*M);\quad  (\hat{\rm{d}}u)^j=\sum_{i}\eta^{ij}\partial_i u, \quad u\in \mathscr{F}(M).
\end{align}
\end{Prop}

Let us now assume that the matrix $\eta^{\{ij\}}$ is invertible, such that $\eta_{\{ij\}}$ is the inverse matrix and let us introduce the following direct sum of three complexes:
\begin{align}\label{omegacomp}
\xymatrixcolsep{30pt}
\xymatrixrowsep{-5pt}
\xymatrix{
0\ar[r]&\Omega^0(M) \ar[r]^{\rm{d}} & \Omega^1(M)[-1] \ar[r]^{*{\rm d}*{\rm d}}  &\Omega^1(M)[-2] \ar[r]^{*{\rm d}*} & \Omega^0(M)[-3]\ar[r]&0\\
 &\quad &     & \quad    &&\\
 && \bigoplus & \bigoplus     &&\\
 &&           & \quad    &&\\
 &0\ar[r]& \Omega^1(M)[-1] \ar[r]^{\Delta}   & \Omega^1(M)[-2]\ar[r]&0\\
&& \bigoplus & \bigoplus     &&\\
&0\ar[r]& \Omega^0(M)[-1] \ar[r]^{\rm id}   & \Omega^0(M)[-2]\ar[r]&0
}
\end{align}
Here $\Omega^0(M)\equiv \mathcal{F}(M)$ and $\Omega^1(M)\equiv \Gamma(T^*M)$ and the Hodge star operator $*$ is constructed 
via the metric corresponding to the invertible and symmetric matrix $\eta_{\{ij\}}$. Let us denote the three 
subcomplexes above as $(\mathcal{G}_i^{\bullet},{Q'}^{\eta})$ $(i=1,2,3)$, where $i$ runs from top to bottom.  

Consider the following maps from $(\mathcal{G}_i^{\bullet},{Q'}^{\eta})$ to $(\mathcal{F}^{\bullet}, Q)$:
\begin{eqnarray}
\mathcal{G}_1^0\xrightarrow{id} \mathcal{F}^0,\quad &\mathcal{G}_1^1\xrightarrow{f_1}\mathcal{F}^1, \quad  \mathcal{G}_1^1\xrightarrow{g_1}\mathcal{F}^1,&\quad \mathcal{G}_1^3\xrightarrow{id} \mathcal{F}^3\nonumber\\
&\mathcal{G}_2^1\xrightarrow{f_2}\mathcal{F}^1, \quad  \mathcal{G}_2^1\xrightarrow{g_2}\mathcal{F}^1,&\nonumber\\
&\mathcal{G}_3^1\xrightarrow{f_3}\mathcal{F}^1,  \quad \mathcal{G}_3^1\xrightarrow{g_3}\mathcal{F}^1, &
\end{eqnarray}
where $f_i, g_i$ are explicitly given by:
\begin{eqnarray}
f_1(\mathbf{B})=\mathbf{B}+\mathbf{B}^*-\hat{\rm div}\mathbf{B},&&
\quad  g_1(\tilde{\mathbf{B}})=\tilde{\mathbf{B}}+
\tilde{\mathbf{B}}^*,\nonumber\\
f_2(\mathbf{B})=\mathbf{B}-\mathbf{B}^*, &&\quad g_2(\tilde{\mathbf{B}})=\tilde {\mathbf{B}}-\tilde{\mathbf{B}}^*,\nonumber\\
f_3\equiv id,&& \quad g_3(\tilde{v})=\tilde{v}- ({\rm d}\tilde{v}+\hat {\rm d}\tilde{v}), 
\end{eqnarray}
so that $\mathbf{B}^*\in \Gamma(TM)$ such that ${\mathbf{B}^*}^i=\eta^{ij}B_j$. This yields the following proposition \cite{zeitcour},\cite{cftym}.

\begin{Prop} \label{ymcomp}\hfill 
\begin{enumerate}
\item The complex $(\mathcal{F}^{\bullet}, Q^{\eta})$ is isomorphic  to $\oplus^3_{i=1}(\mathcal{G}_i^{\bullet}, {Q}^{\eta})$.\\
 
\item There is a structure of $C_{\infty}$-subalgebra on subcomplexes $(\mathcal{G}_1^{\bullet}, {Q}^\eta)$ and 
$(\mathcal{G}_1^{\bullet}\oplus \mathcal{G}_3^{\bullet}, {Q}^\eta)$.\\

\item If we assume that the components of the functions and 1-forms in the complex $(\mathcal{G}_1^{\bullet}, {Q}^\eta)$ are $L_2$-integrable, then the pairing inherited from $(\mathcal{F}^{\bullet}, Q)$ with respect to the satndard volume form on $\mathbb{R}^d$ gives the structure of the cyclic $C_{\infty}$-algebra.
\end{enumerate}
\end{Prop}

One can see that the cohomology of the complex $(\mathcal{F},Q^{\eta})$ in degree 1 is equivalent to the solutions of Maxwell equations and scalar field equations
\begin{eqnarray}
*{\rm d}*{\rm d}\mathbf{A}=0, \quad \Delta\Phi=0,
\end{eqnarray} 
modulo gauge transformations $\mathbf{A}\to \mathbf{A}+{\rm d} u$.
 
Let us now describe the generalized Yang-Mills equations in this case. The answer is given in the Theorem below proved in \cite{zeitcour}.\\

\begin{Thm} \label{ymthm}
\begin{enumerate}
\item The generalized Yang-Mills equations assocaited to Courant algebroid in $\mathbb{R}^d$ with CY structure given by canonical form and the flat metric $\eta^{ij}$ is equivalent to the following system of equations:
\begin{align}\label{aphi}
&&\sum_{i,j}\eta^{ij}[\nabla_{i},[\nabla_{j},\nabla_{k}]]=\sum_{i,j}\eta^{ij}[[\nabla_{k},\Phi_{i}],\Phi_{j}],\nonumber\\
&&\sum_{i,j}\eta^{ij}[\nabla_{i},[\nabla_{j},\Phi_{k}]]=\sum_{i,j}\eta^{ij}[\Phi_{i},[\Phi_{j},\Phi_{k}]],
\end{align}
where $\Phi_{i}={B}_{i}- \sum_jA^{j}\eta_{ij}$, 
$\mathcal{A}_{i}=B_{i}+\sum_jA^{j}\eta_{ij}$ and $\nabla_{i}=\partial_{i}+\mathcal{A}_{i}$.  Here $B_i$ are the components of $\mathbf{B}\in \Gamma(T^*M)\otimes \mathfrak{g}$ and $A_i$ are the components of $\mathbf{A}\in \Gamma(TM)\otimes \mathfrak{g}$, constituting the components of the Maurer-Cartan element.\\
\item 
The gauge symmetries correspond to the following transformation of fields:
\begin{eqnarray}\label{aphisym}
\mathcal{A}_{i}\rightarrow \mathcal{A}_{i}+\partial_i u+[\mathcal{A}_{i},u], \quad 
\Phi_{i}\rightarrow \Phi_{i}+[\Phi_{i},u].
\end{eqnarray}
with $u\in \Omega^0(M)$ is infinitesimal.\\

\item The Maurer-Cartan equations and their symmetries for the subcomplexes  $(\mathcal{G}_1^{\bullet}, {Q'}^\eta)$ and 
$(\mathcal{G}_1^{\bullet}\oplus \mathcal{G}_3^{\bullet}, {Q'}^\eta)$ are equivalent to standard Yang-Mills equations and their gauge symmetries:
\begin{equation}
\sum_{i,j}\eta^{ij}[\nabla_{i},[\nabla_{j},\nabla_{k}]]=0,\quad 
\mathcal{A}_{i}\rightarrow \mathcal{A}_{i}+\partial_i u+[\mathcal{A}_{i},u].
\end{equation}
\end{enumerate}
\end{Thm}

In the end, for completeness we write down explicitly the $C_{\infty}$-operations on the Yang-Mills complex, $(\mathcal{G}_1^{\bullet}, {Q}^\eta)$, which we rewrite as follows for simplicity, using the hodge star applied to $\mathcal{G}_1^{2}$, $\mathcal{G}_1^{3}$:
\begin{eqnarray}\label{maxwell}
0\xrightarrow{}\Omega^{0}(M)\xrightarrow{{\rm d}}\Omega^{1}(M)[-1]
\xrightarrow{{\rm d}*{\rm d}}\Omega^{D-1}(M)[-2]\xrightarrow{{\rm d}}\Omega^{D}(M)[-3]\xrightarrow{} 0,
\end{eqnarray}
and the bilinear operation $\mu^{sym}(~\cdot~,~\cdot~)$ is represented by the following table:\\ 

$\mu^{sym}(a_1,a_2)$=
\begin{tabular}{|l|c|c|c|r|}
\hline
\backslashbox{$ a_2$}{$a_1$}&\makebox{$v$} & \makebox{$\bf A$} & \makebox{$\bf V$} & \makebox{$a$} \\
\hline
$w$ & $vw$& ${\bf A} w$& ${\bf V} w$&$a w$\\
\hline
${\bf  B}$ &$v{\bf B}$  & ${\bf A}\wedge *{\rm d}{\bf B}-{\bf B}\wedge*\rm{d}{\bf A}+{\rm d} *({\bf A}\wedge {\bf B})$   &   ${\bf B}\wedge{\bf V}$&0\\
\hline
${\bf W}$ & $v{\bf W}$ &  $ {\bf A}\wedge {\bf W}$ &    0&0\\
\hline
$b$ & $vb$ &   0 &    0&0\\
\hline
\end{tabular}
\\
\vspace*{3mm}

where $v,w\in\mathcal{G}_1^0$; 
${\bf A},{\bf B}\in \mathcal{G}^0_1$; ${\bf V},{\bf W}\in\mathcal{G}_2^1$; $a,b\in\mathcal{G}_1^3$. 

The trilinear operation $\nu^{sym}(\ \cdot, \ \cdot, \ \cdot)$ is nonzero only when all the arguments belong to ${\mathcal{G}}^1_1$. For ${\bf A}$, 
${\bf B}$, ${\bf C}$ $\in \mathcal{G}^1_1$ we have: 
\begin{eqnarray}
\nu^{sym}({\bf A},{\bf B} ,{\bf C} )={\bf A}\wedge *({\bf B} \wedge {\bf C})-{\bf C}\wedge*({\bf A}\wedge {\bf B}).
\end{eqnarray}

\begin{Rem}
We note that the $C_{\infty}$-operations on the deformed complex can be easily generalized to any Riemannian manifold $M$, however in this case the deformation in general cannot be represented by the formula $(\ref{r})$. We will discuss it further in the next subsection. 
\end{Rem}

\subsection{$\beta$-$\gamma$ systems and the physical interpretation} 

Let us consider the Heisenberg algebra of the following kind:
\begin{eqnarray}\label{heis}
[X^{i}_n,p_{m,j}]=h\delta_{n,-m}\delta^i_j.
\end{eqnarray}
The construction of the corresponding VOA space $F_{p,X}[h^{-1},h]$ is as follows.   
Let us consider again the Fock space for the Heisenberg algebra (\ref{heis}):  
\begin{align}
&&\tilde F_{p,X}=\{p_{-n_1, j_1}....p_{-n_k, j_k}X^{i_1}_{-m_1}...X^{i_l}_{-m_l}|0\rangle, n_1,..., n_k>0, m_1, ..., m_l>0; \\
&&j_1, \dots, j_k, i_1, \dots, i_l=1, \dots, D~|~p_{j,n}|0\rangle=0, n\ge 0, X^i_n |0\rangle=0, n>0\}.\nonumber
\end{align}
In other words, $\tilde F_{X,p}=
\mathbb{K}[p_{j,-n}, X^i_{-m}]_{n>0, m\ge 0; i,j=1, \dots, D}$. The space  $\tilde F_{p,X}[h^{-1},h]$ already carries an algebraic structure of a VOA, but one can update it further, namely 
\begin{equation}
F_{p,X}=\tilde F_{p,X}\otimes_{A} \hat{A},
\end{equation}
where $A=\mathbb{K}[X^j_0]_{j=1, \dots, D}$ and $\hat{A}=\mathbb{K}[[X^j_0]]_{j=1, \dots, D}$ or any other function field on 
$M=\mathbb{R}^D$, containing 
$\mathbb{K}[X^j_0]_{j=1, \dots, D}$.

 One can prove that $F_{p,X}[h^{-1},h]$ is a VOA with a formal parameter, such that the commutation relations (\ref{heis}) can be expressed via the operator product expansions:
\begin{eqnarray}
X^{i}(z)p_{j}(w)\sim\frac{h\delta_{i,j}}{z-w}, \quad X^{i}(z)X^{j}(w)\sim 0, \quad p_{i}(z)p_{j}(w)\sim 0,
\end{eqnarray}
 where $p_j(z)=\sum_n p_{i,n}z^{-n-1}$, $X^i(z)=\sum_n X^i_nz^{-n-1}$ corresponding to the states of  conformal dimensions $1$ and $0$ correspondingly, and the Virasoro element is given by the formula
\begin{eqnarray}
L(z)=-\frac{1}{h}\sum_ip_i\partial X^i.
\end{eqnarray}
The operators from the VOA $F_{p,X}$ of conformal dimensions $0$ and $1$ are:
\begin{equation}
u(z)=u(X)(z),\quad  \mathbf{A}(z)=\sum_i:A^{i}(X)(z)p_{i}(z):,\quad  \mathbf{B}(z)=\sum_jB_{j}(X)(z)\partial X^{j}(z),
\end{equation}
where $u$, ${\bf A}$, ${\bf B}$ can be identified with sections of $\mathcal{O}_M$, $TM$, $T^*M$ correspondingly and lead to the vertex algebroid and the Courant algebroid on $\Gamma(TM\oplus T^*M)$ as in Section 4. 

From a physics perspective the action describing the vertex operator algebra $F_{p,X}$ is given by:
\begin{eqnarray}
S_{p,X}=\frac{1}{2\pi ih}\int_{\Sigma} \langle p\wedge \bar{\partial} X\rangle,
\end{eqnarray} 
where $X:\Sigma\rightarrow M$, $p\in X^*(\Omega^{1}(M))\otimes \Omega^{(1,0)}(\Sigma)$, where $\Sigma=\mathbb{C}$, $M=\mathbb{C}^D$. 
Then $X^i(z)$, $p_j(z)$ could be interpreted as components of the corresponding sections of vector bundles (see, e.g., \cite{wittenchiral},\cite{nekrasov}). 
One can obtain an open version of this action:
\begin{eqnarray}
S_{p,X}=\frac{1}{2\pi ih}\int_{H^+} (\langle p\wedge \bar{\partial} X\rangle-\langle {\bar p}\wedge {\partial} \bar{X}\rangle),
\end{eqnarray} 
where $H^+$ is the upper half-plane,  $p\in {\bar X}^*(\Omega^{1}(M))\otimes \Omega^{(1,0)}(\Sigma)$
with the boundary conditions $p_i(z)|_{H^+}=\bar{p}_i(z)|_{H^+}$,  $X^i(z)|_{H^+}=\bar{X}^i(z)|_{H^+}$. Given these boundary conditions the corresponding physical theory is still described by the vertex algebra $F_{X,p}$.

One can further deform the action $S_{p,X}$:
\begin{eqnarray}
S^{\eta}_{p,X}=S_{p,X}-\int_{H^+}\phi^{(2)}, \quad \phi^{(2)}=\frac{1}{2\pi i h}\eta^{ij}p_i\wedge \bar p_j.
\end{eqnarray}
The regularized deformation of BRST operator for the operator $\phi^{(2)}$ is given by the following operator acting on $F_{p,X}\otimes \Lambda$:
\begin{align}
Q+P_0\int_{C_{\epsilon,0}} \phi^{(1)}=Q+ R^{\eta}, \quad \phi^{(1)}= \frac{1}{2}dzc(\bar{z})\eta^{ij}p_i(z){p}_j(\bar{z})-\frac{1}{2}d{z}c({z})\eta^{ij}p_i(z){p}_j(\bar{z}),&&\nonumber\\
R^{\eta}=\eta^{ij}\mu(cp_i,\{cp_j,\cdot \}\},\quad [Q,\phi^{(2)}]=d\phi^{(1)}.
\end{align}
In the formula above coefficients of $\phi^{(1)}$, $\phi^{(2)}$ are already given by the  corresponding vertex operators $p_i(z)$. The regularized integral is understood as integral over the circle of radius $\epsilon$ and  the projection $P_0$ is given by removing all $\epsilon$-dependent terms. 

With the deformation $Q^{\eta}$ the space of $VOA$ $F_{p,X}$ becomes the Hilbert space of open string tensored with Heisenberg VOA as described in \cite{azbg}. On the level of BRST subcomplexes  of light modes, those correspond to $(\mathcal{G}_1^{\bullet}\oplus \mathcal{G}_3^{\bullet},{Q}^{\eta})$ and $(\mathcal{G}_2^{\bullet}\oplus \mathcal{G}_3^{\bullet},{Q}^{\eta})$ correspondingly. We will return to the discussion of open string in subsection 6.2.

We finally remark, that this construction of flat metric deformation as can in general be applied to any other open conformal field theory based on vertex algebroid using perturbation of the form: 
\begin{equation}
\int\phi^{(2)}=\int dz\wedge d\bar z\eta^{ij}\Big([b_{-1}f_i(z)] [b_{-1}f_j(\bar z)]\Big)
\end{equation}
where $f_i\in V\otimes \Lambda$, $i=1, \dots, D$.

\section{Homotopy algebras for Double Field Theory and Gravity}

\subsection{Holomorphic double copy and the first order sigma-model}
Let us consider the homotopy BV-LZ algebra corresponding to 
holomorphic Courant $\mathcal{O}_M$-algebroid, soo that $M$ is a complex manifold with $V_0=\mathcal{O}_M$ is a sheaf of holomorphic functions and $V_1=\mathcal{O}(T^{(1,0)}(M)\oplus{T^*}^{(1,0)}(M))$ is an $\mathcal{O}_M$-module associated with holomorphic vector fields and 1-forms. In this section we will refer to the corresponding complex as $(\mathcal{F}^{\bullet}, Q)$ while the complex conjugate complex associated with antiholomorphic sections  as $(\bar{\mathcal{F}}^{\bullet}, Q)$.

Assume that locally there exists a holomorphic volume form $\omega$ which gives locally a Calabi-Yau structure (we do not require it globally).  
At the same time we can consider $V_0=\bar{\mathcal{O}}_M$ is a sheaf of atiholomorphic functions and $V_1=\mathcal{O}(T^{(0,1)}(M)\oplus{T^*}^{(0,1)}(M))$ and Calabi-Yau structure given by $\bar{\omega}$ thus producing the conjugated homotopy BV-LZ algebra associated with the complex $(\bar{\mathcal{F}}^{\bullet}, \bar{Q})$. 

One consider the tensor product of these two BV-LZ algebras.  
Naturally there is a homotopy Gerstenhaber algebra on the completed tensor product. One can choose the bracket structure using operator  
$\mathbf{b_-}=\mathbf{b}-\bar{\mathbf{b}}$, which is physically motivated (see below):
\begin{equation}
(-1)^{|a_1|}\{a_1,a_2\}_-=
\mathbf{b_-}\mu(a_1,a_2)-\mu(\mathbf{b_-}a_1,a_2)-(-1)^{|a_1|}\mu(a_1\mathbf{b_-}a_2)\nonumber,
\end{equation}
Together with the differential $\mathcal{Q}=Q+\bar{Q}$ induced product there is a structure of a homotopy Gerstenhaber algebra on the tensor  product:
\begin{equation}
({\bf F}^{\bullet},\mathcal{Q})=(\mathcal{F}^{\bullet}, Q){\hat\otimes}(\bar{\mathcal{F}}^{\bullet}, \bar{Q}), 
\end{equation}
where we consider a completion of the tensor product. 
Let us start from the simplest case when instead of the homotopy Gerstenhaber algebras we have the product of Gerstenhaber algebras associated to vector fields.

\subsubsection{Doubling of Gerstanhaber subalgebras}
Consider a subcomplex of $(\mathcal{F}_{sm}^{\bullet}, Q)\subset
(\mathcal{F}^{\bullet}, Q)$ corresponding to vector fields only:
\begin{align}
\xymatrixcolsep{30pt}
\xymatrixrowsep{3pt}
\xymatrix{
\mathbb{C}\ar[r]^{0\qquad}& \mathcal{O}(T^{(1,0)}M)[-1]
\ar[r]^{\frac{1}{2}{\rm div}_{\omega}} & \mathcal{O}_M[-2]& \\
&&  &&\\
& \bigoplus & \bigoplus & \\
&&  &&\\
& \mathbb{C}[-1]\ar[uuuur]^{i}\ar[r]^{0} & \mathcal{O}(T^{(1,0)}M)\ar[r]^{\quad\frac{1}{2}{\rm div}_{\omega}}[-2]  & \mathcal{O}_M[-3]
}\nonumber
\end{align}

Then the following proposition holds.
\begin{Prop}
The homotopy BV-LZ algebra on $(\mathcal{F}_{sm}^{\bullet}, Q)$ reduces to BV algebra and thus to Gerstenhaber algebra.
\end{Prop}

Now let us consider a completed tensor product: 

\begin{equation}
({\bf F}_{sm}^{\bullet},\mathcal{Q})=(\mathcal{F}_{sm}^{\bullet}, Q){\hat\otimes}(\bar{\mathcal{F}}_{sm}^{\bullet}, \bar{Q}), 
\end{equation}

For the elements $\Psi\in {\bf F}_{sm}^{2}$, $\Lambda\in {\bf F}_{sm}^{1}$ one can write a Maurer Cartan equation with its symmetries:
\begin{equation}\label{mclsm}
\mathcal{Q}\Psi+\frac{1}{2}\{\Psi, \Psi\}_-=0, \quad 
\Psi\rightarrow \Psi+\mathcal{Q}\Lambda+\{\Psi,\Lambda\}_-
\end{equation}

We will further restrict the space of Maurer-Cartan elements and symmetries. The following complex 
\begin{equation}
({\bf F}_{sm, {\mathbf{b}^-}}^{\bullet},\mathcal{Q})=
({\bf F}_{sm}^{\bullet},\mathcal{Q})|_{\mathbf{b}^-=0}
\end{equation}
is closed under the bracket $\{\cdot~, ~\cdot ~\}_-$. The corresponding space Maurer-Cartan elements, namely ${\bf F}_{sm, {\mathbf{b}^-}}^{2}$ has the following form: 
\begin{equation}
\Gamma(T^{(1,0)}(M)\otimes T^{(0,1)}(M))\oplus \mathcal{O}(T^{(0,1)}(M))\oplus \mathcal{O}(T^{(1,0)}(M))\oplus
\mathcal{O}_M\oplus \bar{\mathcal{O}}_M 
\end{equation}
At the same time the space of symmetries, namely ${\bf F}_{sm, {\mathbf{b}^-}}^{1}$ is given by the following direct sum:
\begin{equation}
\mathcal{O}(T^{(0,1)}(M))\oplus \mathcal{O}(T^{(1,0)}(M))\oplus \mathbb{C}
\end{equation}

Let us denote the components of Maurer-Cartan elements as:
\begin{equation}
\Psi=(g, \bar v, v, \phi', \bar \phi'), 
\end{equation} 
where 
\begin{align}
&g\in \Gamma(T^{(1,0)}(M)\otimes T^{(0,1)}(M)),&\\
 &v \in\mathcal{O}(T^{(1,0)}(M)),\quad  \bar{v} \in \mathcal{O}(T^{(0,1)}(M)),\quad  
\phi'\in \mathcal{O}(M),\quad  \bar{\phi}'\in \bar{\mathcal{O}}(M),\nonumber&
\end{align}
and introducing the notation:
\begin{equation}
\log{\Omega}=-2\Phi_0=\log{\omega}+\log{\bar{\omega}}-2(\phi'+\bar \phi')
\end{equation}
where $\Omega$ stands for the volume form on $M$, we obtain the following theorem.

\begin{Thm}\label{EinstG}
\begin{enumerate}

\item The Maurer-Cartan equation (\ref{mclsm}), where 
$$\Psi=(g, \bar v, v, \phi', \bar \phi')\in \mathbf{F}_{sm, {\bf b}^-}^{2}$$ 
is equivalent to the following set of equations on components:\\

\begin{itemize}
\item $div_{\Omega}g\in \mathcal{O}(T^{(1,0)} M)\oplus \bar{\mathcal{O}}(T^{(0,1)} M)$.\\

\item  Bivector field $g\in \Gamma(T^{(1,0)}M\otimes T^{(0,1)}M)$ obeys the following equation: 
\begin{equation}\label{bil}
[[g,g]]+\mathcal{L}_{div_{\Omega}(g)}g=0, 
\end{equation} 
where $\mathcal{L}_{div_{\Omega}(g)}$ is a Lie derivative with respect to 
the corresponding vector field and the double bracket is given by the following expression:
\begin{equation}
[[g,h]]^{k\bar l}\equiv
(g^{i\bar j}\partial_i\partial_{\bar j}h^{k\bar l}+h^{i\bar j}\partial_i\partial_{\bar j}g^{k\bar l}-\partial_ig^{k\bar j}\partial_{\bar j}h^{i\bar l}-
\partial_ih^{k\bar j}\partial_{\bar j}g^{i\bar l})
\end{equation}\\

\item $div_{\Omega}div_{{\Omega}}(g)=0$.\\

\end{itemize}

\item The symmetries correspond to infinitesimal holomorphic transformations of bivector field $g$ and constant shifts of $\Omega$. \\

\item The equations and their symmetries from (1) are equivalent to the Einstein equations with the $B$-field and dilaton if the biveector $g$ is nondegenerate: 
\begin{align}\label{Einsteq}
&&R_{\mu\nu}={\frac{1} {4}} H_{\mu}^{\lambda\rho}H_{\nu\lambda\rho}-2\nabla_{\mu}
\nabla_{\nu}\Phi,\\
&&\nabla^{\mu}H_{\mu\nu\rho}-2(\nabla^{\lambda}\Phi)H_{\lambda\nu\rho}=0,
\nonumber\\
&&4(\nabla_{\mu}\Phi)^2-4\nabla_{\mu}\nabla^{\mu}\Phi+
R+{\frac{1}{12}} H_{\mu\nu\rho}H^{\mu\nu\rho}=0,\nonumber
\end{align}
where 3-form $H=dB$, and $R_{\mu\nu}, R$ are Ricci and scalar curvature correspondingly, with the following constraints imposed:
\begin{equation}
G_{i\bar{k}}=g_{i\bar{k}}, \quad B_{i\bar{k}}=-g_{i\bar{k}}, \quad G_{ik}=G_{\bar i \bar k}=B_{ik}=B_{\bar i \bar k}=0\quad \Phi=\log\sqrt{g}+\Phi_0,
\end{equation}
while the holomorphic symmetries from (2) correspond to the holomorphic coordinate transformations of metric and $B$-field tensor as well as constant shift of the dilaton field $\Phi$. 
\end{enumerate}

\end{Thm}

\subsubsection{Physical interpretation}
It was known since 1980s \cite{friedan}, \cite{eeq1}, \cite{eeq2}, \cite{fts1}, \cite{fts2} that the Einstein equations emerge in sigma model with the action:
\begin{eqnarray}
S_{so}=\frac{1}{4\pi h}\int_{\Sigma}(\langle dX\wedge *dX\rangle_G+X^*B)+\int_{\Sigma}\Phi(X)R^{(2)}(\gamma){\rm vol}_{\Sigma}
\end{eqnarray}
as the conformal invariance conditions on the quantum level at one loop. Here $X: \Sigma\to M$, where $\Sigma$ is a Riemann surface (worldsheet) and  $M$ is a Riemannian manifold (target space) with metruc $G$, 2-form $B$ and $\Phi$ is known as a dilaton field and it is just a function on $M$. String Field Theory suggest that these conformal invariance conditions appear as Maurer-Cartan equations for a certain $L_{\infty}$-algebra \cite{sen}, \cite{zwiebach}.

There is another approach to sigma models by requiring $M$ to be complex and introducing the first order action:
$$
S^{free}_{fo}=\frac{1}{2\pi i h}\int_\Sigma (\langle p\wedge\bar{\partial} X\rangle-
\langle \bar{p}\wedge{\partial} \bar{X}\rangle)+\int_{\Sigma} R^{(2)}(\gamma)\Phi_0(X), 
$$
where $$p\in X^*(\Omega^{(1,0)}(M))\otimes \Omega^{(1,0)}(\Sigma)
, \quad \bar{p}\in X^*(\Omega^{(0,1)}(M))\otimes \Omega^{(0,1)}(\Sigma).$$
where $-2\Phi_0=\log(\omega+\bar{\omega})$, where $\omega$ is a holomorphic volume form. 
Mathematically, the quantum field theory for this action is described by the tensor product of sheaves of vertex operator algebras corresponding to the vertex $\mathcal{O}_M$-algebroids on $\mathcal{O}(T^{(1,0)}M\oplus T^{(0,1)}M)$. 

The corresponding vertex algebra is given by a family of $\beta-\gamma$ systems, corresponding to the components of sections of $p, \bar{p}$ and coordinate maps $X, \bar{X}$: 
\begin{equation}
X^i(z)p_j(w)\sim \frac{h}{z-w}, \quad { X}^{\bar i}(\bar{z}){\bar p}_{\bar j}({\bar w})\sim \frac{h}{\bar{z}-\bar{w}}.
\end{equation}
To get to the sigma-model, one can deform the action in following way:
\begin{equation}\label{fosimp}
S_{fo}=S^{free}_{fo}
-\frac{1}{2\pi i h}\int_{\Sigma}\langle g, p\wedge \bar{p} \rangle,
\end{equation}
where $g\in \Gamma(T^{(1,0)}M\otimes T^{(0,1)}M)$.

From the physics perspective one is interested in the study of the following path integral:
\begin{equation}
\int [dp][d\bar p][dX][d\bar X]e^{-S_{fo}[X, \bar{X}, p, \bar{p}]}.
\end{equation}
One can integrate $p$-fields and arrive to the original action of the sigma-model: 
\begin{align}
&&\int [dX][d\bar X]e^{\frac{-1}{4\pi h}\int \langle dX\wedge *dX\rangle_G+X^*B+\int R^{(2)}(\gamma)(\Phi_0(X)+\log\sqrt{g})},\nonumber\\
&& {\rm where}\quad G_{i\bar{k}}=g_{i\bar{k}}, \quad B_{i\bar{k}}=-g_{i\bar{k}}, \quad G_{ik}=G_{\bar i \bar k}=B_{ik}=B_{\bar i \bar k}=0.
\end{align}
Notice that it exactly corresponds to the conditions on the Einstein equations in the Theorem \ref{EinstG}. This suggests that the corresponding $L_{\infty}$-algebra is produced by the homotopy BV-LZ algebra on the tensor product. In the next subsection we describe the generalization of the action (\ref{fosimp}) which leads to the full sigma model and its symmetries.

\subsubsection{Full action, Beltrami-Courant differential and the homotopy algebra}
The generalization of the action (\ref{fosimp}) which is classically conformally invariant involves considering the element of $\Gamma(\mathcal{E}\otimes\bar{\mathcal{E}})$, where $\mathcal{E}=T^{(1,0)}M\oplus {T^*}^{(1,0)}M, \quad \bar{\mathcal{E}}=T^{(0,1)}M\oplus {T^*}^{(0,1)}M$.
Namely, we introduce what we call {\it Beltrami-Courant differential} 
\begin{equation*}\label{mmat}
\mathbb{M}=\begin{pmatrix} g & \mu \\ 
\bar{\mu} & b \end{pmatrix}
\in \Gamma(\mathcal{E}\otimes\bar{\mathcal{E}}),
\end{equation*}
with the following components:
\begin{align}
g\in \Gamma(T^{(1,0)}M\otimes {T}^{(0,1)}M), \quad b\in \Gamma({T^*}^{(1,0)}M\otimes {T^*}^{(0,1)}M)&&\\
\mu\in \Gamma(T^{(1,0)}M\otimes {T^*}^{(0,1)}M),\quad \bar{\mu}\in \Gamma({T^*}^{(1,0)}M\otimes {T}^{(0,1)}M).&&\nonumber
\end{align}
the corresponding full first order action is:
\begin{align}
S^{full}_{fo}=\frac{1}{2\pi i h}\int_\Sigma \Big(\langle p\wedge\bar{\partial} {X}\rangle-
\langle \bar{p}\wedge{\partial} \bar X\rangle-\langle \bar{\rm v}\wedge\mathbb{M}{\rm v}\rangle\Big)+\int_{\Sigma} R^{(2)}(\gamma)\Phi_0(X),
\end{align}
where ${\rm v}=(p,\partial X)$, $\bar {\rm v}=(\bar p,\bar{\partial}\bar{X})$.

Integrating over $p, \bar{p}$ as in previous subsection we obtain:

\begin{align}
&\int [dp][d\bar p][dX][d\bar X]e^{-S^{full}_{fo}[X, \bar{X}, p, \bar{p}]}=\\
&\int [dX][d\bar X]e^{\frac{-1}{4\pi h}\int \langle dX\wedge *dX\rangle_G+X^*B+\int R^{(2)}(\gamma)(\Phi_0(X)+\log\sqrt{g})},&&\nonumber
\end{align}
where
\begin{align}\label{twistor}
G_{s\bar{k}}=g_{\bar{i}j}
\bar{\mu}^{\bar{i}}_s\mu^{j}_{\bar{k}}+g_{s\bar{k}}-
b_{s\bar{k}}, \quad
B_{s\bar{k}}=g_{\bar{i}j}\bar{\mu}^{\bar{i}}_s\mu^{j}_{\bar{k}}-g_{s\bar{k}}-
b_{s\bar{k}}&&\\
G_{si}=-g_{i\bar{j}}\bar{\mu}^{\bar{j}}_s-g_{s\bar{j}}\bar{\mu}^{\bar{j}}_i
, \quad
G_{\bar{s}\bar{i}}=-g_{\bar{s}j}\mu^{j}_{\bar{i}}-g_{\bar{i}j}\mu^{j}_{\bar{s}}
\nonumber&&\\
B_{si}=g_{s\bar{j}}\bar{\mu}^{\bar{j}}_i-g_{i\bar{j}}\bar{\mu}^{\bar{j}}_s,
\quad
B_{\bar{s}\bar{i}}=g_{\bar{i}j}\mu^{j}_{\bar{s}}-g_{\bar{s}j}\mu^{j}_{\bar{i}},\quad && \Phi=\Phi_0(X)+\log\sqrt{g}.\nonumber
\end{align}

The symmetries of sigma model are the diffeomorphism symmetry and the shift of the $B$-field by the exact form. To reproduce these symmetries on the level of the parametrization (\ref{twistor}) of the metric and $B$-field, and we introduce
\begin{align}
\alpha=(v, \lambda, \bar{v}, \bar{\lambda})\in \Gamma(\mathcal{E}\oplus \bar{\mathcal{E}}), &&\\
v\in \Gamma(T^{(1,0)}M),\quad {\bar v}\in \Gamma(T^{(0,1)}M),\quad  \lambda\in \Gamma({T^*}^{(1,0)}M),\quad {\bar \lambda}\in \Gamma({T^*}^{(0,1)}M).&& \nonumber
\end{align}

Let $D:\Gamma(\mathcal{E}\oplus\bar{\mathcal{E}})\to \Gamma(\mathcal{E}\otimes\bar{\mathcal{E}})$, such that

\begin{equation}
 D\alpha=\left( \begin{array}{cc}
0 & -{\bar \partial }v\\
-{\partial \bar v} & {\bar\partial} \lambda-\partial{\bar\lambda} \end{array} \right).
\end{equation}

Let us consider elements $\mathbb{M}$ and $\alpha$ on some neighborhood $U\subset M$ as elements of the completed tensor product of holomorphic and antiholomorphic sections:
\begin{equation}
\alpha=\sum_Jf^J\hat{\otimes} {\bar{b}}^J+\sum_Kb^K\hat{\otimes} {\bar{f}}^K,\quad  
{\mathbb{M}}=\sum_I a^I\hat{\otimes} \bar{a}^I, 
\end{equation}
where $a^I, b^J\in {\mathcal{O}}(\mathcal{E})(U)$, $f^I\in 
\mathcal{O}_U$ and ${\bar{a}}^I, {\bar{b}}^J\in {\bar{\mathcal{O}}}(\bar{\mathcal{E}})(U)$, 
${\bar{f}}^I\in 
\bar{\mathcal{O}}_U$.\\

Let us locally define the following operations:
\begin{align}
\phi_2(\alpha,{\mathbb{M}})=
\sum_{I,J}[b^J, a^I]_D\otimes 
{\bar{f}}^{J}{\bar{a}}^I+\sum_{I,K}f^Ka^I\otimes[{\bar{b}}^K, {\bar{a}}^I]_D,&&\nonumber\\
\phi_3(\alpha, {\mathbb{M}},{\mathbb{M}})=\frac{1}{2}\sum_{I,J,K}\langle b^I, a^K\rangle a^J\otimes \langle \bar{a}^J, {\rm d}\bar{f}^I\rangle~ {\bar{a}}^K+\frac{1}{2}\sum_{I,J,K}
 \langle a^J, {\rm d}f^I\rangle {a}^K\otimes \langle {\bar{b}}^I, {\bar{a}}^K\rangle~ {\bar{a}}^J.&&\nonumber
\end{align}

Then the following theorem is true \cite{azginf}.

\begin{Thm}\hfill

\begin{itemize}
\item The operations $\phi_2$, $\phi_3$ can be defined globally on $M$.\\
\item
The following transformation of ${\mathbb{M}}$, with infinitesimal $\alpha=(v, \lambda, \bar{v}, \bar{\lambda})\in \Gamma(\mathcal{E}\oplus \bar{\mathcal{E}})$:
\begin{equation}\label{Mtransf}
{\mathbb{M}}\rightarrow {\mathbb{M}}+D\alpha+ \phi_2(\alpha,{\mathbb{M}})+\phi_3(\alpha, {\mathbb{M}},{\mathbb{M}}).
\end{equation}
is equivalent to
\begin{align}
G\rightarrow G-\mathcal{L}_{\bf v}G, \quad B\rightarrow B-\mathcal{L}_{\bf v}B,&&\nonumber\\
B\rightarrow B-2{\rm d}{\boldsymbol {\lambda}},&&
\end{align}
where ${\mathbf{v}}=(v,{\bar v})\in \Gamma(T^{(1,0)}M\oplus T^{(0,1)}M)$, ${\boldsymbol \lambda}=(\lambda, \bar \lambda)\in \Omega^{(1,0)}(M)\oplus\Omega^{(0,1)}(M)$, and $\mathcal{L}_{\bf v}$ is a Lie derivative with respect to ${\bf v}$.
\end{itemize}
\end{Thm}

The algebraic structure of the symmetry transformation and the cnsiderations from previous subsection suggest that there is an 
$L_{\infty}$-algebra structure on the complex 
\begin{equation}
({\bf F}^{\bullet}_{-},\mathcal{Q})=(\mathcal{F}^{\bullet}, Q){\otimes}(\bar{\mathcal{F}}^{\bullet}, \bar{Q})\vert_{{\bf b}^-=0} 
\end{equation}
which leads to Einstein equations as Maurer-Cartan equations and their symmetries. The corresponding space of Maurer-Cartan elements  has the following form: 
\begin{eqnarray}
\mathbf{F}_-^2\cong \Gamma(\mathcal{E}\otimes\bar{\mathcal{E}})\oplus \Gamma(TM\oplus T^*M)\oplus \mathscr{F}(M)\oplus\mathscr{F}(M).
\end{eqnarray}
Thus any $\Psi\in \mathbf{F}_-^2$ can presented as 
\begin{equation}\label{mcel}
\Psi=(\mathbb{M}, {\bf v}, {\boldsymbol \lambda}, \phi', \bar{\phi}'), 
\end{equation}
so that the first summand corresponds to Beltrami-Courant differential, the latter two provide the dilaton shift and the second term corresponding to vector fields and 1-forms is auxiliary. 
More explicitly, we expect the folllowing conjecture to be true.

\begin{Con}
There exists an $L_{\infty}$ structure on the complex $({\bf F}^{\bullet}_{{\bf b}^-},\mathcal{Q})$, so that the bilinear operation is given by symmetrized version of 
$\{~\cdot~,~\cdot~\}_- $ reproducing Einstein equations and their symmetries as the generalized Maurer-Cartan equations and their symmetries.
\end{Con}

The observation of the papers \cite{azginf}, \cite{azginfrev} is that the formula \ref{Mtransf} corresponding to the transformation of the Beltrami-Courant differential is obtained from the following formula:
\begin{equation}
\Psi\to \Psi+\mathcal{Q}\Psi-\{\Theta, \Psi\}_-+\frac{1}{2}\{\Theta, \Psi, \Psi\}_-,
\end{equation}
where $\Psi\in \mathbf{F}_-^2$ is a Maurer-Cartan element as in $(\ref{mcel})$, $\Theta \in  \mathbf{F}_-^1$ 
and operation $\{\cdot~,~\cdot, ~\cdot\}_-$ is the homotopy for the Jacobi identity of the nonsymmetrized operation $\{~\cdot~,~ \cdot~\}_-$. That suggests some field redefinitions for $\alpha$ as suggested in Proposition 4.2. of \cite{azginf} and at the same time some specific nontrivial relations within the resulting algebra which takes into account the Leibniz property of the homotopy BV-LZ algebra.

\subsection{C-bracket and Yang-Mills $C_{\infty}$-algebra}

The analogue of vertex algebra of open string is given by the following operator product expansion \cite{pol}, which adds logarithmic terms in the operator product expansion:
\begin{equation}\label{ope}
X^{i}(z_1)X^{j}(z_2)\sim -\frac{h}{2}\eta^{ij}\log|z_1-z_2|^2-\frac{h}{2}\eta^{ij}\log|z_1-\bar z_2|^2, 
\end{equation}
where 
\begin{equation}
X^i(z)=X^i-\eta^{ij}p_j\ln|z|^2-\frac{1}{i\sqrt 2}\sum_{n,n\neq 0}\frac{a^i_n}{n}(z^{-n}+\bar{z}^{-n}),
\end{equation}
where
\begin{equation}
[a^i_n, a^j_m]=hn\delta_{n,-m}, \quad [p_i, X^{j}]=\delta^{i}_{j},
\end{equation}
and the space underlying open string is the 
\begin{equation}
F_X={\rm Fock}\otimes \mathscr{F}(M),
\end{equation}
where the Fock space is: 
\begin{equation}
{\rm Fock}=\{a^{i_1}_{-n_1}....a^{i_k}_{-n_k}|0\rangle, n_1,..., n_k>0; i_1, \dots, i_k=1, \dots, D~|~a^j_{n}|0\rangle=0, n\ge 0,\}.
\end{equation}
and $\mathscr{F}(M)$ is some function space on $M=\mathbb{R}^D$.
The correspondence between those modes and the modes of $\beta-\gamma$ system from Section 5.3., where we discussed flat metric deformation is as follows:
\begin{equation}
a^i_n=\sqrt{2}^{-1}(X^{i}_n+\eta^{ij}p_{j,n}),~ n \neq 0;\quad X^i=X_0^i, \quad p_i=p_{i,0}.  
\end{equation}

The BRST operator given by \cite{pol}:
\begin{align}\label{brst}
Q=J_0, \quad J(z)=c(z)T(z)+b(z) c(z)\partial c(z) +\frac{3}{2}\partial^2 c(z),&&\\ 
\quad T(z)=-h^{-1}\eta_{ij}\partial X^{i}(z) {\partial} X^{j}(z).\nonumber&&
\end{align}
It acts as usual on the space $F_X\otimes \Lambda$, where $\Lambda$ is a $b$-$c$ VOA from Section 4. 

The states corresponding to the following vertex operators:
\begin{align}\label{f0}
:u(X):, \quad 2c:A_{j}(X)\partial X^{j}:-h\partial c\eta^{ij}\partial_{i}A_{j}(X),\quad 2hc\partial c :\tilde{A}_{j}(X)\partial X^{j}:, \quad h^2 c\partial c \partial^2 c :\tilde{u}(X):,&&\nonumber\\
h\partial c:v(x):, \quad hc\partial^2 c:\tilde{v}(X):+2hc\partial c :\tilde{v}(X):\partial X. &&
\end{align} 
produce the familiar direct sum of complexes $(\mathcal{F}^{\bullet}_{YM}, Q^{\eta}):=
(\mathcal{G}_1^{\bullet}\oplus \mathcal{G}_3^{\bullet},{Q'}^{\eta})$ from Proposition \ref{ymcomp}:
\begin{align}
\xymatrixcolsep{30pt}
\xymatrixrowsep{-5pt}
\xymatrix{
0\ar[r]&\Omega^0(M) \ar[r]^{\rm{d}} & \Omega^1(M)[-1] \ar[r]^{*{\rm d}*{\rm d}}  &\Omega^1(M)[-2] \ar[r]^{*{\rm d}*} & \Omega^0(M)[-3]\ar[r]&0\\
 &\quad &     & \quad    &&\\
&& \bigoplus & \bigoplus     &&\\
&0\ar[r]& \Omega^0(M)[-1] \ar[r]^{id}   & \Omega^0(M)[-2]\ar[r]&0
}
\end{align}
This is part of the complex we had for the flat metric deformation of the Courant agebroid, which excludes scalar fields.

An interesting feature here is that Lian-Zuckerman products can be defined in a similar fashion to the standard case:

For any two vertex operators corresponding to the open string, the 
operator product expansion is given by,
\begin{eqnarray}\label{opestr}
V(t)W= \sum_{k,l}t^{-l}(V,W)_l^{(k)}(h\log|t|)^k.
\end{eqnarray}
assuming that the parameter $t$ here is real.
One can define the product using the classical limit of the familiar formula:
\begin{equation}\label{muopen}
\mu_h(V,W)=(V,W)^{(0)}_1
\end{equation}
Namely,
\begin{eqnarray}
\mu_h: \mathcal{F}_{YM}^i\otimes \mathcal{F}_{YM}^j\rightarrow\mathcal{F}_{YM}^{i+j}[h], 
\end{eqnarray}
so that defining $\mu:=\lim_{h\to 0}\mu_h$ we obtain the following proposition first obtained in \cite{cftym}.
\begin{Thm}
The symmetrized version of the operation $\mu$ from (\ref{muopen}) produces the $C_{\infty}$-algebra on the complex $(\mathcal{F}^{\bullet}_{YM}, Q^{\eta})$, which is the reduction of the $C_{\infty}$-algebra from Proposition \ref{ymcomp}, which leads to Yang-Mills equations as in the part (3) of the Theorem \ref{ymthm}. 
\end{Thm}
Surprisingly the quasiclassical limit of the vertex operator version of the Lian-Zuckerman trilinear operation also works in this case, when one ignores logarithms as it was shown in  \cite{cftym}.

\begin{Rem}
We note here that this is different from the construction of open String Field Theory \cite{witopen} and the way Yang-Mills equations emerge from there by  integrating out the "massive" modes \cite{taylor}, \cite{berkovits}. One may expect that there is an $A_{\infty}$ morphism between the two homotopy algebras. 
\end{Rem}

An important feature of this algebra is that the bilinear operation  is described via the so-called $C$-bracket. Indeed, considering 1-form elements of $\Omega^1(M)$ as $\phi_{\bf A}=2cA_i\partial X^i$, $\phi_{\bf B}=2cB_i\partial X^i$:
\begin{align}
&&\mu(\phi_{\bf A}, \phi_{\bf B})={h}c\partial^2 c  A^{\mu}B_{\mu}+2{h}c\partial c \eta_{ij}[A,B]^{\eta,i}\partial X^j, \nonumber\\
&&[{\bf A},{\bf B}]^{\eta,j}=(A^{i}\partial_{i}B^{j}-\partial_{i}A^{j}B^{i}+\eta^{rj}\eta_{kl}\partial_{r}A^{k}B^{l}),\label{etab}
\end{align}
where $A^i=\eta^{ij}A_j$, $B^i=\eta^{ij}B_j$. Here the bilinear operation on vector fields 
\begin{equation}
[~,~]^{\eta} : {\rm Vect}(\mathbb{R}^N)\otimes {\rm Vect}(\mathbb{R}^N)\rightarrow {\rm Vect}(\mathbb{R}^N)
\end{equation}
given by the formula (\ref{etab}) is known as the $C$-bracket. It replaces the standard Courant bracket in what is called $BV^{\square}_{\infty}$ algebra of YM theory which we already mentioned in Section 5.
This bracket does not satisfy Jacobi identity unless the following condition imposed on any of the two participating vector fields:
\begin{equation}
\eta^{ij}\partial_{i}\partial_jA^k=0, \quad \eta^{ij}\partial_{i}\partial_jB^l=0, \quad 
\eta^{ij}\partial_{i}A^k\partial_jB^l=0, \quad k,l=1, \dots, N.
\end{equation}
In particular, if we attempt to construct a Lian-Zuckerman bracket out of $\mu$ using the operator $b_0$ in a similar  we did it for vertex operator algebra it fails to satisfy Jacobi identity unless the condition above is satisfied. In the next section we discuss emergence of this structure in a slightly different context. 

\subsection{Gravity and Double Field Theory from Yang-Mills}

Here we discuss another double copy: $BV^{\square}_{\infty}\otimes \overline{BV}^{\square}_{\infty}$-algebra.
Namely, now we consider the reduced completed tensor product of two isomorphic copies of the same complex,   
$(\mathcal{F}_{YM}^{\bullet}, {Q}^{\eta})$, $(\mathcal{F}_{YM}^{\bullet}, \bar{{Q}}^{\eta})$:
\begin{equation}
({\bf F}_{DFT},\mathcal{Q}^{\eta})=(\mathcal{F}_{YM}^{\bullet},Q^{\eta})\hat{\otimes} (\mathcal{F}_{YM}^{\bullet},\bar{Q}^{\eta})|_{{\bf b}_-=0}
\end{equation}
where $\mathcal{Q}^{\eta}={Q}^{\eta}+\bar{Q}^{\eta}$ and ${\bf b}_{\pm}=\frac{1}{2}({\bf b}\pm{\bar{\bf b})}$. 
One may inquire what structure we obtain on the above complex. It is known that the tensor product of $C_{\infty}$-algebras is a $C_{\infty}$-algebra and thus we have a structure of $C_{\infty}$-structure on $M\times {\bar M}$, where $M, \bar{M}\cong\mathbb{R}^D$, where the bilinear operation is given by the symmetrized version 
\begin{equation}
\hat{\mu}:=\mu^{\eta, sym}\otimes \bar{\mu}^{\eta, sym}.
\end{equation}
However, one wants to better understand the bracket structure. Let us first of all note that: 
\begin{eqnarray}
[Q^{\eta},{\bf b}]=-\Delta, \quad [{\bar Q}^{\eta},\bar{{\bf b}}]=-\bar{\Delta},
\end{eqnarray}
and thus:
\begin{eqnarray}
[\mathcal{Q}^{\eta},{\bf b}_{\pm}]=-\Delta_{\pm}, \quad\Delta_{+}=\frac{1}{2}{\Delta}+\frac{1}{2}{\bar \Delta}, \quad \Delta_{-}=\frac{1}{2}{{\Delta}}-\frac{1}{2}{\bar{\Delta}},
\end{eqnarray}
where $\Delta=\eta^{ij}\partial_{X^i}\partial_{X^j}$, 
$\bar{\Delta}=\eta^{ij}\partial_{\bar{X}^i}\partial_{\bar{X}^j}$
so that  $X^i$, ${\bar X}^i$ are the coordinates for $M$, $\bar M$.
Let us now introduce another set of coordinates:
\begin{equation}
x^i=X^i+\bar{X}^i, \quad \tilde{x}_i=\eta_{ij}(X^j-\bar{X}^j),\quad  i=1, \dots, D.
\end{equation}
This way 
$$
\Delta_{-}=2\sum_i\partial_{i}\tilde{\partial}^{{i}},
$$
where $\partial_{x^i}=\partial_i$, $\partial_{\tilde{x}_i}=\tilde{\partial}^i$.\\

Some of the results of \cite{hohmdouble2} can be summarized in the following Theorem.
\begin{Thm}\label{Jacbv}
The bilinear operation $\{~\cdot ~,~\cdot ~\}_-^{\eta}$ on $({\bf F}_{DFT},\mathcal{Q}^{\eta})$ defined as: 
\begin{equation}\label{bracketa}
(-1)^{|a_1|}\{a_1,a_2\}^{\eta}_-=
\mathbf{b_-}\hat{\mu}(a_1,a_2)-\hat{\mu}(\mathbf{b_-}a_1,a_2)-(-1)^{|a_1|}\hat{\mu}(a_1\mathbf{b_-}a_2)\nonumber
\end{equation}
satisfies homotopy Jacobi identity if the following is true: 
\begin{enumerate}
\item For any section $A$ on $M\times {\bar M}$, which represents an element of $({\bf F}_{DFT},\mathcal{Q}^{\eta})$ is annihilated by $\Delta_-$:
\begin{equation}
\Delta_- A=0.
\end{equation}
\item Any product of two local sections $A,B$ constituting elements of complex $({\bf F}_{DFT},\mathcal{Q}^{\eta})$ is annihilated by $\Delta_-$, or in local coordinates:
\begin{equation}
\sum_i\partial_i A\tilde{\partial}^iB+\tilde{\partial}^i A{\partial}_iB=0
\end{equation}
\end{enumerate}
\end{Thm}
\begin{Rem}In fact, the result of \cite{hohmdouble2} is stronger. It was 
shown that the product and the bracket satisfy first few relations of the 
$BV_{\infty}^{\square}$ algebra, but we will not go into details here.
\end{Rem}
The reason to introduce the new variables $x^i, \tilde{x}_i$ is a s follows. It is expected that there are higher brackets, satisfying 
$L_{\infty}$-relations on $({\bf F}_{DFT},\mathcal{Q}^{\eta})$ subject to the conditions (1), (2) of the above theorem. 

Based on the holomorphic double copy one could expect that the corresponding hypothetical Maurer-Cartan equations will be related to the Einstein equations (\ref{Einsteq}). Below we will explain the evidence.

If one imposes the condition 
\begin{equation}
X^i=\bar{X}^i, \quad 1, \dots, D,
\end{equation}
and the complex 
\begin{equation}
({\bf F}^{\bullet}_{CS},\mathcal{Q}^{\eta}):=({\bf F}^{\bullet}_{DFT},\mathcal{Q}^{\eta})_{\{\tilde{x}_i=0\}},
\end{equation}
which is the complex on the diagonal 
\begin{eqnarray}
\delta: M\to M\times {\bar M}
\end{eqnarray}
coincides with the BRST complex of the closed string with the metric $\eta^{ij}$. 

In this case, the corresponding Maurer-Cartan equation up to second order was studied earlier, in \cite{zeit3}. There one associates to the local sections in the complex $({\bf F}_{CS},\mathcal{Q}^{\eta})$ the combinations normal ordered products of open string vertex operators $X^i(z,{\bar z})$ where the corresponding OPE is:
\begin{eqnarray}
X^i(z_1,{\bar z}_1)X^j(z_2,{\bar z}_2)\sim -\frac{h}{2}\eta^{ij}\log|z_1-z_2|^2
\end{eqnarray}
Here 
\begin{equation}
X^i(z,\bar{z})=X^i-\frac{1}{2}\eta^{ij}p_j\ln|z|^2-\frac{1}{i\sqrt 2}\sum_{n,n\neq 0}\Big(\frac{a^i_n}{n}z^{-n}+\frac{\bar{a}^i_n}{n}\bar{z}^{-n}\Big),
\end{equation}
where
\begin{equation}
[a^i_n, a^j_m]=hn\eta^{ij}\delta_{n,-m}, \quad, [\bar{a}^i_n, \bar{a}^j_m]=hn\eta^{ij}\delta_{n,-m},\quad  [p_i, X^{j}]=h\delta^{i}_{j}
\end{equation}
and the space underlying open string is the 
\begin{equation}
F^{c}_X={\rm Fock}\otimes \overline{{\rm Fock}}\otimes\mathscr{F}(M),
\end{equation}
where the Fock space is: 
\begin{align}
&&{\rm Fock}=\{a^{i_1}_{-n_1}....a^{i_k}_{-n_k}|0\rangle, n_1,..., n_k>0; i_1, \dots, i_k=1, \dots, D~|~a^j_{n}|0\rangle=0, n\ge 0,\},\\
&&\overline{{\rm Fock}}=\{\bar{a}^{i_1}_{-n_1}....\bar{a}^{i_k}_{-n_k}|0\rangle, n_1,..., n_k>0; i_1, \dots, i_k=1, \dots, D~|~\bar{a}^j_{n}|0\rangle=0, n\ge 0,\}
\end{align}
and $\mathscr{F}(M)$ is some function space on $M=\mathbb{R}^D$.
Accompanying it with two $b$-$c$ systems $b(z)c(w)\sim \frac{1}{z-w}$, $\bar{b}(z)\bar{c}(w)\sim \frac{1}{z-w}$ generating vertex algebras on the spaces $\Lambda$, $\bar{\Lambda}$ give the BRST operator
\begin{align}\label{brstcl}
\mathcal{Q}^{\eta}=J_0+\bar{J}_0, &&\nonumber\\
J(z)=c(z)T(z)+b(z) c(z)\partial c(z) +\frac{3}{2}\partial^2 c(z), \quad
\bar{J}(\bar{z})=\bar{c}(\bar{z})\bar{T}(\bar{z})+\bar{b}(\bar{z}) \bar{c}(\bar{z})\partial \bar{c}(\bar{z}) +\frac{3}{2}\bar{\partial}^2 \bar{c}(\bar{z})&&\nonumber\\ 
T(z)=-h^{-1}\eta_{ij}\partial X^{i}(z) {\partial} X^{j}(z),\quad  {\bar T}(z)=-h^{-1}\eta_{ij}\bar{\partial} X^{i}(\bar{z}) \bar{\partial} X^{j}(\bar{z})&&
\end{align}
acting on the space $F^c_X\otimes \Lambda\otimes {\bar{\Lambda}}$. It is well-known that it nilpotent \cite{pol} when dimension $D=26$. One can realize the complex $({\bf F}^{\bullet}_{CS},\mathcal{Q}^{\eta})$ as a subcomplex in $F^c_X\otimes \Lambda\otimes {\bar{\Lambda}}$.

The bracket (\ref{bracketa}) is identified with the classical limit of the operation
\begin{eqnarray}\label{LZbrdouble}
(-1)^{|a_1|}\{a_1,a_2\}^{\eta}_{-,h}=P_0\int_{C_{\epsilon,0}} a_1^{(1)}(z) a_2-(-1)^{|a_1||a_2|}P_0\int_{C_{\epsilon,0}} a_2^{(1)}(z) a_1,
\end{eqnarray}
where the 1-form $a^{(1)}(z)$ for a given vertex operator $a(z)$ in the BRST complex is given by $$a^{(1)}(z)=dz[b_{-1},a(z,{\bar z})]+d{\bar z}[{\bar b}_{-1},a(z,{\bar z})],$$ and the projection $P_0$ ignores all $\epsilon$-dependent terms in the expansion of the integral over the circular contour $C_{\epsilon,0}$ around 0 of radius $\epsilon$.

In a familiar fashion, the operation we have the following expression for 
its quasiclassical limit of $\{~\cdot~,~\cdot~ \}^{\eta}_{-,h}$ :
\begin{equation}
\{~\cdot~,~\cdot~\}^{\eta}_{-,h}:{\bf F}^i_{CS}\otimes {\bf F}^j_{CS}\rightarrow h{\bf F}^{i+j-1}_{CS}[h],\quad \{~\cdot~,~\cdot~\}^{\eta}_{-}=\lim_{h\rightarrow 0}h^{-1}\{~\cdot~,~\cdot~\}^{\eta}_{-,h}.
\end{equation}

The role of this operation for Einstein equation
was indicated in \cite{zeit3}, where the following partial result was obtained.

\begin{Prop}\label{MCsecond}
The Maurer-Cartan equation second order approximation and its symmetries:
\begin{equation}\label{MCapprox}
\mathcal{Q}^{\eta}\Psi+\{\Psi, \Psi\}^{\eta}_-+\dots=0,\quad  \Psi\rightarrow \mathcal{Q}^{\eta}\Theta+\{\Psi,\Theta\}_-^{\eta}+\dots,
\end{equation}
where $\Psi\in {\bf F}^{2}_{CS}$, $\Theta \in {\bf F}^{1}_{CS}$ 
is equivalent to the Einstein equations (\ref{Einsteq}) with the metric being expanded around the flat metric $\{\eta_{ij}\}$:
\begin{equation}
G_{\ij}=\eta_{ij}+t^2{\rm h}^1_{ij}+t^2{\rm h}^2_{ij}+...,\quad  B_{ij}=0, \quad \Phi=t\Phi_1+t^2\Phi_2+\dots,
\end{equation}
where $t$ is the expansion parameter, and dots stand for the expansion of cubic or higher terms, so that 
\begin{equation}
\Psi=t\Psi_1({\rm h}^1, \Phi_1)+t^2\Psi_2({\rm h}^1, {\rm h}^2, \Phi_1, \Phi_2)+...\in {\bf F}^{2}_{CS}.
\end{equation}
\end{Prop}
\begin{Rem}
Here the dependence on ${\rm h}^1$ accompanying ${\rm h}^2$ in $\Psi^2$ is known as {\it field redefinition}, and, as indicated in \cite{zeit3}, is related to the perturbation theory of sigma model.  
\end{Rem} 
The problem here is that on the level of symmetries, if one applies naively the operation (\ref{LZbrdouble}) to the elements of ${\bf F}^{1}_{CS}$, one obtains the $C$-bracket from (\ref{etab}), which does not satisfy Jacobi identity. This has direct relation to the constraints (1), (2) from the Theorem \ref{Jacbv}.

Let us return back to the full complex $({\bf F}^{\bullet}_{DFT},\mathcal{Q}^{\eta})$. Hull and Zwiebach introduced the so-called Double Field Theory, which is a field theory on the product of two D-dimensional tori, which are $T$-$dual$ to each other with natural coordinates $\{x^i\}$ and $\{\tilde{x}_i\}$. It gives a certain generalization of Einstein equations, which have symmetries governed by the C-bracket (see  Subsection 7.2) with the $\eta$ metric is replaced by the natural pairing between the T-dual coordinates  on the tori corresponding to the metric $dx^id{\tilde x}_i$.

What has been shown in \cite{hohmdouble2} recently is the generalization of the Proposition \ref{MCsecond}, which relates DFT equations and generalized Maurer-Cartan equations.

\begin{Prop}
The Maurer-Cartan equation (\ref{MCapprox}) and its symmetry transformation, 
where $\Psi\in {\bf F}^2_{DFT}$, so that conditions (1), (2) of the Theorem \ref{Jacbv} are satisfied for all local sections constituting $\Psi$, $\Lambda$,    
is equivalent to the DFT equations and their symmetries expanded around the flat metric up expanded to the second order.
\end{Prop}

This statement is true up to field redefinition as in (\ref{MCsecond}). The big advantage for considering the bigger complex $({\bf F}^{\bullet}_{DFT},\mathcal{Q}^{\eta})$ is that upon constraints (1), (2) of the Theorem (\ref{Jacbv}) we have a version of a homotopy Lie algebra. The version of DFT, where fields satisfy these constraints is known as {\it strongly constrained} DFT. Using the transfer procedure, one can obtain the homotopy Lie algebra structure on the complex  
$({\bf F}^{\bullet}_{DFT},\mathcal{Q}^{\eta})|_{\Delta_-=0}$, leading to {\it weakly constrained} DFT. However, the trilinear operations there involve inversions of the Laplacian operator and become nonlocal, which means the proper trilinear and possible higher operations on the complex $({\bf F}^{\bullet}_{CS},\mathcal{Q}^{\eta})$ are nonlocal as well. Thus, although at the expense of keeping constraints (1), (2), one may expect to have local expressions for higher operations on the Double Field Theory complex $({\bf F}^{\bullet}_{DFT},\mathcal{Q}^{\eta})$.

\end{document}